\documentclass[11pt]{amsart}
\usepackage{amsmath}
\usepackage[psamsfonts]{amssymb}
\usepackage{graphicx}
\usepackage[english]{babel}
\usepackage[norelsize]{algorithm2e}
\usepackage[all]{xy}
\usepackage{hyperref}
\usepackage{verbatim}

\theoremstyle{definition}
\newtheorem{theorem}{Theorem}[section]
\newtheorem{prop}[theorem]{Proposition}
\newtheorem{lemma}[theorem]{Lemma}
\newtheorem{cor}[theorem]{Corollary}

\newtheorem{notation}[theorem]{Notation}
\newtheorem*{theorema}{Theorem A}
\newtheorem*{theoremb}{Theorem B}
\newtheorem{ex}[theorem]{Example}
\theoremstyle{remark}
\newtheorem{dfn}[theorem]{Definition}
\newtheorem{remark}[theorem]{Remark}
\newtheorem{claim}[theorem]{Claim}

\def\co{\colon\thinspace}

\def\ep{\epsilon}

\def\R{\mathbb{R}}
\def\Z{\mathbb{Z}}
\def\N{\mathbb{N}}
\def\C{\mathbb{C}}

\DeclareMathOperator{\Img}{Im}

\def\co{\colon\thinspace}

\begin{document}

\title{Persistent homology and Floer-Novikov theory}

\author{Michael Usher}
\email{usher@math.uga.edu}
\address{Department of Mathematics\\University of Georgia\\Athens, GA 30602}

\author{Jun Zhang}
\email{jzhang@math.uga.edu}
\address{Department of Mathematics\\University of Georgia\\Athens, GA 30602}

\date{\today}

\begin{abstract}
We construct ``barcodes'' for the chain complexes over Novikov rings that arise in Novikov's Morse theory for closed one-forms and in Floer theory on not-necessarily-monotone symplectic manifolds. In the case of classical Morse theory these coincide with the barcodes familiar from persistent homology.  Our barcodes completely characterize the filtered chain homotopy type of the chain complex; in particular they subsume in a natural way previous filtered Floer-theoretic invariants such as boundary depth and torsion exponents, and also reflect information about spectral invariants.  We moreover prove a continuity result which is a natural analogue both of the classical bottleneck stability theorem in persistent homology and of standard continuity results for spectral invariants, and we use this to prove a $C^0$-robustness result for the fixed points of Hamiltonian diffeomorphisms.  Our approach, which is rather different from the standard methods of persistent homology, is based on a non-Archimedean singular value decomposition for the boundary operator of the chain complex.
\end{abstract}

\maketitle
  
\tableofcontents

\section{Introduction}

Persistent homology is a well-established tool in the rapidly-developing field of topological data analysis.  On an algebraic level, the subject studies ``persistence modules,'' i.e., structures $\mathbb{V}$ consisting of a module $V_t$ associated to each $t\in \R$ with homomorphisms $\sigma_{st}\co V_s\to V_t$ whenever $s\leq t$ satisfying the functoriality properties that $\sigma_{ss}=I_{V_s}$, the identity map on module $V_s$, and $\sigma_{su}=\sigma_{tu}\circ\sigma_{st}$ (more generally $\R$ could be replaced by an arbitrary partially ordered set, but this generalization will not be relevant to this paper).  Persistence modules arise naturally in topology when one considers a continuous function $f\co X\to \R$ on a topological space $X$; for a field $\mathcal K$ one can then let $V_t=H_*(\{f\leq t\}; \mathcal K)$ be the homology of the $t$-sublevel set, with the $\sigma_{st}$ being the inclusion-induced maps.  For example if $X=\R^n$ and the function $f\co \R^n\to \R$ is given by the minimal distance to a finite collection of points sampled from some subset $S\subset \R^n$, then $V_t$ is the homology of the union of balls of radius $t$ around the points of the sample; the structure of the associated persistence module has been used effectively to make inferences about the topological structure of the set $S$ in some real-world situations, see e.g. \cite{CBull}.

Under finiteness hypotheses on the modules $V_t$ (for instance finite-type as in \cite{ZC} or more generally pointwise-finite-dimensionality as in \cite{CB}), provided that the coefficient ring for the modules $V_t$ is a field $\mathcal K$, it can be shown that the persistence module $\mathbb{V}$ is isomorphic in the obvious sense to a direct sum of ``interval modules'' $\mathcal{K}_I$, where $I\subset \R$ is an interval and by definition $(\mathcal{K}_I)_t = \mathcal K$ for $t \in I$ and $\{0\}$ otherwise and the morphisms $\sigma_{st}$ are the identity on $\mathcal K$ when $s,t\in I$ and $0$ otherwise.  The {\it barcode} of $\mathbb{V}$ is then defined to be the multiset of intervals appearing in this direct sum decomposition.  When $\mathbb{V}$ is obtained as the filtered homology of a finite-dimensional chain complex, \cite{ZC} gives a worst-case-cubic-time algorithm that computes the barcode given the boundary operator on the chain complex.

If $f\co X\to \R$ is a Morse function on a compact smooth manifold, a standard construction (see e.g. \cite{Sc93}) yields a ``Morse chain complex'' $(CM_{*}(f),\partial)$. The degree-$k$ part $CM_{k}(f)$ of the complex is formally spanned (say over the field $\mathcal K$) by the critical points of $f$ having index $k$. The boundary operator $\partial \co CM_{k+1}(f)\to CM_{k}(f)$ counts (with appropriate signs) negative gradient flowlines of $f$ which are asymptotic as $t\to-\infty$ to an index-$(k+1)$ critical point and as $t\to\infty$ to an index $k$ critical point.  For any $t\in \R$, if we consider the subspace  $CM^{t}_{*}(f)\leq CM_{*}(f)$ spanned only by those critical points $p$ of $f$ with $f(p)\leq t$, then the fact that $f$ decreases along its negative gradient flowlines readily implies that $CM^{t}_{*}(f)$ is a subcomplex of $CM_{*}(f)$. So taking homology gives filtered Morse homology groups $HM_{*}^{t}(f)$, with inclusion induced maps $HM_{*}^{s}(f)\to HM_{*}^{t}(f)$ when $s\leq t$ that satisfy the usual functoriality properties.  Thus the filtered Morse homology groups associated to a Morse function yield a persistence module; given a formula for the Morse boundary operator one could then apply the algorithm from \cite{ZC} to compute its barcode.  In fact, standard results of Morse theory show that this persistence module is (up to isomorphism) simply the persistence module comprising the sublevel homologies $H_{*}(\{f\leq t\};\mathcal{K})$ with the inclusion-induced maps.

There are a variety of situations in which one can do some form of Morse theory for a suitable function $\mathcal{A}\co \mathcal{C}\to \R$ on an appropriate infinite-dimensional manifold $\mathcal{C}$.  Indeed Morse himself \cite{Morse} applied his theory to the energy functional on the loop space of a Riemannian manifold in order to study its geodesics.  Floer discovered some rather different manifestations of infinite-dimensional Morse theory \cite{F88}, \cite{F88b}, \cite{F89} involving functions $\mathcal{A}$ which, unlike the energy functional, are unbounded above and below and have critical points of infinite index.  In these cases, one still obtains a Floer chain complex analogous to the Morse complex of the previous paragraph and can still speak of the filtered homologies $HF^t$ with their inclusion-induced maps $HF^s\to HF^t$; however it is no longer true that these filtered homology groups relate directly to classical topological invariants---rather they are new objects.  Thus Floer's construction gives (taking filtrations into account as above) a persistence module.  If the persistence module satisfies appropriate finiteness conditions one then obtains a barcode by the procedure indicated earlier; however as we will explain below the finiteness conditions only hold in rather restricted circumstances.  While the filtered Floer groups have been studied since the early 1990's and have been a significant tool in symplectic topology since that time (see e.g. \cite{FH}, \cite{Sc00}, \cite{EP03}, \cite{Oh05}, \cite{U13}, \cite{HLS}), it is only very recently that they have been considered from a persistent-homological point of view.  Namely, the authors of \cite{PS} apply ideas from persistent homology to prove interesting results about autonomous Hamiltonian diffeomorphisms of  symplectic manifolds, subject to a topological restriction that is necessary to guarantee the finiteness property that leads to a barcode.  This paper will generalize the notion of a barcode to more general Floer-theoretic situations. In particular this opens up the possibility of extending the results from \cite{PS} to manifolds other than those considered therein; this is the subject of work in progress by the second author.

The difficulty with applying the theory of barcodes to general Floer complexes lies in the fact that, typically, Floer theory is more properly viewed as an infinite dimensional version of Novikov's Morse theory for closed one-forms (\cite{Nov}, \cite{Fa}) rather than of classical Morse theory.  Here one considers a closed $1$-form $\alpha$ on some manifold $M$ which vanishes transversely with finitely many zeros, and takes a regular covering $\pi\co \tilde{M}\to M$ on which we have $\pi^*\alpha = d\tilde{f}$ for some function $\tilde{f}\co\tilde{M}\to \R$.    Then $\tilde{f}$ will be a Morse function whose critical locus consists of the preimage of the (finite) zero locus of $\alpha$ under $\pi$; in particular if the de Rham cohomology class of $\alpha$ is nontrivial then $\pi\co\tilde{M}\to M$ will necessarily have infinite fibers and so $\tilde{f}$ will have infinitely many critical points.  

One then attempts to construct a Morse-type complex $CN_{*}(\tilde{f})$ by setting $CN_k(\tilde{f})$ equal to the span over $\mathcal K$ of the index-$k$ critical points\footnote{``index'' means Morse index in the finite-dimensional case (see, e.g. \cite{Sc93}), and typically some version of the Maslov index in the Floer-theoretic case (see, e.g. \cite{RS}).} of $\tilde{f}$, with boundary operator $\partial\co CN_{k+1}(\tilde{f})\to CN_k(\tilde{f})$ given by setting, for an index-$(k+1)$ critical point $p$ of $\tilde{f}$, \begin{equation}\label{boundary} \partial p=\sum_{ind_{\tilde{f}}(q)=k}n(p,q)q \nonumber \end{equation} where $n(p,q)$ is a count of negative gradient flowlines for $\tilde{f}$ (with respect to suitably generic Riemannian metric pulled back to $\tilde{M}$ from $M$) asymptotic to $p$ in negative time and to $q$ in positive time.  However the above attempt does not quite work because the sum on the right-hand side may have infinitely many nonzero terms; thus it is necessary to enlarge $CN_k(\tilde{f})$ to accommodate certain formal infinite sums.  The correct definition is, denoting by $Crit_k(\tilde{f})$ the set of critical points of $\tilde{f}$ with index $k$: \begin{equation}\label{novcomp} CN_k(\tilde{f})=\left\{\sum_{p\in Crit_k(\tilde{f})}a_pp\,\left|\, a_p\in \mathcal{K},\,(\forall C\in \R)(\#\{p|a_p\neq 0,\,\tilde{f}(p)>C\}<\infty)\right.\right\}. \end{equation} Then under suitable hypotheses it can be shown that the definition of $\partial$ above gives a well-defined map $\partial\co CN_{k+1}(\tilde{f})\to CN_k(\tilde{f})$ such that $\partial^2=0$. This construction can be carried out in many contexts, including the classical Novikov complex where $M$ is compact and various Floer theories where $M$ is infinite-dimensional. In the latter case, the zeros of $\alpha$ are typically some objects of interest, such as closed orbits of a Hamiltonian flow, on some other finite-dimensional manifold.  In these cases, just as in Morse theory, $\partial$ preserves the $\R$-filtration given by, for $t\in \R$, letting  $CN^{t}_{k}(\tilde{f})$ consist of only those formal sums ${\textstyle \sum_p a_pp}$ where each $\tilde{f}(p)\leq t$.  In this way we obtain filtered Novikov homology groups $HN^{t}_{*}(\tilde{f})$ with inclusion-induced maps $HN^{s}(\tilde{f})\to HN^{t}(\tilde{f})$ satisfying the axioms of a persistence module over $\mathcal{K}$.  

 However when the cover $\tilde{M}\to M$ is nontrivial, this persistence module over $\mathcal K$ does not satisfy the hypotheses of many of the major theorems of persistent homology---the maps $HN^{s}(\tilde{f})\to HN^{t}(\tilde{f})$ generally have infinite rank over $\mathcal K$  (due to a certain ``lifting" scenario which is described later in this paragraph) and so the persistence module is not ``q-tame'' in the sense of \cite{CSGO}.  As is well-known, to get a finite-dimensional object out of the Novikov complex one should work not over $\mathcal K$ but over a suitable Novikov ring.  From now on we will assume that the cover $\pi\co \tilde{M}\to M$ is minimal subject to the property that $\pi^*\alpha$ is exact---in other words the covering group coincides with the kernel of the homomorphism $I_{\alpha}\co\pi_1(M)\to \R$ induced by integrating $\alpha$ over loops; this will lead to our Novikov ring being a field.  Given this assumption, let $\hat{\Gamma}\leq \R$ be the image of $I_{\alpha}$.  Then by, for any $g\in \hat{\Gamma}$, lifting loops in $M$ with integral equal to $-g$ to paths in $\tilde{M}$, we obtain an action of $\hat{\Gamma}$ on the critical locus of $\tilde{f}$ such that $\tilde{f}(p)-\tilde{f}(gp)=g$.  In some Floer-theoretic situations this action can shift the index by $s(g)$ for some homomorphism $s\co \hat{\Gamma}\to \Z$. For instance, in Hamiltonian Floer theory $s$ is given by evaluating twice the first Chern class of the symplectic manifold on spheres, whereas in the classical case  of the Novikov chain complex of a closed one-form on a finite-dimensional manifold, $s$ is zero. Now let $\Gamma=\ker s$, so that $\Gamma$ acts on the index-$k$ critical points of $\tilde{f}$, and this action then gives rise to an action of the following {\it Novikov field} on $CN_{k}(\tilde{f})$: \[
\Lambda^{\mathcal K,\Gamma}=\left\{\sum_{g\in \Gamma}a_gT^g\, \bigg|\,a_g\in\mathcal{K},(\forall C\in \R)(\#\{g|a_g\neq 0,\,g<C\}<\infty)\right\}. \] It follows from the description that $CN_{k}(\tilde{f})$ is a vector space over $\Lambda^{\mathcal K,\Gamma}$ of (finite!) dimension equal to the number of zeros of our original $\alpha\in\Omega^1(M)$ which admit a lift to $\tilde{M}$ which is an index-$k$ critical point for $\tilde{f}$---indeed if the set $\{\tilde{p}_1,\ldots,\tilde{p}_{m_i}\}\subset \tilde{M}$ consists of exactly one such lift of each of these zeros of $\alpha$ then $\{\tilde{p}_1,\ldots,\tilde{p}_{m_i}\}$ is a $\Lambda^{\mathcal K,\Gamma}$-basis for $CN_{k}(\tilde{f})$.

Now since the action by an element $g$ of $\Gamma$ shifts the value of $\tilde{f}$ by $-g$, the filtered groups $CN^{t}_{k}(\tilde{f})$ are not preserved by multiplication by scalars in $\Lambda^{\mathcal K,\Gamma}$, and so the aforementioned persistence module $\{HF^{t}(\tilde{f})\}$ over $\mathcal K$ can {\it not} be viewed as a persistence module over $\Lambda^{\mathcal K,\Gamma}$, unless of course $\Gamma=\{0\}$, in which case $\Lambda^{\mathcal K,\Gamma}=\mathcal K$.   Our strategy in this paper is to understand filtered Novikov and Floer complexes not through their induced persistence modules on homology (cf. Remark \ref{homology} below) but rather through the {\it non-Archimedean geometry} that the filtration induces on the chain complexes.    This will lead to an alternative theory of barcodes which recovers the standard theory in the case that $\Gamma=\{0\}$ (cf. \cite{ZC}, \cite{CSGO} and, for a different perspective, \cite{Bar})  but which also makes sense for arbitrary $\Gamma$, while continuing to enjoy various desirable properties.  

We should mention that, in the case of Morse-Novikov theory for a function $f\co X\to S^1$, a different approach to persistent homology is taken in \cite{BD}, \cite{BH}.  These works are based around the notion of the (zigzag) persistent homology of level sets of the function; this is a rather different viewpoint from ours, as in order to obtain insight into Floer theory we only use the algebraic features of the Floer chain complex---in a typical Floer theory there is nothing that plays the role of the  homology of a level set.  Rather we construct what could be called an algebraic simulation of the more classical sublevel set persistence, even though (as noted in \cite{BD}) from a geometric point of view it does not make sense to speak of the sublevel sets of an $S^1$-valued function.  Also our theory, unlike that of \cite{BD}, \cite{BH}, applies to the Novikov complexes of closed one-forms that have dense period groups.  Notwithstanding these differences there are some indications (see in particular the remark after \cite[Theorem 1.4]{BH}) that the constructions may be related on their common domains of applicability; it would be interesting to understand this further.

\subsection{Outline of the paper and summary of main results}

With the exception of an application to Hamiltonian Floer theory in Section \ref{hamsect}, the entirety of this paper is written in a general algebraic context involving chain complexes of certain kinds of non-Archimedean normed vector spaces over Novikov fields $\Lambda=\Lambda^{\mathcal{K},\Gamma}$. (In particular, no knowledge of Floer theory is required to read the large majority of the paper, though it may be helpful as motivation.) The definitions necessary for our theory are somewhat involved and so will not be included in detail in this introduction, but they make use of the standard notion of orthogonality in non-Archimedean normed vector spaces, a subject which is reviewed Section \ref{orthsect}.  Our first key result is Theorem \ref{existsvd}, which shows that any linear map $A\co C\to D$  between two finite-dimensional non-Archimedean normed vector spaces $C$ and $D$ over $\Lambda$ having orthogonal bases admits a \emph{singular value decomposition}: there are orthogonal bases $B_C$ for $C$ and $B_D$ for $D$ such that $A$ maps each member of $B_C$ either to zero or to one of the elements of $B_D$.  In the case that $C$ and $D$ admit ortho\emph{normal} bases and not just orthogonal ones this was known (see \cite[Section 4.3]{Ked}); however Floer complexes typically admit orthogonal but not orthonormal bases (unless one extends coefficients, which leads to a loss of information), and in this case Theorem \ref{existsvd} appears to be new.  

In Definition \ref{dfnfcc} we introduce the notion of a ``Floer-type complex'' $(C_*,\partial,\ell)$ over a Novikov field $\Lambda$; this is a chain complex of $\Lambda$-vector spaces $(C_*,\partial)$ with a non-Archimedean norm $e^\ell$ on each graded piece $C_k$ that induces a filtration which is respected by $\partial$.  We later construct our versions of the barcode by consideration  of singular value decompositions of the various graded pieces of the boundary operator.  Singular value decompositions are rather non-unique, but we prove a variety of results reflecting that data about filtrations of the elements involved in a singular value decomposition is often independent of choices and so gives rise to invariants of the Floer-type complex $(C_*,\partial,\ell)$. The first instance of this appears in Theorem \ref{genebd}, which relates the boundary depth of \cite{U11},\cite{U13}, as well as generalizations thereof, to singular value decompositions.  Theorem \ref{torsion} shows that these generalized boundary depths are equal to (an algebraic abstraction of) the torsion exponents from \cite{FOOO09}.  Since the definition of the torsion exponents in \cite{FOOO09} requires first extending coefficients to the universal Novikov field (with $\Gamma=\R$), whereas our definition in terms of singular value decompositions does not require such an extension, this implies new restrictions on the values that the torsion exponents can take: in particular they all must be equal to differences between filtration levels of chains in the original Floer complex.  

\subsubsection{Barcodes}
Our fundamental invariants of a Floer-type complex, the ``verbose barcode'' and the ``concise barcode,'' are defined in Definition \ref{dfnbc}.  The verbose barcode in any given degree is a \emph{finite} multiset of elements $([a],L)$ of the Cartesian product $(\R/\Gamma)\times [0,\infty]$, where $\Gamma\leq \R$ is the subgroup described above and involved in the definition of the Novikov field $\Lambda=\Lambda^{\mathcal{K},\Gamma}$.  The concise barcode is simply the sub-multiset of the verbose barcode consisting of elements $([a],L)$ with $L>0$.   Both barcodes are constructed in an explicit way from singular value decompositions of the graded pieces of the boundary operator on a Floer-type complex.

To be a bit more specific, as is made explicit in Proposition \ref{ournormalform}, a  singular value decomposition can be thought of as expressing the Floer-type complex as an \emph{orthogonal} direct sum of very simple complexes\footnote{The ``Morse-Barannikov complex'' described in \cite{Bar}, \cite[Section 2]{LNV} can be seen as a special case of this direct sum decomposition when $\Gamma=\{0\}$ and the Floer-type complex is the Morse complex of a Morse function whose critical values are all distinct; see Remark \ref{brmk} for details.} having the form \begin{equation}\label{simple} \cdots\to 0\to span_{\Lambda}\{y\}\to span_{\Lambda}\{\partial y\}\to 0\to\cdots\quad \mbox{or}\quad \cdots\to 0\to span_{\Lambda}\{x\}\to 0\to\cdots \end{equation} and the verbose barcode consists of the elements $([\ell(\partial y)],\ell(y)-\ell(\partial y))$ for summands of the first type and $([\ell(x)],\infty)$ for summands of the second type.  The concise barcode discards those elements coming from summands with $\ell(\partial y)=\ell(y)$ (as these do not affect any of the filtered homology groups).

To put these barcodes into context, suppose that $\Gamma=\{0\}$ and that our Floer-type complex $(C_*,\partial,\ell)$ is given by the Morse complex $CM_*(f)$ of a Morse function $f$ on a compact manifold $X$ (with $\ell$ recording the highest critical value attained by a given chain in the Morse chain complex).  Then standard persistent homology methods associate to $f$ a barcode, which is a collection of intervals $[a,b)$ with $a<b\leq\infty$, given the interpretation that each interval $[a,b)$ in the collection corresponds to a topological feature of $X$ which is ``born'' at the level $\{f=a\}$ and ``dies'' at the level $\{f=b\}$ (or never dies if $b=\infty$).  Theorem \ref{classical} proves that, when $\Gamma=\{0\}$ (so that $\R/\Gamma=\R$), our concise barcode is equivalent to the classical persistent homology barcode under the correspondence that sends a pair $(a,L)$ in the concise barcode to an interval $[a,a+L)$.  (Thus the second coordinates $L$ in our elements of the concise barcode correspond to the lengths of bars in the persistent homology barcode.)  To relate this back to the persistence module $\{HM^{t}_{*}(f)\}_{t\in \R}\cong \left\{H_*(\{f\leq t\};\mathcal{K})\right\}_{t\in \R}$ discussed earlier in the introduction, each $HM^{t}_{k}(f)$ has dimension equal to the number of elements $(a,L)$ in the degree-$k$ concise barcode such that $a\leq t<a+L$, and the rank of the inclusion-induced map $HM^{s}_{k}(f)\to HM^{t}_{k}(f)$ is equal to the number of such elements with $a\leq s\leq t<a+L$.

When $\Gamma$ is a nontrivial subgroup of $\R$, a Floer-type complex over $\Lambda$ is more akin to the Morse-Novikov complex of a multivalued function $f$, where the ambiguity of the values of $f$ is given by the group $\Gamma$ (for instance, identifying $S^1=\R/\Z$, for an $S^1$-valued function we would have $\Gamma=\Z$).  While this situation lies outside the scope of classical persistent homology barcodes for reasons indicated earlier in the introduction, on a naive level it should be clear that if a topological feature of $X$ is born where $f=a$ and dies
where $f=b$ (corresponding to a bar $[a,b)$ in a hypothetical barcode), then it should equally be true that, for any $g\in \Gamma$, a topological feature of $X$ is born where $f=a+g$ and dies where $f=b+g$.  So bars would come in $\Gamma$-parametrized families with $\Gamma$ acting on both endpoints of the interval; such families in turn can be specified by the coset $[a]$ of the left endpoint $a$ in $\R/\Gamma$ together with the length $L=b-a\in[0,\infty]$.  This motivates our definition of the verbose and concise barcodes as multisets of elements of $(\R/\Gamma)\times [0,\infty]$. In terms of the summands in (\ref{simple}), the need to quotient by $\Gamma$ simply comes from the fact that the elements $y$ and $x$ are only specified up to the scalar multiplication action of $\Lambda\setminus\{0\}$, which can affect their filtration levels by an arbitrary element of $\Gamma$.   The following classification results are two of the main theorems of this paper.

\begin{theorema} Two Floer-type complexes $(C_*, \partial_C, \ell_C)$ and $(D_*, \partial_D, \ell_D)$ are filtered chain isomorphic to each other if and only if they have identical verbose barcodes in all degrees.\end{theorema}

\begin{theoremb} Two Floer-type complexes $(C_*, \partial_C, \ell_C)$ and $(D_*, \partial_D, \ell_D)$ are filtered chain homotopy equivalent to each other if and only if they have identical concise barcodes in all degrees.
\end{theoremb}

Theorem A includes the statement that the verbose (and hence also the concise) barcode is independent of the singular value decomposition used to define it; indeed this statement is probably the hardest part of Theorems A and B to prove.  We prove these theorems in Section \ref{classsect}.  

As should already be clear from the above discussion, the only distinction between the verbose and concise barcodes of a Floer-type complex $(C_*,\partial,\ell)$ arises from elements $y\in C_*$ with $\ell(\partial y)=\ell(y)$.  While our definition of a Floer-type complex only imposes the inequality $\ell(\partial y)\leq \ell(y)$, in many of the most important examples, including the Morse complex of a Morse function or the Hamiltonian Floer complex of a nondegenerate Hamiltonian, one in fact always has a strict inequality $\ell(\partial y)<\ell(y)$ for all $y\in C_*\setminus\{0\}$.  For complexes satisfying this latter property the verbose and concise barcodes are equal, and so Theorems A and B show that the filtered chain isomorphism classification of such complexes is exactly the same as their filtered chain homotopy equivalence classification.  (This fact can also be proven in a more direct way, see for instance the argument at the end of \cite[Proof of Lemma 3.8]{U11}.)  

In Remark \ref{nonstrict} below we mention some examples of naturally-occurring Floer-type complexes in which an equality $\ell(\partial y)=\ell(y)$ can sometimes hold.  In these complexes the verbose and concise barcodes are generally different, and thus the filtered chain homotopy equivalence classification is coarser than the filtered chain isomorphism classification.  For many purposes the filtered chain isomorphism classification is likely too fine, in that it may depend on auxiliary choices made in the construction of the complex (for instance, in the Morse-Bott complex as constructed in \cite{Fr}, it would depend on the choices of Morse functions on the critical submanifolds of the Morse-Bott function under consideration).  The filtered chain homotopy type (and thus, by Theorem B, the concise barcode) is generally insensitive to such choices, and moreover is robust in a sense made precise in Theorem \ref{mainstab} below.

When $\Gamma=\{0\}$, Theorem B may be seen as an analogue of standard results from persistent homology theory (like \cite[Corollary 3.1]{ZC}) which imply that the degree-$k$ barcode of a Floer-type complex completely classifies the persistence module obtained from its filtered homologies $H^{t}_{k}(C_*)$. Of course, the filtered chain homotopy type of a filtered chain complex is sufficient to determine its filtered homologies.  Conversely, still assuming that $\Gamma=\{0\}$,
by using the description of finite-type persistence modules as $\mathcal{K}[t]$-modules in \cite{ZC}, and taking advantage of the fact that (because $\mathcal{K}[t]$ is a PID) chain complexes of free $\mathcal{K}[t]$-modules are classified up to chain homotopy equivalence by their homology, one can show that the filtered chain homotopy type of a Floer-type complex is determined by its filtered homology persistence module.   Thus  although the persistent homology literature generally focuses on homological invariants rather than classification of the underlying chain complexes up to filtered isomorphism or filtered homotopy equivalence, when $\Gamma=\{0\}$ Theorem B can be deduced from \cite{ZC} together with a little homological algebra and Theorem \ref{classical}.  

For any choice of the group $\Gamma$, the concise barcode contains information about various numerical invariants of Floer-type complexes that have previously been used in filtered Floer theory.  In particular, by Theorems \ref{genebd} and \ref{torsion} and the definition of the concise barcode, the torsion exponents from \cite{FOOO09} are precisely the second coordinates $L$ of elements $([a],L)$ of the concise barcode having $L<\infty$, written in decreasing order; the boundary depth of \cite{U11} is just the largest of these.  Meanwhile in Section \ref{specrel} we show that the concise barcode also carries information about the spectral invariants as in \cite{Sc00}, \cite{Oh05}.  In particular a number $a$ arises as the spectral invariant of some class in the homology of the complex if and only if there is an element of form $([a],\infty)$ in the concise barcode.  By contrast, the numbers $a$ appearing in elements $([a],L)$ of the concise barcode with $L<\infty$ do not seem to have standard analogues in Floer theory, and so could be considered as new invariants. Whereas the spectral invariants and boundary depth have the notable feature of varying in Lipschitz fashion with respect to the Hofer norm on the space of Hamiltonians, these numbers $a$ have somewhat more limited robustness properties, which can be understood in terms of our stability results such as Corollary \ref{hamintro} below. 

In Section \ref{dualcodes} we show how the verbose (and hence also the concise) barcodes of a Floer-type complex in various degrees are related to those of its dual complex, and to those of the complex obtained by extending the coefficient field by enlarging the group $\Gamma$.  The relationships are rather simple; in the case of the dual complex they can be seen as extending results from \cite{U10} on the Floer theory side and from \cite{SMV} on the persistent homology side.

\begin{remark}\label{homology}
Our approach differs from the conventional approach in the persistent homology literature in that we work almost entirely at the chain level; for the most part our theorems do not directly discuss the homology persistence modules $\{H_{k}^{t}(C_*)\}_{t\in\R}$.  The primary reason for this is that, when $\Gamma\neq \{0\}$, such homology persistence modules are unlikely to fit into any reasonable classification scheme.  The basic premise of the original introduction of barcodes in \cite{ZC} is that a finite-type persistence module over a field $\mathcal{K}$ can be understood in terms of the classification of finitely-generated $\mathcal{K}[x]$-modules; however, when $\Gamma\neq \{0\}$ our persistence modules are infinitely-generated over $\mathcal{K}$, leading to infinitely-generated $\mathcal{K}[x]$-modules and suggesting that one should work with a larger coefficient ring than $\mathcal{K}$.  Since the action of the Novikov field does not preserve the filtration on the chain complex, the $H_{k}^{t}(C_*)$ are not modules over the full Novikov field $\Lambda$.  They are however modules over the subring $\Lambda_{\geq 0}$ consisting of elements $\sum_g a_gT^g$ with all $g\geq 0$, and if $\Gamma$ is nontrivial and discrete (in which case $\Lambda_{\geq 0}$ is isomorphic to a formal power series ring $\mathcal{K}[[t]]$) then each $H_{k}^{t}(C_*)$ is a finitely generated $\Lambda_{\geq 0}$-module.  But then the approach from \cite{ZC} leads to the consideration of finitely generated $\mathcal{K}[[t]][x]$-modules, which again do not admit a simple description in terms of barcode-type data since $\mathcal{K}[[t]][x]$ is not a PID.

Our chain-level approach exploits the fact that the chain groups $C_k$ in a Floer-type complex, unlike the filtered homologies, are finitely generated vector spaces over a field (namely $\Lambda$), which makes it more feasible to obtain a straightforward classification.  It does follow from our results that the filtered homology persistence module of a Floer-type complex can be expressed as a finite direct sum of filtered homology persistence modules of the building blocks $\mathcal{E}(a,L,k)$ depicted in (\ref{simple}).  However, since the filtered homology persistence modules of the $\mathcal{E}(a,L,k)$ are themselves somewhat complicated (as the interested reader may verify by direct computation) it is not clear whether this is a useful observation. For instance we do not know whether the image on homology of a filtered chain map between two Floer-type complexes can always likewise be written as a direct sum of these basic persistence modules; if this is true then it might be possible to adapt arguments from \cite{BL} or \cite[Section 3.4]{CSGO} to remove the factor of $2$ in Theorem \ref{mainstab} below.
\end{remark}

\subsubsection{Stability}
Among the most important theorems in persistent homology theory is the bottleneck stability theorem, which in its original form \cite{CEH} shows that, for the sublevel persistence modules $\{H_*(\{f\leq t\};\mathcal K\}_{t\in \R}$ associated to suitably tame functions $f\co X\to \R$ on a fixed topological space $X$, the barcode of the persistence module depends in $1$-Lipschitz fashion on $f$, where we use the $C^0$-norm to measure the distance between functions and the bottleneck distance (recalled below) to measure distances between barcodes.  Since in applications there is inevitably some imprecision in the function $f$, some sort of result along these lines is evidently important in order to ensure that the barcode detects robust information.  More recently a number of extensions and new proofs of the bottleneck stability theorem have appeared, for instance in \cite{CCGGO}, \cite{CSGO}, \cite{BL}; these have recast the theorem as an essentially algebraic result about persistence modules satisfying a finiteness condition such as $q$-tameness or pointwise finite-dimensionality (see \cite[p. 4]{BL} for precise definitions).  When recast in this fashion  the stability theorem can be improved to an isometry theorem, stating that two natural metrics on an appropriate class of persistence modules are equal.

Hamiltonian Floer theory (\cite{F89},\cite{HS},\cite{LT},\cite{FO},\cite{Pa}) associates a Floer-type complex to any suitably non-degenerate Hamiltonian $H\co S^1\times M\to \R$ on a compact symplectic manifold $(M,\omega)$.  A well-established and useful principle in Hamiltonian Floer theory is that many aspects of the filtered Floer complex are robust under $C^0$-small perturbations of the Hamiltonian; for instance various $\R$-valued quantities that can be extracted from the Floer complex such as spectral invariants and boundary depth are Lipschitz with respect to the $C^0$-norm on Hamiltonian functions (\cite{Sc00},\cite{Oh05},\cite{U11}).  Naively this is rather surprising since $C^0$-perturbing a Hamiltonian can dramatically alter its Hamiltonian flow.  Our notion of the concise barcode---which by Theorem B gives a complete invariant of the filtered chain homotopy type of a Floer-type complex---allows us to obtain a more complete understanding of this $C^0$-rigidity property, as an instance of a general algebraic result which extends the bottleneck stability/isometry theorem to Floer-type complexes for general subgroups $\Gamma\leq \R$.

In order to formulate our version of the stability theorem we must explain  the notions of distance that we use between Floer-type complexes on the one hand and concise barcodes on the other.  Beginning with the latter, consider two multisets $\mathcal{S}$ and $\mathcal{T}$ of elements of $(\R/\Gamma)\times [0,\infty]$.  For $\delta\geq 0$, a $\delta$-\textbf{matching} between $\mathcal{S}$ and $\mathcal{T}$ consists of the following data:
\begin{itemize} \item[(i)] submultisets $\mathcal{S}_{short}$ and $\mathcal{T}_{short}$ such that the second coordinate $L$ of every element $([a],L)\in \mathcal{S}_{short}\cup\mathcal{T}_{short}$ obeys $L\leq 2\delta$.
\item[(ii)] A bijection $\sigma\co \mathcal{S}\setminus\mathcal{S}_{short}\to \mathcal{T}\setminus\mathcal{T}_{short}$ such that, for each $([a],L)\in \mathcal{S}\setminus\mathcal{S}_{short}$ (where $a\in \R$, $L\in [0,\infty]$) we have $\sigma([a],L)=([a'],L')$ where for all $\ep>0$ the representative $a'$ of the coset $[a']\in \R/\Gamma$ can be chosen such that both $|a'-a|\leq \delta+\ep$ and either $L=L'=\infty$ or $|(a'+L')-(a+L)|\leq \delta+\ep$.
\end{itemize}

Thus, viewing elements $([a],L)$ as corresponding to intervals $[a,a+L)$ (modulo $\Gamma$-translation), a $\delta$-matching is a matching which shifts both endpoints of each interval by at most $\delta$, with the proviso that we allow an interval $I$ to be matched with a fictitious zero-length interval at the center of $I$.  

\begin{dfn} If $\mathcal{S}$ and $\mathcal{T}$ are two multisets of elements of $(\R/\Gamma)\times [0,\infty]$ then the \emph{bottleneck distance} between $\mathcal{S}$ and $\mathcal{T}$ is \[ d_{B}(\mathcal{S},\mathcal{T})=\inf\{\delta\geq 0\,|\,\mbox{There exists a $\delta$-matching between }\mathcal{S}\mbox{ and }\mathcal{T}\}. \]
If $\mathcal{S}=\{\mathcal{S}_k\}_{k\in\Z}$ and  
$\mathcal{T}=\{\mathcal{T}_k\}_{k\in\Z}$ are two $\Z$-parametrized families of multisets of elements of $(\R/\Gamma)\times[0,\infty]$ then we write \[ d_B(\mathcal{S},\mathcal{T})=\sup_{k\in\Z}d_B(\mathcal{S}_k,\mathcal{T}_k). \]
\end{dfn}

It is easy to see that in the special case that $\Gamma=\{0\}$ the above definition agrees with the notion of bottleneck distance in \cite{CEH}.  
Note that the value $d_{B}$ can easily be infinity. For instance this occurs if $\mathcal S = \{([a], \infty)\}$ and $\mathcal T = \{([a], L)\}$ where $L < \infty$. 

On the Floer complex side, we make the following definition which is a slight modification of \cite[Definition 3.7]{U13}.  As is explained in Appendix \ref{intsect} this is very closely related to the notion of \emph{interleaving} of persistence modules from  \cite{CCGGO}.

\begin{dfn} Let $(C_*,\partial_C,\ell_C)$ and $(D_*,\partial_D,\ell_D)$ be two Floer-type complexes, and $\delta\geq 0$.  A $\delta$-\textbf{quasiequivalence} between $C_*$ and $D_*$ is a quadruple $(\Phi,\Psi,K_1,K_2)$ where:
\begin{itemize} \item $\Phi\co C_*\to D_*$ and $\Psi\co D_*\to C_*$ are chain maps, with $\ell_D(\Phi c)\leq \ell_C(c)+\delta$ and $\ell_C(\Psi d)\leq \ell_D(d)+\delta$ for all $c\in C_*$ and $d\in D_*$.
\item $K_C\co C_*\to C_{*+1}$ and $K_D\co D_*\to D_{*+1}$ obey the homotopy equations $\Psi\circ\Phi-I_{C_*}=\partial_CK_C+K_C\partial_C$ and $\Phi\circ\Psi-I_{D_*}=\partial_DK_D+K_D\partial_D$, and for all $c\in C_*$ and $d\in D_*$ we have $\ell_C(K_Cc)\leq \ell_C(c)+2\delta$ and $\ell_D(K_Dd)\leq \ell_D(d)+2\delta$.
\end{itemize}
The \textbf{quasiequivalence distance} between   $(C_*,\partial_C,\ell_C)$ and $(D_*,\partial_D,\ell_D)$ is then defined to be \[ d_Q((C_*,\partial_C,\ell_C),(D_*,\partial_D,\ell_D))=\inf\left\{\delta\geq 0\left|\begin{array}{cc}\mbox{There exists a $\delta$-quasiequivalence between }\\(C_*,\partial_C,\ell_C)\mbox{ and }(D_*,\partial_D,\ell_D)\end{array}\right.\right\}. \]
\end{dfn}
 
We will prove the following as Theorems \ref{stabthm} and \ref{convstab} in Sections \ref{stabproof} and \ref{convsect}:

\begin{theorem} \label{mainstab} Given a Floer-type complex $(C_*,\partial_C,\ell_C)$, denote its concise barcode by $\mathcal{B}(C_*,\partial_C,\ell_C)$ and the degree-$k$ part of its concise barcode by $\mathcal{B}_{C,k}$.  Then the bottleneck and quasiequivalence distances obey, for any Floer-type complexes $(C_*,\partial_C,\ell_C)$ and $(D_*,\partial_D,\ell_D)$:
\begin{itemize}\item[(i)] $ \begin{aligned}[t]  d_Q((C_*,\partial_C,\ell_C),(D_*,\partial_D,\ell_D)&\leq d_B(\mathcal{B}(C_*,\partial_C,\ell_C),\mathcal{B}(D_*,\partial_D,\ell_D))\\ & \qquad \qquad \leq 2d_Q((C_*,\partial_C,\ell_C),(D_*,\partial_D,\ell_D)).\end{aligned}$
\item[(ii)]  For $k\in \Z$ let $\Delta_{D,k}>0$ denote the smallest second coordinate $L$ of all of the elements of $\mathcal{B}_{D,k}$.  If $d_Q((C_*,\partial_C,\ell_C),(D_*,\partial_D,\ell_D))<\frac{\Delta_{D,k}}{4}$, then \[  
 d_B(\mathcal{B}_{C,k},\mathcal{B}_{D,k})\leq d_Q((C_*,\partial_C,\ell_C),(D_*,\partial_D,\ell_D)). \] \end{itemize}
\end{theorem}

Thus the map from filtered chain homotopy equivalence classes of Floer-type complexes to concise barcodes is  at least bi-Lipschitz, with Lipschitz constant $2$.  We expect that it is always an isometry; in fact when $\Gamma=\{0\}$ this can be inferred from  \cite[Theorem 4.11]{CSGO} and Theorem \ref{classical}, and as mentioned in Remark \ref{densecase} it is also true in the opposite extreme case when $\Gamma$ is dense. 

Our proof that the bottleneck distance $d_B$  obeys the upper bounds of Theorem \ref{mainstab} is roughly divided into two parts.  First, in Proposition \ref{twofilt}, we prove the sharp inequality $d_B\leq d_Q$ in the special case that the Floer-type complexes $(C_*,\partial_C,\ell_C)$ and $(D_*,\partial_D,\ell_D)$ have the same underlying chain complex, and differ only in  their filtration functions $\ell_C$ and $\ell_D$.  In the rest of Section \ref{stabproof} we approximately reduce the general case to this special case, using a mapping cylinder construction to obtain two different filtration functions on a single chain complex, one of which has concise barcode equal to that of $(D_*,\partial_D,\ell_D)$ (see Proposition \ref{decompcyl}), and the other of which has concise barcode consisting of the concise barcode of $(C_*,\ell_C,\partial_C)$ together with some ``extra''  elements $([a],L)\in(\R/\Gamma)\times[0,\infty]$ all having $L\leq 2d_Q((C_*,\partial_C,\ell_C),(D_*,\partial_D,\ell_D))$ (see Proposition \ref{changefil}).  These constructions are quickly seen in Section \ref{endproof} to yield the upper bounds on $d_B$ in the two parts of Theorem \ref{mainstab}; the factor of $2$ in part (i) arises from the ``extra'' bars in the concise barcode of the Floer-type complex from Proposition \ref{changefil}.

Meanwhile, the proof of the other inequality $d_Q\leq d_B$ in Theorem \ref{mainstab}(i) is considerably simpler, and is carried out by a direct construction in Section \ref{convsect}.

As mentioned earlier, it is likely that the factor of $2$ in Theorem \ref{mainstab}(i) is unnecessary, \emph{i.e.} that the map from  Floer-type complexes to concise barcodes is an isometry with respect to the quasiequivalence distance $d_Q$ on Floer-type complexes and the bottleneck distance $d_B$ on concise barcodes.  Although we do not prove this, by taking advantage of Theorem \ref{mainstab}(ii) we show in Section \ref{interp} that, if $d_Q$ is replaced by a somewhat more complicated distance $d_P$ that we call the interpolating distance, then the map is indeed an isometry (see Theorem \ref{globaliso}).  The expected isometry between $d_Q$ and $d_B$ is then equivalent to the statement that $d_P=d_Q$. Consistently with this, our experience in concrete situations has been that methods which lead to bounds on one of $d_P$ or $d_Q$ often also produce identical bounds on the other.

%Moreover, by taking advantage of the last sentence of Theorem \ref{mainstab}, in Section \ref{interp}, if we replace $d_Q$ by a somewhat more complicated distance function $d_P$ that we call the interpolating distance (see Theorem \ref{globaliso}), we show this map is an isometry. Of course the expectation indicated in the previous paragraph can then be rephrased as saying that the distances $d_P$ and $d_Q$ are in fact equal (and apparently this is true when $\Gamma$ is either trivial or dense).  In any case, our experience has been that methods which lead to bounds on one of the distance functions $d_P$ or $d_Q$ often also produce bounds on the other.

The final section of the body of the paper applies our general algebraic results to Hamiltonian Floer theory, the relevant features of which are reviewed at the beginning of that section.\footnote{While we focus on Hamiltonian Floer theory in Section \ref{hamsect}, very similar results would apply to the Hamiltonian-perturbed Lagrangian Floer chain complexes or to the chain complexes underlying Novikov homology.}     Combining Theorem \ref{globaliso} with standard results from Hamiltonian Floer theory proves the following, later restated as Corollary \ref{hamint}:

\begin{cor} \label{hamintro} If $H_0$ and $H_1$ are two non-degenerate Hamiltonians on any compact symplectic manifold $(M,\omega)$, then the bottleneck distance between their associated concise barcodes of $(CF_*(H_0),\partial_{H_0},\ell_{H_0})$ and $(CF_{*}(H_1),\partial_{H_1},\ell_{H_1})$ is less than or equal to ${\textstyle \int_{0}^{1}}\|H_1(t,\cdot)-H_0(t,\cdot)\|_{L^{\infty}}dt$.\end{cor}

To summarize, we have shown how to associate to the Hamiltonian Floer complex combinatorial data in the form of the concise barcode, which completely classifies the complex up to filtered chain homotopy equivalence, and which is continuous with respect to variations in the Hamiltonian in a way made precise in Corollary \ref{hamintro}.  Given the way in which torsion exponents, the boundary depth, and spectral invariants are encoded in the concise barcode, this continuity can be seen as a simultaneous extension of continuity results for those quantities (\cite[Theorem 6.1.25]{FOOO09}, \cite[Theorem 1.1(ii)]{U11}, \cite[(12)]{Sc00}).  

We then apply Corollary \ref{hamintro} to prove our main application, Theorem \ref{pert}, concerning the robustness of the fixed points of a nondegenerate Hamiltonian diffeomorphism under $C^0$-perturbations of the Hamiltonian: roughly speaking, as long as the perturbation is small enough (as determined by the concise barcode of the original Hamiltonian), the perturbed Hamiltonian, if it is still nondegenerate, will have at least as many fixed points as the original one, with actions that are close to the original actions.  Moreover, depending in a precise way on the concise barcode, fixed points with certain actions may be identified as enjoying stronger robustness properties (in the sense that  a larger perturbation is required to eliminate them) than general fixed points of the same map.  While  $C^0$-robustness of fixed points is a familiar idea in Hamiltonian Floer theory (see, \emph{e.g}, \cite[Theorem 2.1]{CR}), Theorem \ref{pert} goes farther than previous results both in its control over the actions of the perturbed fixed points and in the way that it gives stronger bounds for the robustness of unperturbed fixed points with certain actions (see Remark \ref{pertrem}). 

\begin{comment}In Appendix \ref{algo} we present an explicit algorithm for computing (almost all of) a singular value decomposition for a linear map between appropriate vector spaces over a Novikov field, which in particular suffices to find the barcode of a Floer-type complex in any given degree provided that one has a matrix representation of the boundary operator.  The algorithm is similar to that introduced in classical persistent homology in \cite{ZC} and in particular should have similar (cubic) complexity, measured in (Novikov) field operations.\end{comment}

Finally, Appendix \ref{intsect} identifies the quasiequivalence distance $d_Q$ that features in Theorem \ref{mainstab} with a chain level version of the interleaving distance that is commonly used (e.g. in \cite{CCGGO}) in the persistent homology literature.

\subsection*{Acknowledgements} The first author thanks the TDA Study Group in the UGA Statistics Department for introducing him to persistent homology, and K. Ono for pointing out the likely relationship between torsion thresholds and boundary depth.  Both authors are grateful to L. Polterovich for encouraging us to study Floer theory from a persistent-homological point of view, for comments on an initial version of the paper, and for various useful conversations.  Some of these conversations occurred during a visit of the second author to Tel Aviv University in Fall 2014; he is indebted to its hospitality and also to a guided reading course overseen by L. Polterovich. In particular, discussions with D. Rosen during this visit raised the question considered in Appendix \ref{intsect}. The authors also thank an anonymous referee for his/her careful reading and many  suggestions.  This work was partially supported by NSF grants DMS-1105700 and DMS-1509213.

\section{Non-Archimedean orthogonality} \label{orthsect}

\subsection{Non-Archimedean normed vector spaces}

 Fixing a ground field $\mathcal K$ and an additive subgroup $\Gamma \leq \mathbb R$ as in the introduction, we will consider vector spaces over the {\it Novikov field} defined as 
\[ \begin{array}{l}
\Lambda = \Lambda^{\mathcal K, \Gamma} = \left\{\left. \sum_{g \in \Gamma} a_g T^g \, \right \vert \,a_g\in \mathcal K,\,(\forall C \in \mathbb R) \left( \# \{g \,| \, a_g \neq 0, g < C \} < \infty \right)     \right\}
\end{array} \]
where $T$ is a formal symbol and we use the obvious ``power series'' addition and multiplication.  This Novikov field adapts the ring used by Novikov in his version of Morse theory for multivalued functions; see \cite{HS} both for some of its algebraic properties and for its use in Hamiltonian Floer homology.  Note that when $\Gamma$ is the trivial group, $\Lambda$ reduces to the ground field $\mathcal{K}$. 

First, we need the following classical definition.

\begin{dfn} \label{dfn-val} A {\bf valuation} $\nu$ on a field $\mathcal F$ is a function $\nu: \mathcal F \to \R \cup \{\infty\}$ such that 
\begin{itemize}
\item [(V1)]  $\nu(x) = \infty$ if and only if $x = 0$;
\item[(V2)] For any $x, y \in \mathcal F$, $\nu(xy) = \nu(x) + \nu(y)$;
\item[(V3)] For any $x, y \in \mathcal F$, $\nu(x+y) \geq \min\{ \nu(x), \nu(y)\}$ with equality when $\nu(x) = \nu(y)$. 
\end{itemize}
Moreover, we call a valuation $\nu$ {\bf trivial} if $\nu(x) = 0$ for $x \neq 0$ and $\nu(x) = \infty$ precisely when $x = 0$. 
\end{dfn}

For $\mathcal F = \Lambda$ defined as above, we can associate a valuation simply by 
\[ \begin{array}{l}
\nu\left( \sum_{g \in \Gamma} a_g T^g \right) = \min \{g \,| \,a_g \neq 0\}
\end{array} \] where we use the standard convention that the minimum of the empty set is $\infty$. 
It is easy to see that this $\nu$ satisfies conditions (V1), (V2) and (V3). Note that the finiteness condition in the definition of Novikov field ensures that the minimum exists. If $\Gamma=\{0\}$, then the valuation $\nu$ is trivial.

\begin{dfn} \label{dfn-fil} A \textbf{non-Archimedean normed vector space} over $\Lambda$ is a pair $(C,\ell)$ where $C$ is a vector space over $\Lambda$ endowed with a filtration function $\ell: C \rightarrow \mathbb R \, \cup \, \{-\infty\}$ satisfying the following axioms:

\begin{itemize}
\item[(F1)] $\ell(x) = -\infty$ if and only if $x=0$;
\item[(F2)] For any $\lambda \in \Lambda$ and $x \in C$, $\ell(\lambda x) = \ell(x) - \nu(\lambda)$;
\item[(F3)] For any $x,y \in C$, $\ell(x+y) \leq \max\{\ell(x), \ell(y)\}$.
\end{itemize}
\end{dfn}

In terms of Definition \ref{dfn-fil}, the standard convention would be that the norm on a non-Archimedean normed vector space $(C,\ell)$ is $e^{\ell}$, not $\ell$. The phrasing of the above definition reflects the fact that we will focus on the function $\ell$, not on the norm $e^{\ell}$.

We record the following standard fact:

\begin{prop} \label{strict} If $(C,\ell)$ is a non-Archimedean normed vector space over $\Lambda$ and the elements $x,y\in C$ satisfy $\ell(x)\neq \ell(y)$, then \begin{equation}\label{strictna} \ell(x+y)=\max\{\ell(x),\ell(y)\}.\end{equation}
\end{prop}

\begin{proof} Of course the inequality ``$\leq$'' in (\ref{strictna}) is just (F3).  For ``$\geq$'' we assume without loss of generality that $\ell(x)>\ell(y)$, so we are to show that $\ell(x+y)\geq \ell(x)$.  Now (F2) implies that $\ell(-y)=\ell(y)$, so $\ell(x)=\ell((x+y)+(-y))\leq\max\{\ell(x+y),\ell(y)\}$.  Thus since we have assumed that $\ell(x)>\ell(y)$ we indeed must have $\ell(x)\leq \ell(x+y)$.
\end{proof}

\begin{ex}\label{RipsComplex} (Rips complexes). 
Let $X$ be a collection of points in Euclidean space. We will define a one-parameter family of ``Rips complexes'' associated to $X$ as follows. Let $CR_*(X)$ be the simplicial chain complex over $\mathcal{K}$ of the complete simplicial complex on the set $X$, so that $CR_k(X)$ is the free $\mathcal{K}$-vector space generated by the $k$-simplices all of whose vertices lie in $X$.  Define $\ell\co CR_*(X)\to \R\cup\{-\infty\}$ by setting $\ell({\textstyle{\sum_i}} a_i\sigma_i)$ equal to the largest diameter of any of the simplicies $\sigma_i$ with $a_i\neq 0$ (and to $-\infty$ when ${\textstyle \sum_i} a_i\sigma_i=0$).  Then $(CR_*(X),\ell)$ is a non-Archimedean vector space over $\Lambda^{\mathcal{K},\{0\}}=\mathcal{K}$. For any $\ep>0$ we define the Rips complex with parameter $\ep$, $CR_*(X;\ep)$, to be the subcomplex of $C_*$ with  degree-$k$ part given by \[ CR_k(X;\ep)=\{c\in CR_k(X)\,|\,\ell(x)\leq \ep\}.\] Thus $CR_*(X;\ep)$ is spanned by those simplices with diameter at most $\ep$.  The standard simplicial boundary operator maps $CR_k(X;\ep)$ to $CR_{k-1}(X;\ep)$, yielding Rips homology groups $HR_k(X;\ep)$, and the dependence of these homology groups on $\ep$ is a standard object of study in applied persistent homology, as in \cite{ZC}.

\end{ex}

\begin{ex}\label{MorseComplex} (Morse complex). Suppose we have a closed manifold $X$ and $f$ is a Morse function on $X$.  We may then consider its Morse chain complex $CM_*(X; f)$ over the field $\mathcal K=\Lambda^{\mathcal{K},\{0\}}$ as in \cite{Sc93}. Let $C = {\textstyle \bigoplus_{k}} CM_k(X;f)$. For any element $x \in C$, by the definition of the Morse chain complex, $x = {\textstyle \sum_i a_i p_i}$ where each $p_i$ is a critical point and $a_i\in\mathcal{K}$. Then define $\ell: C \rightarrow \mathbb R\cup\{-\infty\}$ by
\[ \begin{array}{l}
\ell \left( \sum_i a_i p_i\right) = \max\{f(p_i) \, |\, a_i \neq 0\}
\end{array}, \]
with the usual convention that the maximum of the empty set is $-\infty$. It is easy to see that $\ell$ satisfies (F1), (F2) and (F3) above.  Therefore, $({\textstyle \bigoplus_{k}} CM_k(X;f), \ell)$ is a non-Archimedean normed vector space over $\mathcal{K}=\Lambda^{\mathcal{K},\{0\}}$.
\end{ex}

\begin{ex}\label{NovComplex} 
Given a closed one-form $\alpha$ on a closed manifold $M$, let $\pi\co \tilde{M}\to M$ denote the regular covering space of $M$ that is minimal subject to the property that $\pi^*\alpha$ is exact, and choose $\tilde{f}\co \tilde{M}\to \R$ such that $d\tilde{f} = \pi^* \alpha$. The graded parts $CN_k(\tilde{f})$ of the Novikov complex (see (\ref{novcomp})) can likewise be seen as non-Archimedean vector spaces over $\Lambda=\Lambda^{\mathcal{K},\Gamma}$ where the group $\Gamma\leq \R$ consists of all possible integrals of $\alpha$ around loops in $M$.  Namely, just as in the previous two examples we put \[ \ell\left(\sum a_pp\right)=\max\{\tilde{f}(p)\,|\,a_p\neq 0\}. \]  We leave verification of axioms (F1), (F2), and (F3) to the reader.
\end{ex}

\subsection{Orthogonality}

We use the standard notions of orthogonality in non-Archimedean normed vector spaces (cf. \cite{MS65}). 

\begin{dfn}\label{dfnor} Let $(C,\ell)$ be a non-Archimedean normed vector space over a Novikov field $\Lambda$. \begin{itemize}
\item 
 Two subspaces $V$ and $W$ of $C$ are said to be {\bf orthogonal} if for all $v \in V$ and $w \in W$, we have 
\[ 
\ell(v+w) = \max\{\ell(v), \ell(w)\}
. \]
\item A finite ordered collection $(w_1,\ldots,w_r)$ of elements of $C$ is said to be \textbf{orthogonal} if, for all $\lambda_1,\ldots,\lambda_r\in \Lambda$, we have \begin{equation}\label{eltorth} \ell\left(\sum_{i=1}^{r}\lambda_iw_i\right)=\max_{1\leq i\leq r}\ell(\lambda_iw_i). \end{equation}
\end{itemize}
\end{dfn}

In particular a pair $(v,w)$ of elements of $C$ is orthogonal if and only if the spans $\langle v\rangle_{\Lambda}$ and $\langle w\rangle_{\Lambda}$ are orthogonal as subspaces of $C$.  Of course, by (F2), the criterion (\ref{eltorth}) can equivalently be written as \begin{equation} \label{eltorth2} \ell\left(\sum_{i=1}^{r}\lambda_iw_i\right)=\max_{1\leq i\leq r}(\ell(w_i)-\nu(\lambda_i)). \end{equation}

\begin{ex}\label{exor} Here is a simple example illustrating the notion of orthogonality. Let $\Gamma=\{0\}$ so that $\Lambda=\mathcal{K}$ has the trivial valuation defined in Definition \ref{dfn-val}. Let $C$ be a two-dimensional $\mathcal{K}$-vector space, spanned by elements $x,y$.  We may define a filtration function $\ell$ on $C$ by declaring $(x,y)$ to be an orthogonal basis with, say, $\ell(x)=1$ and $\ell(y)=0$; then in accordance with (\ref{eltorth2}) and the definition of the trivial valuation $\nu$  we will have \[ \ell(\lambda x+\eta y)=\left\{\begin{array}{ll} 1 & \lambda\neq 0 \\ 0 & \lambda = 0,\,\eta\neq 0 \\ -\infty & \lambda=\eta=0 \end{array}\right.. \]

The ordered basis $(x+y,y)$ will likewise be orthogonal: indeed for $\lambda,\eta\in \mathcal{K}$ we have \[ \ell(\lambda(x+y)+\eta y)=\ell(\lambda x+(\lambda+\eta)y) =  \left\{\begin{array}{ll} 1 & \lambda\neq 0 \\ 0 & \lambda = 0,\,\lambda+\eta\neq 0 \\ -\infty & \lambda=\eta=0 \end{array}\right. \] which is indeed equal to the maximum of $\ell(\lambda(x+y))$ and $\ell(\eta y)$ (the former being $1$ if $\lambda\neq 0$ and $-\infty$ otherwise, and the latter being $0$ if $\eta\neq 0$ and $-\infty$ otherwise).  

On the other hand the pair $(x,x+y)$ is \emph{not} orthogonal: letting $\lambda=-1$ and $\eta=1$ we see that $\ell(\lambda x+\eta(x+y))=\ell(y)=0$ whereas $\max\{\ell(\lambda x),\ell(\eta(x+y))\}=1$.
 \end{ex}

Here are some simple but useful observations that follow directly from Definition \ref{dfnor}.

\begin{lemma}\label{bsorprop} If $(C, \ell)$ is an non-Archimedean normed vector space over $\Lambda$, then:
\begin{itemize} 
\item[(i)]If two subspaces $U$ and $V$ are orthogonal, then $U$ intersects $V$ trivially. 
\item[(ii)] For subspaces $U, V, W$, if $U$ and $V$ are orthogonal, and $U \oplus V$ and $W$ are orthogonal, then $U$ and $V \oplus W$ are orthogonal.
\item[(iii)] If $U$ and $V$ are orthogonal subspaces of $C$, and if $(u_1,\ldots,u_r)$ is an orthogonal ordered collection of elements of $U$ while $(v_1,\ldots,v_s)$ is an orthogonal ordered collection of elements of $V$, then $(u_1,\ldots,u_r,v_1,\ldots,v_s)$ is orthogonal in $U\oplus V$.
\end{itemize}\end{lemma}

\begin{proof}
For (i), if $w \in U \cap V$, then noting that (F2) implies that $\ell(-w)=\ell(w)$, we see that, since $w\in U$ and $-w\in V$ where $U$ and $V$ are orthogonal, \[ -\infty=\ell(0)=\ell(w+(-w))=\max\{\ell(w),\ell(w)\}=\ell(w) \] and so $w=0$ by (F1). So indeed $U$ intersects $V$ trivially.

For (ii), first note that if $U \oplus V$ and $W$ are orthogonal, then in particular, $V$ and $W$ are orthogonal. For any elements $u \in U, v \in V$ and $w \in W$, we have
\begin{align*}
\ell(u + (v + w)) = \ell((u+v) + w) &= \max\{\ell(u+v), \ell(w)\} \\
&= \max\{\ell(u), \ell(v), \ell(w)\} = \max\{\ell(u), \ell(v+w)\}. 
\end{align*}
The second equality comes from orthogonality between $U \oplus V$ and $W$; the third equality comes from orthogonality between $U$ and $V$; and the last equality comes from orthogonality between $V$ and $W$. 

Part (iii) is an immediate consequence of the definitions.
\end{proof}

\begin{dfn} {\bf An orthogonalizable $\Lambda$-space $(C,\ell)$} is a finite-dimensional non-Archimedean normed vector space over $\Lambda$ such that there exists an orthogonal basis for $C$.\end{dfn}

\begin{ex} $(\Lambda, - \nu)$ is an orthogonalizable $\Lambda$-space. \end{ex}

\begin{ex}\label{geneorvs} $(\Lambda^n, -\vec{\nu})$ is an orthogonalizable $\Lambda$-space, where $\vec{\nu}$ is defined as $\vec{\nu}(\lambda_1, ..., \lambda_n) = {\textstyle \min_{1 \leq i \leq n}} \nu(\lambda_i)$.  Moreover, fixing some vector $\vec{t} = (t_1, ..., t_n)\in\R^n$, the shifted version $(\Lambda^n, - \vec{\nu}_{\vec{t}})$ is also an orthogonalizable $\Lambda$-space, where $\vec{\nu}_{\vec{t}}$ is defined as 
\[\vec{\nu}_{\vec{t}}(\lambda_1, ..., \lambda_n) = {\textstyle \min_{1 \leq i \leq n}} (\nu(\lambda_i) - t_i).\]  Specifically, an orthogonal ordered  basis is given by the standard basis $(e_1,\ldots,e_n)$ for $\Lambda^n$: indeed, we have $-\vec{\nu}_{\vec{t}}(e_i)=t_i$, and \[ -\vec{\nu}_{\vec{t}}\left(\sum_{i=1}^{n}\lambda_ie_i\right)=\max_{1\leq i\leq n}(t_i-\nu(\lambda_i))=\max_{1\leq i\leq n}(-\vec{\nu}_{\vec{t}}(e_i)-\nu(\lambda_i)).\] 

%This explanation had problems---the $T^{t_i}$ aren't elements of $\Lambda$ in general (MJU 9/26) 
%In fact, first, it is easy to check that $-\vec{\nu}_{\vec{t}}$ is also a filtration function. Second, notice the following shifting operation on any vector $\vec{e} = (\lambda_1, ..., \lambda_n)$ in $\Lambda^n$ defined as 
%\[ {\vec e}^{\vec{t}} := (T^{t_1} \lambda_1, ..., T^{t_n} \lambda_n) \]
%is linear, that is $\lambda ({\vec e}^{\vec{t}}) = (\lambda \vec{e})^{\vec{t}}$ and $\vec{e}^{\vec{t}} + \vec{f}^{\vec{t}} = (\vec{e} + \vec{f})^{\vec{t}}$. More importantly, 
%\[ \vec{\nu}_{\vec{t}}(\vec{e}^{\vec{t}}) = \min_{1 \leq i \leq n} (\nu(T^{t_i} \lambda_i) - t_i) = \min_{1 \leq i \leq n} \nu(\lambda_i) = \vec{\nu}(\vec{e}). \]
%Therefore, taking any orthogonal basis of $(\Lambda^n, - \vec{\nu})$, say $(\vec{e}_1, ..., \vec{e}_n)$, we claim $(\vec{e}_1^{\vec{t}}, ..., \vec{e}_n^{\vec{t}})$ is an orthogonal basis of $(\Lambda^n, - \vec{v}_{\vec{t}})$. Indeed, for all $\eta_1, ..., \eta_n \in \Lambda$,
%\begin{align*}
%\vec{\nu}_{\vec{t}}\left( \sum_{1 \leq i \leq n} \eta_i \vec{e_i}^{\vec{t}} \right) &=  \vec{\nu}_{\vec{t}}\left(\left(\sum_{1 \leq i \leq n} \eta_i \vec{e}_i \right)^{\vec{t}}\right) = \vec{\nu} \left(\sum_{1 \leq i \leq n} \eta_i \vec{e}_i \right)= \min_{1 \leq i \leq n} \vec{\nu} (\eta_i \vec{e}_i) = \min_{1 \leq i \leq n} \vec{\nu}_{\vec{t}} (\eta_i \vec{e}^{\vec{t}}_i)
%\end{align*}
%where the third equality comes from the orthogonality of $(\vec{e}_1, ..., \vec{e}_n)$.

In Example \ref{NovComplex} above, if we let $\{\tilde{p}_i\}_{i=1}^{n}\subset \tilde{M}$ consist of one point in every fiber of the covering space $\tilde{M}\to M$ that contains an index-$k$ critical point, then it is easy to see that we have a vector space isomorphism $CN_{k}(\tilde{f})\cong \Lambda^n$, with the filtration function $\ell$ on $CN_{k}(\tilde{f})$ mapping to the shifted filtration function $-\vec{\nu}_{\vec{t}}$ where $t_i=\tilde{f}(\tilde{p}_i)$.

\end{ex}

\begin{remark}\label{nu-t-univ}
In fact, using (F2) and the definition of orthogonality, it is easy to see that \emph{any} orthogonalizable $\Lambda$-space $(C,\ell)$ is isomorphic in the obvious sense to some $(\Lambda^n,-\vec{\nu}_{\vec{t}})$: if $(v_1,\ldots,v_n)$ is an ordered orthogonal basis for $(C,\ell)$ then mapping $v_i$ to the $i$th standard basis vector for $\Lambda^n$ gives an isomorphism of vector spaces which sends $\ell$ to $-\vec{\nu}_{\vec{t}}$ where $t_i=\ell(v_i)$. 
\end{remark}

\subsection{Non-Archimedean Gram-Schmidt process}

In classical linear algebra, the Gram-Schmidt process is applied to modify a set of linearly independent elements into an orthogonal set. A similar procedure can be developed in the non-Archimedean context. The key part of this process comes from the following theorem, which we state using our notations in this paper (see Remark \ref{nu-t-univ}).

\begin{theorem}\label{U082.5} (\cite{Ush08}, Theorem 2.5). Suppose $(C,\ell)$ is an orthogonalizable $\Lambda$-space and $W \leq C$ is a $\Lambda$-subspace. Then for any $x \in C \backslash W$ there exists some $w_0 \in W$ such that 
\begin{equation}
\ell(x - w_0) = \inf\{\ell(x-w) \,| \, w \in W\}.
\end{equation} \end{theorem}

Thus $w_0$ achieves the minimal distance to $x$ among all elements of $W$.  Note that (in contrast to the situation with more familiar notions of distance such as the Euclidean distance on $\R^n$) the element $w_0$ is generally not unique.  However, similarly to the case of the Euclidean distance, solutions to this distance-minimization problem are closely related to orthogonality, as the following lemma shows. 

\begin{lemma}\label{testor} Let $(C,\ell)$ be a non-Archimedean normed vector space over  $\Lambda$, and let $W \leq C$ be a $\Lambda$-subspace and $x \in C \backslash W$. Then $W$ and $\left< x\right>_{\Lambda}$ are orthogonal if and only if $\ell(x) = \inf\{\ell(x - w) \,| \, w \in W\}$.\end{lemma}

\begin{proof} Suppose $W$ and $\left< x\right>_{\Lambda}$ are orthogonal. Then for any $w \in W$, by orthogonality, 
\[ \begin{array}{l}
\ell(x-w) = \max \{\ell(x), \ell(w)\} \geq \ell(x).
\end{array} \]
Therefore, taking an infimum, we get $\inf\{\ell(x - w) \,| \, w \in W\} \geq \ell(x)$. Moreover, by taking $w =0$, we have $\inf\{\ell(x - w) \,| \, w \in W\} \leq \ell(x-0) = \ell(x)$. Therefore, $\ell(x) = \inf\{\ell(x - w) \,| \, w \in W\}$.

Conversely, suppose that  $\ell(x) = \inf\{\ell(x - w) \,| \, w \in W\}$ and let $y=w+\mu x$ be a general element of $W\oplus\langle x\rangle_{\Lambda}$. We must show that $\ell(y)=\max\{\ell(w),\ell(\mu x)\}$; in fact, the inequality ``$\leq$'' automatically follows from (F3), so we just need to show  that $\ell(y)\geq \max\{\ell(w),\ell(\mu x)\}$.  If $\mu=0$ this is obvious since then $y=w$, so assume from now on that $\mu\neq 0$.  Then \[ \ell(y)=\ell\left(\mu(\mu^{-1}w+x)\right)
=\ell(\mu^{-1}w+x)-\nu(\mu)\geq \ell(x)-\nu(\mu)=\ell(\mu x) \] where the inequality uses the assumed optimality property of $x$.  If $\ell(\mu x)\geq \ell(w)$ this proves that $\ell(y)\geq \max\{\ell(w),\ell(\mu x)\}$.  On the other hand if $\ell(\mu x)<\ell(w)$ then the fact that $\ell(y)\geq \max\{\ell(w),\ell(\mu x)\}$ simply follows by Proposition \ref{strict}. \end{proof}

\begin{theorem}\label{GSp} (non-Archimedean Gram-Schmidt process). Let $(C,\ell)$ be an orthogonalizable $\Lambda$-space and let $\{x_1, ..., x_r\}$ be a basis for a subspace $V \leq C$.  Then there exists an \emph{orthogonal} ordered basis $(x_1', ..., x_r')$ for $V$ whose members have the form
\begin{align*}
& x_1' = x_1;\\
& x_2' = x_2 - \lambda_{2,1} x_1;\\
& \ldots\\
& x_r' = x_r - \lambda_{r, r-1} x_{r-1} - \lambda_{r, r-2} x_{r-2} -  \ldots - \lambda_{r,1} x_1,
\end{align*}
where $\lambda_{\alpha, \beta}$'s are constants in $\Lambda$.  Moreover if the first $i$ elements of the initial basis are such that $(x_1,\ldots,x_i)$ are orthogonal, then we can take $x'_j=x_j$ for $j=1,\ldots,i$.   \end{theorem}

\begin{proof} We proceed by induction on the dimension $r$ of $V$.  If $V$ is one-dimensional then we simply take $x'_1=x_1$.  Assuming the result to be proven for all $k$-dimensional subspaces, let $(x_1,\ldots,x_{k+1})$ be an ordered basis for $V$, with $(x_1,\ldots,x_i)$ orthogonal for some $i\in\{1,\ldots,k+1\}$. If $i=k+1$ then we can set $x'_j=x_j$ for all $j$ and we are done.  Otherwise apply the inductive hypothesis to the span $W$ of $\{x_1,\ldots,x_k\}$
to obtain an orthogonal ordered basis $(x'_1,\ldots,x'_k)$ for $W$, with $x'_j=x_j$ for all $j\in\{1,\ldots,i\}$.  Now apply Theorem \ref{U082.5} to $W$ and the element $x_{k+1}$ to obtain some $w_0\in W$ such that $\ell(x_{k+1}-w_0)=\inf\{\ell(x_{k+1}-w)|w\in W\}$. Let $x'_{k+1}=x_{k+1}-w_0$. It then follows from Lemma \ref{testor} that $W$ and $\langle x'_{k+1}\rangle_{\Lambda}$ are orthogonal, and so by Lemma \ref{bsorprop} (iii)  $(x'_1,\ldots,x'_k,x'_{k+1})$ is an orthogonal ordered basis for $V$.  Moreover since $x'_{k+1}=x_{k+1}-w_0$ where $w_0$ lies in the span of $x_1,\ldots,x_k$, it is clear that $x_{k+1}$ has the form required in the theorem.  This completes the inductive step and hence the proof.\end{proof}

\begin{cor}\label{orsub} If $(C,\ell)$ is an orthogonalizable $\Lambda$-space, then for every subspace $W \leq C$, $(W, \ell|_W)$ is also an orthogonalizable $\Lambda$-space.\end{cor}

\begin{proof}  Apply Theorem \ref{GSp} to an arbitrary basis for $W$ to obtain an orthogonal ordered basis for $W$. \end{proof}

\begin{cor}\label{orext} If $(C,\ell)$ is an orthogonalizable $\Lambda$-space and  $V \leq W \leq C$, any orthogonal ordered basis of $V$ may be extended to an orthogonal basis of $W$.\end{cor}

\begin{proof} By Corollary \ref{orsub}, we have an orthogonal ordered basis $(v_1, ..., v_i)$ for $V$. Extend it arbitrarily to a basis $\{v_1, ..., v_i, v_{i+1}, ..., v_r\}$ for $W$, and then apply Theorem \ref{GSp} to obtain an orthogonal ordered basis for $W$ whose first $i$ elements are $v_1,\ldots,v_i$. \end{proof}

\begin{cor}\label{orcomple} Suppose that $(C,\ell)$ is an orthogonalizable $\Lambda$-space and $U \leq C$. Then there exists a subspace $V$ such that $U \oplus V = C$ and $U$ and $V$ are orthogonal. (We call any such $V$ an {\it orthogonal complement} of $U$).\end{cor}

\begin{proof} By Corollary \ref{orsub}, we have an orthogonal ordered basis $(u_1, ..., u_k)$ for subspace $U$. By Corollary \ref{orext}, extend it to an orthogonal ordered basis for $C$, say $(u_1, ..., u_k, v_1, ..., v_l)$ (so $\dim(C) = k+l$). Then $V = span_{\Lambda} \{v_1, ..., v_l\}$ satisfies the desired properties. \end{proof}

Orthogonal complements are generally not unique, as is already illustrated by Example \ref{exor} in which $\langle x+ay\rangle_{\mathcal{K}}$ is an orthogonal complement to $\langle y\rangle_{\mathcal{K}}$ for any $a\in\mathcal{K}$.

\subsection{Duality} \label{dualss}

Given a non-Archimedean normed vector space $(C, \ell)$, the dual space $C^*$ (over $\Lambda$) becomes a non-Archimedean normed vector space if we associate a filtration function $\ell^*: C^* \rightarrow \mathbb R \cup \{\infty\}$ defined by 
\[ \begin{array}{l}
\ell^*(\phi) = {\displaystyle \sup_{0 \neq x \in C}} (-\ell(x) - \nu(\phi(x))).
\end{array} \]
Indeed, for $\phi$ and $\psi$ in $C^*$ and $x\in C$, we have 
\begin{align*}
-\ell(x) - \nu(\phi (x) + \psi (x)) &\leq - \ell(x) - \min\{\nu(\phi(x)), \nu(\psi(x))\} \\
& = \max\{- \ell(x) - \nu(\phi(x)), -\ell(x) - \nu(\psi(x))\} \leq \max\{\ell^*(\phi), \ell^*(\psi)\}
\end{align*}
and so taking the supremum over $x$ shows that $\ell^*(\phi+\psi)\leq \max\{\ell^*(\phi),\ell^*(\psi)\}$, and  it is easy to check the other 
axioms (F1) and (F3) required of $\ell^*$. The following proposition demonstrates a relation between bases of the original space and its dual space.

\begin{prop}\label{dualor} If $(C, \ell)$ is an orthogonalizable $\Lambda$-space with orthogonal ordered basis $(v_1, \ldots, v_n)$, then $(C^*, \ell^*)$ is an orthogonalizable $\Lambda$-space with an orthogonal ordered basis given by the dual basis $(v_1^*,\ldots, v_n^*)$. Moreover, for each $i$, we have 
\begin{equation} \label{flip}
\ell^*(v_i^*) = -\ell(v_i).
\end{equation} \end{prop}

\begin{proof} For any $x\in C$, written as ${\textstyle \sum_{j=1}^n \lambda_i v_i}$,  we have have $v_{i}^{*}x=\lambda_i$ for each $i$, so if $\lambda_i=0$ then $-\ell(x)-\nu(v_{i}^{*}x)=-\infty$, while otherwise
\begin{align*}
-\ell(x) - \nu(v_i^* x) &= -\max_{1 \leq j \leq n} (\ell(v_j) - \nu(\lambda_j) ) - \nu(\lambda_i) \\
& \leq -(\ell(v_i) - \nu(\lambda_i)) - \nu(\lambda_i)  = - \ell(v_i).
\end{align*}
Equality holds in the above when $x = v_i$, so $\ell^*(v_i^*) = -\ell(v_i)$.\\

To prove orthogonality, given any $\lambda_1,\ldots, \lambda_n \in \Lambda$,  choose $i_0$ to maximize the quantity $-\ell(v_{i}) - \nu(\lambda_{i})$ over $i \in \{1, \ldots, n\}$. Then 
\begin{align*}
\ell^*\left( \sum_{i=1}^n \lambda_i v_i^* \right) &\geq - \ell(v_{i_0}) - \nu\left( \left( \sum_{i=1}^n \lambda_i v_i^* \right) v_{i_0} \right)\\
& =- \ell(v_{i_0}) - \nu(\lambda_{i_0}) = \max_{1 \leq i \leq n} (\ell^* (v_i^*) - \nu(\lambda_i)).
\end{align*}
The reverse direction immediately follows from the non-Archimedean triangle inequality (F3) in Definition \ref{dfn-fil}. Therefore, we have proven the orthogonality of the dual basis.\end{proof}

\subsection{Coefficient extension}\label{coefsect}
This is a somewhat technical subsection which is not used for most of the main results---mainly we are including it in order to relate our barcodes to the torsion exponents from \cite{FOOO09}---so it could reasonably be omitted on first reading.

Throughout most of this paper we consider a fixed subgroup $\Gamma\leq \R$, with associated Novikov field $\Lambda=\Lambda^{\mathcal{K},\Gamma}$, and we consider orthogonalizable $\Lambda$-spaces over this fixed Novikov field $\Lambda$.  Suppose now that we consider a larger subgroup $\Gamma'\geq \Gamma$ (still with $\Gamma'\leq \R$).
 The inclusion $\Gamma\hookrightarrow \Gamma'$ induces in obvious fashion a field extension $\Lambda\hookrightarrow \Lambda^{\mathcal{K},\Gamma'}$, and so for any $\Lambda$ vector space $C$ we obtain a $\Lambda^{\mathcal{K},\Gamma'}$-vector space \[ C'=C\otimes_{\Lambda}\Lambda^{\mathcal{K},\Gamma'}.\]  

If $(C,\ell)$ is an orthogonalizable $\Lambda$-space with orthogonal ordered basis $(w_1,\ldots,w_n)$ then $\{w_1\otimes 1,\ldots,w_n\otimes 1\}$ is a basis for $C'$ and so  we can make $C'$ into an orthogonalizable $\Lambda^{\mathcal{K},\Gamma'}$-space $(C',\ell')$ by putting \[ \ell'\left(\sum_{i=1}^{n}\lambda'_{i}w_i\otimes 1\right) = \max_i\left(\ell(w_i)-\nu(\lambda'_i)\right)\] for all $\lambda'_1,\ldots,\lambda'_n\in\Lambda^{\mathcal{K},\Gamma'}$; in other words we are defining $\ell'$ by declaring $(w_1\otimes 1,\ldots,w_n\otimes 1)$ to be an orthogonal ordered basis for $(C',\ell')$.  The following proposition might be read as saying that this definition is independent of 
the choice of orthogonal basis $(w_1,\ldots,w_n)$ for $(C,\ell)$.

\begin{prop}\label{coextorth} With the above definition, if $(x_1,\ldots,x_n)$ is \textbf{any} orthogonal ordered basis for $(C,\ell)$ then $(x_1\otimes 1,\ldots,x_n\otimes 1)$ is an orthogonal ordered basis for $(C',\ell')$.
\end{prop}

\begin{proof} Let $(w_1,\ldots,w_n)$ denote the orthogonal basis that was used to define $\ell'$. Let $N\in GL_{n}(\Lambda)$ be the basis change matrix from $(w_1,\ldots,w_n)$ to $(x_1,\ldots,x_n)$, i.e., the matrix characterized by the fact that for $j\in\{1,\ldots,n\}$ we have $x_j={\textstyle \sum_i} N_{ij}w_i$. Then for $\vec{\lambda'}=(\lambda'_1,\ldots,\lambda'_n)\in(\Lambda^{\mathcal{K},\Gamma'})^n$ we have \begin{equation} \label{ellprime}\ell'\left(\sum_{j=1}^{n}\lambda'_j x_j\otimes 1\right) = \ell\left(\sum_{i=1}^{n} (N\vec{\lambda'})_i w_i\right) = \max_i \left(\ell(w_i)-\nu((N\vec{\lambda'})_i)\right).\end{equation}
Now the vector $\vec{\lambda'}\in (\Lambda^{\mathcal{K},\Gamma'})^n$ is a formal sum $\vec{\lambda'} = {\textstyle \sum_{g\in \Gamma'}}\vec{v}_g T^g$ where $\vec{v}_g\in \mathcal{K}^n$ and where the set of $g$ with $\vec{v}_g\neq 0$ is discrete and bounded below.  Let $S_{\vec{\lambda'}}\subset \Gamma'$ consist of those $g\in\Gamma'$ such that $g$ is the minimal element in its coset $g+\Gamma\subset \Gamma'$ having $\vec{v}_g\neq 0$.  We can then reorganize the above sum as \[ \vec{\lambda'} = \sum_{g\in S_{\vec{\lambda'}}}\vec{\lambda}_g T^g \] where now $\vec{\lambda}_g\in \Lambda^n$, and where  the set $S_{\vec{\lambda'}}$ is discrete and bounded below and has the property that distinct elements of $S_{\vec{\lambda'}}$ belong to distinct cosets of $\Gamma$ in $\Gamma'$.

Now since $N$ has its coefficients in $\Lambda$, we will have \[ N\vec{\lambda'} =  \sum_{g\in S_{\vec{\lambda'}}}N\vec{\lambda}_g T^g \] where each $N\vec{\lambda}_g\in \Lambda^n$. For each $i$ the various $\nu((N\vec{\lambda}_g)_i T^g)$ are equal to $g+\nu((N\vec{\lambda}_g)_i)$ and so belong to distinct cosets of $\Gamma$ in $\Gamma'$ (in particular, they are distinct from each other) and so we have for each $i$ \[ \nu((N\vec{\lambda'})_i) = \min_{g\in S_{\vec{\lambda'}}}\left(g+\nu((N\vec{\lambda}_g)_i)\right),\] and similarly $\nu(\lambda'_j)=\min_{g}(g+\nu((\vec{\lambda}_g)_j))$ for each $j$.
Combining this with (\ref{ellprime}) and  using the orthogonality of $(w_1,\ldots,w_n)$ and $(x_1,\ldots,x_n)$ with respect to $\ell$ and the fact that the $\vec{\lambda}_g$ belong to $\Lambda^n$ gives \begin{align*} \ell'\left(\sum_{j=1}^{n}\lambda'_j x_j\otimes 1\right)&=\max_{i,g}\left(\ell(w_i) - g -\nu((N\vec{\lambda}_g)_i)\right)
\\ &=\max_g\left(-g+\max_i(\ell(w_i)-\nu((N\vec{\lambda}_g)_i))\right) = \max_g\left(-g+\ell\left(\sum_i(N\vec{\lambda}_g)_iw_i\right)\right) 
\\ &= \max_g\left(-g+\ell\left(\sum_j(\vec{\lambda}_g)_jx_j\right)\right) = \max_g\left(-g+\max_j(\ell(x_j)-\nu((\vec{\lambda}_g)_j))\right)
\\ &= \max_j\left(\ell(x_j)-\min_g(g+\nu((\vec{\lambda}_g)_j))\right) = \max_j(\ell(x_j)-\nu(\lambda'_j)),\end{align*} proving the orthogonality of $(x_1\otimes 1,\ldots,x_n\otimes 1)$ since it follows directly from the original definition of $\ell'$ in terms of $(w_1,\ldots,w_n)$ that $\ell'(x\otimes 1)=\ell(x)$ whenever $x\in C$.
\end{proof}

\section {(Non-Archimedean) singular value decompositions}

Recall that in linear algebra over $\C$ with its standard inner product, a singular value decomposition for a linear transformation $A\co \C^n\to \C^m$ is typically defined to be a factorization $A=X\Sigma Y^*$ where $X\in U(m)$, $Y\in U(n)$, and $\Sigma_{ij}=0$ when $i\neq j$ while each $\Sigma_{ii}\geq 0$.  The ``singular values'' of $A$ are by definition the diagonal entries $\sigma_i=\Sigma_{ii}$, and then we have an orthonormal basis $(y_1,\ldots,y_n)$ for $\C^n$ (given by the columns of $Y$) and an orthonormal basis $(x_1,\ldots,x_m)$ for $\C^m$ (given by the columns of $X$) with $Ay_i=\sigma_ix_i$ for all $i$ with $\sigma_i\neq 0$, and $Ay_i=0$ otherwise.

An analogous construction for linear transformations between orthogonalizable $\Lambda$-spaces will play a central role in this paper.  In the generality in which we are working, we should not ask for the bases $(y_1,\ldots,y_n)$ to be ortho\emph{normal}, since an orthogonalizable $\Lambda$-space may not even admit an orthonormal basis (for the examples $(\Lambda^n,-\vec{\nu}_{\vec{t}})$ of Example \ref{geneorvs}, an orthonormal basis exists if and only if each $t_i$ belongs to the value group $\Gamma$).  However in the classical case asking for a singular value decomposition is equivalent to asking for orthogonal bases $(y_1,\ldots,y_n)$ for the domain and $(x_1,\ldots,x_m)$ for the codomain such that for all $i$ either $Ay_i=x_i$ or $Ay_i=0$; the singular values could then be recovered as the numbers $\textstyle {\frac{\|Ay_i\|}{\|y_i\|}}$.  This is precisely what we will require in the non-Archimedean context.  For the case in which the spaces in question do admit orthonormal bases (and so are equivalent to $(\Lambda^n,-\vec{\nu})$) such a construction can be found in \cite[Section 4.3]{Ked}.

\subsection{Existence of (non-Archimedean) singular value decomposition}

\begin{dfn} \label{dfnsvd} Let $(C, \ell_C)$ and $(D, \ell_D)$ be orthogonalizable $\Lambda$-spaces and let $A: C \rightarrow D$ be a linear map with rank $r$. A {\it singular value decomposition of $A$} is a choice of \textbf{orthogonal} ordered bases $(y_1, ..., y_n)$ for $C$ and $(x_1, ..., x_m)$ for $D$ such that:
\begin{itemize}
\item[(i)] $(y_{r+1}, ..., y_n)$ is an orthogonal ordered basis for $\ker A$;

\item[(ii)] $(x_1, ..., x_r)$ is an orthogonal ordered basis for ${\rm Im}A$;

\item[(iii)] $A y_i = x_i$ for $i \in \{1, ..., r\}$;

\item[(iv)] $\ell_C(y_1) - \ell_D(x_1) \geq \ldots \geq \ell_C(y_r) - \ell_D(x_r)$. \end{itemize} \end{dfn}

\begin{remark}
Consistently with the remarks at the start of the section, the singular values of $A$ would then be the quantities $e^{\ell_D(x_i)-\ell_C(y_i)}$ for $1\leq i\leq r$, as well as $0$ if $r<n$.  So the quantities $\ell_C(y_i)-\ell_D(x_i)$ from (iv) are the negative logarithms of the singular values.
\end{remark} 

\begin{remark}\label{unordered}
Occasionally it will be useful to consider data $\left((y_1,\ldots,y_n),(x_1,\ldots,x_m)\right)$ which satisfy all of the conditions of Definition \ref{dfnsvd} except condition (iv); such  $\left((y_1,\ldots,y_n),(x_1,\ldots,x_m)\right)$ will be called an \emph{unsorted singular value decomposition}.  Of course passing from an unsorted singular value decomposition to a genuine singular value decomposition is just a matter of sorting by the quantity $\ell_C(y_i)-\ell_C(x_i)$.
\end{remark}

The rest of this subsection will be devoted to proving the following existence theorem:

\begin{theorem} \label{existsvd} If $(C, \ell_C)$ and $(D, \ell_D)$ are orthogonalizable $\Lambda$-spaces, then any $\Lambda$-linear map $A: C \rightarrow D$ has a singular value decomposition.\end{theorem}

We will prove Theorem \ref{existsvd} by providing an algorithm (with proof) for producing a singular value decomposition of linear map $A$ between orthogonalizable $\Lambda$-spaces.\\
The algorithm is essentially Gaussian elimination, but with a carefully-designed rule for pivot selection which allows us to achieve the desired orthogonality properties.  In this respect it is similar to the algorithm from \cite{ZC} (that computes barcodes in classical persistent homology); however \cite{ZC} uses a pivot-selection rule which does not adapt well to our context where the value group $\Gamma$ may be nontrivial, leading us to use a different such rule.  Like the algorithm from \cite{ZC}, our algorithm requires a number of field operations that is at most cubic in the dimensions of the relevant vector spaces, and can be expected to do better than this in common situations where the matrix representing the linear map is sparse.  Of course, when working over a Novikov field there is an additional concern regarding how one can implement arithmetic operations in this field on a computer; we do not attempt to address this here.

\begin{theorem} \label{algsvd} (Algorithmic version of Theorem \ref{existsvd}). Let $(C,\ell_C)$ and $(D,\ell_D)$ be orthogonalizable $\Lambda$-spaces, let $A: C \rightarrow D$ be a $\Lambda$-linear map, and let $(v_1,\ldots,v_n)$ be an orthogonal ordered basis for $C$.  Then one may algorithmically construct an orthogonal ordered basis $(v'_1,\ldots,v'_n)$ of $C$ such that
\begin{itemize}
\item [(i)] $\ell_C(v'_i)=\ell_C(v_i)$ and $\ell_D(Av'_i)\leq \ell_D(Av_i)$ for each $i$;
\item[(ii)] Let $\mathcal{U} = \left\{\left.i\in \{1,\ldots,n\}\,\right| \,Av'_i\neq 0\right\}$. Then the ordered subset $(Av'_i\,|\,i\in\mathcal{U})$ is orthogonal in  $D$.\end{itemize} \end{theorem}

\begin{remark}\label{orthker}
In particular, $(v'_i\,|\,i\notin \mathcal{U})$ then gives an orthogonal ordered basis for $\ker A$.
\end{remark}

\begin{proof}
Fix throughout the algorithm an orthogonal ordered basis $(w_1,\ldots,w_m)$ for $D$. Represent $A$ by a matrix $(A_{ij})$ with respect to these bases, so that $Av_j={\textstyle \sum_iA_{ij}w_i}$.  Note that $v_j$ changes  as the algorithm proceeds (though the $w_i$ do not), so the elements $A_{ij}\in \Lambda$ will likewise change  in a corresponding way. Initialize the set of ``unused column indices'' to be $\mathcal{J}=\{1,\ldots,n\}$, and the set of ``{\rm pivot} pairs'' to be $\mathcal{P}=\varnothing$; at each step an element will be removed from $\mathcal{J}$ and an element will be added to $\mathcal{P}$. Here is the algorithm:

\begin{algorithm}[H]
\While{ $(\exists j\in \mathcal{J})(Av_j\neq 0)$ }{ 
 Choose $i_0\in\{1,\ldots,m\}$ and $j_0\in \mathcal{J}$ which maximize the quantity $\,\,\ell_D(w_i)-\nu(A_{ij})-\ell_C(v_j)$ over all $(i,j)\in\{1\ldots,m\}\times\mathcal{J}$ \;
 Add $(i_0,j_0)$ to the set $\mathcal{P}$\;
 Remove $j_0$ from the set $\mathcal{J}$\;
 For each $j\in\mathcal{J}$, replace $v_j$ by $v'_j:=v_j-{\textstyle \frac{A_{i_0j}}{A_{i_0j_0}}v_{j_0}}$\;
 For each $j\in \mathcal{J}$ and $i\in \{1,\ldots,m\}$, replace $A_{ij}$ by $A'_{ij}:= A_{ij}-{\textstyle \frac{A_{i_0j}A_{ij_0}}{A_{i_0j_0}}}$ (thus restoring the property that $Av_j={\textstyle \sum_{i=1}^{m}A_{ij}w_i}$)\;
}
\end{algorithm}

Note that the while loop predicate implies that in each iteration there is some $(i,j)\in\{1,\ldots,m\}\times\mathcal{J}$ such that $A_{ij}\neq 0$, so in particular $A_{i_0j_0}\neq 0$ (otherwise $A=0$) and so the divisions by $A_{i_0j_0}$ in the last two steps of the iteration are not problematic. The ordered basis $(v'_{1},\ldots,v'_{n})$ promised in the statement of this theorem is then simply the tuple to which $(v_1,\ldots,v_n)$ has evolved upon the termination of the while loop.  To prove that this satisfies the required properties it suffices to prove that, in each iteration of the while loop, the following assertions hold:

\begin{claim}\label{claim1} If the initial basis $(v_{1},\ldots,v_{n})$ is orthogonal, then so is the basis obtained by replacing $v_{j}$ by $v'_{j}=v_j-{\textstyle \frac{A_{i_0j}}{A_{i_0j_0}}v_{j_0}}$ for each $j\in\mathcal{J}\setminus\{j_0\}$. Moreover $\ell_C(v'_j)= \ell_C(v_j)$ while $\ell_D(Av'_j)\leq \ell_D(Av_j)$.\end{claim}

\begin{claim}\label{claim2} After each iteration, the ordered set $(Av_j\,|\,j\notin\mathcal{J})\subset D$ is orthogonal.\end{claim}

\begin{proof}[Proof of Claim \ref{claim1}] For any $j \in \mathcal J\setminus\{j_0\}$, by the orthogonality of $(v_1,\ldots,v_n)$ and the definition of $v'_j$, we have 
\[ \begin{array}{l}
\ell_C(v'_j) = \max\left\{\ell_C(v_j),\ell_C\left(\frac{A_{i_0j}}{A_{i_0j_0}}v_{j_0}\right)\right\}.
\end{array} \]
Because $(i_0,j_0)$ is chosen to satisfy  $\ell_D(w_{i_0}) - \nu(A_{i_0j_0}) - \ell_C(v_{j_0}) \geq \ell_D(w_{i}) - \nu(A_{ij}) - \ell_C(v_{j})$ for all $i$ and $j$, it in particular holds that
\[ \begin{array}{l}
\ell_D(w_{i_0}) - \nu(A_{i_0j_0}) - \ell_C(v_{j_0}) \geq \ell_D(w_{i_0}) - \nu(A_{i_0j}) - \ell_C(v_{j})
\end{array} \]
which can be rearranged to give 
\begin{equation}\label{J}
\ell_C\left({\textstyle \frac{A_{i_0j}}{A_{i_0j_0}}v_{j_0}}\right)\leq \ell_C(v_j).
\end{equation}
So we get 
\begin{equation}\label{j}
\ell_C(v'_j)=\ell_C(v_j).
\end{equation}
As for the statement about $\ell_D(Av'_j)$, note that 
 \[ \ell_D(Av_{j_0}) = \ell_D\left(\sum_{i=1}^{m}A_{ij_0}w_i\right)=\max_i(\ell_D(w_i)-\nu(A_{ij_0}))=\ell_D(w_{i_0})-\nu(A_{i_0j_0})\] where the last equation follows from the optimality criterion satisfied by $(i_0,j_0)$. Therefore,
 \begin{align*} 
 \ell_D\left(\frac{A_{i_0j}}{A_{i_0j_0}}Av_{j_0}\right)=\ell_D(w_{i_0})-\nu(A_{i_0j})\leq \max_{1\leq i\leq n}\ell_D(A_{ij}w_i)=\ell_D\left(\sum_{i=1}^{n}A_{ij}w_i\right)=\ell_D(Av_j)   \end{align*}
and hence $\ell_D(Av'_j)\leq{\textstyle \max\left\{\ell_D(Av_j),\ell_D\left(\frac{A_{i_0j}}{A_{i_0j_0}}Av_{j_0}\right)\right\}}=\ell_D(Av_j)$.\\

It remains to prove orthogonality of the basis obtained by replacing the $v_j$ by $v'_j$ for $j\in\mathcal{J}$. Here and for the rest of the proof we use the variable values as they are after the third step of the given iteration of the while loop---thus the $v_j$ have not been changed but $j_0$ has been removed from $\mathcal{J}$.  The new basis will be $\{v'_1,\ldots,v'_n\}$ where $v'_j=v_j$ if $j\notin \mathcal{J}$ and $v'_j=v_j-{\textstyle \frac{A_{i_0j}}{A_{i_0j_0}}v_{j_0}}$ otherwise. Let $\lambda_1,\ldots,\lambda_n\in\Lambda$ and observe that, by the orthogonality of $\{v_1,\ldots,v_n\}$,
\begin{align}
\ell_C\left(\sum_{j=1}^{n}\lambda_jv'_j\right)&=\ell_C\left(\sum_{j=1}^{n}\lambda_jv_j-\sum_{j\in\mathcal{J}}\lambda_j\frac{A_{i_0j}}{A_{i_0j_0}}v_{j_0}\right)\nonumber \\ \label{vjprime}
&= \max\left\{\ell_C\left(\left(\lambda_{j_0}-\sum_{k\in\mathcal{J}}\lambda_k\frac{A_{i_0k}}{A_{i_0j_0}}\right)v_{j_0}\right)  , \max_{j\neq j_0}\ell_C(\lambda_jv_j) \right\}.
\end{align}
If $\ell(\lambda_{j_0}v'_{j_0})>\ell(\lambda_j v'_j)$ for all $j\neq j_0$, then of course $\ell_C\left({\textstyle \sum_{j=1}^{n}}\lambda_j v'_j\right)=\ell_C(\lambda_{j_0}v'_{j_0})={\textstyle \max_j} \{\ell_C(\lambda_j v'_j)\}$. Otherwise, there is $j_1\neq j_0$ such that 
\begin{equation} \label{j_1}
{\textstyle \max_j}\ell_C(\lambda_jv'_j)=\ell_C(\lambda_{j_1}v'_{j_1}).
\end{equation}
Now by (\ref{j}) and the optimality condition (\ref{j_1}), we have  
\begin{equation} \label{j_0>j_1}
\ell_C(\lambda_{j_1}v_{j_1}) = \ell_C(\lambda_{j_1}v'_{j_1}) \geq\ell_C(\lambda_{j_0}v'_{j_0}) = \ell_C(\lambda_{j_0}v_{j_0}).
\end{equation}
Also, by (\ref{J}) and (\ref{j_1}), for all $k\in\mathcal{J}$, 
\[
\ell_C(\lambda_{j_1}v_{j_1})\geq \ell_C\left(\lambda_k {\textstyle\frac{A_{i_0k}}{A_{i_0j_0}}v_{j_0}}\right).
\]
Thus \begin{equation}  \label{A-} \ell_C(\lambda_{j_1}v_{j_1})\geq \ell_C\left(\left(\lambda_{j_0}-\sum_{k\in\mathcal{J}}\lambda_k  \frac{A_{i_0k}}{A_{i_0j_0}}\right)v_{j_0}\right). \end{equation}
So combining (\ref{vjprime}), (\ref{j_1}), and (\ref{A-}), we have \[ \ell_C\left(\sum_{j=1}^{n}\lambda_jv'_j\right)=\max_j\ell_C(\lambda_{j}v'_j),\] proving the orthogonality of $(v'_1,\ldots,v'_n)$. This completes the proof of Claim \ref{claim1}.\end{proof}

\begin{proof}[Proof of Claim \ref{claim2}] For $k\geq 1$ let $(i_k,j_k)$ denote the {\rm pivot} pair that is added to the set $\mathcal{P}$ during the $k$-th iteration of the while loop. In particular $j_k$ is removed from $\mathcal{J}$ during the $k$-th iteration, and after this removal we have $\mathcal{J}=\{1,\ldots,n\}\setminus\{j_1,\ldots,j_k\}$.  So the column operation in the last step of the $k$-th iteration replaces the matrix entries $A_{i_kj}$ for $j\notin\{j_1,\ldots,j_k\}$ by $A_{i_kj}-{\textstyle \frac{A_{i_kj}A_{i_kj_k}}{A_{i_kj_k}}}=0$. Moreover for $j\notin\{j_1,\ldots,j_k\}$ and any $i\in\{1,\ldots,m\}$ such that after the prior iteration we had $A_{ij_{k}}=A_{ij}=0$ (for instance this applies, inductively, to any $i\in\{i_1,\ldots,i_{k-1}\}$), the fact that $A_{ij}=0$ will be preserved after the $k$-th iteration. Thus, \begin{equation}\label{vanish}
\mbox{After the $k$th iteration, $A_{i_lj}=0$ for $l\in\{1,\ldots,k\}$ and } j\notin\{j_1,\ldots,j_l\}.
\end{equation}

We now show that, after the $k$-th iteration, the ordered set $(Av_{j_1},\ldots,Av_{j_k})$ is orthogonal; this is evidently equivalent to the statement of the claim. Note that, for $1\leq l\leq k$, neither the element $v_{j_l}$ nor the $j_l$-th column of the matrix $(A_{ij})$ changes during or after the $l$-th iteration of the while loop, due to the removal of $j_l$ from $\mathcal{J}$ during that iteration. For $l\in\{1,\ldots,k\}$, the optimality condition satisfied by the pair $(i_l,j_l)$ guarantees that $\ell_D(w_i)-\nu(A_{ij_l})\leq \ell_D(w_{i_l})-\nu(A_{i_lj_l})$ for all $i$ and hence 
\begin{equation}\label{av}
\ell_D(Av_{j_l})=\max_i(\ell_D(A_{ij_l}w_i))=\ell_D(A_{i_lj_l}w_{i_l}).
\end{equation}
Given $\lambda_1,\ldots,\lambda_{k}\in\Lambda$ we shall show that $\ell_D({\textstyle \sum_{l=1}^{k}}\lambda_lAv_{j_{l}})=\max_{l}\ell_D(\lambda_lAv_{j_l})$.  Let $l_0$ be the {\it smallest} element of $\{1,\ldots,k\}$ with the property that \[ \ell_D(\lambda_{l_0}A_{i_{l_0}j_{l_0}}w_{i_{l_0}})=\max_{1\leq l\leq k}\ell_D(\lambda_{l}A_{i_lj_l}w_{i_l}).\]  For all $i\in \{1,\ldots,m\}$ and $l\in\{1,\ldots,k\}$ we have, by the choice of $(i_l, j_l)$, \[ \ell_D(\lambda_lA_{ij_l}w_{i})\leq \ell_D(\lambda_{l}A_{i_lj_l}w_{i_l})\leq \ell_D(\lambda_{l_0}A_{i_{l_0}j_{l_0}}w_{i_{l_0}}). \]  Meanwhile, using (\ref{vanish}), $A_{i_{l_0}j_l}\neq 0$ only for $l\leq l_0$, and so
\[ 
\sum_l\lambda_{l}A_{i_{l_0}j_l}w_{i_{l_0}}
=\lambda_{l_0}A_{i_{l_0}j_{l_0}}w_{i_{l_0}}+\sum_{l<l_0}\lambda_{l}A_{i_{l_0}j_l}w_{i_{l_0}}
. \]
Each term $\lambda_{l}A_{i_{l_0}j_l}w_{i_{l_0}}$ has filtration level bounded above by $\ell_D(\lambda_lA_{i_lj_l}w_{i_l})$ by the second equality in (\ref{av}), and this latter filtration level is, for $l<l_0$, \emph{strictly lower than} $\ell_D(\lambda_{l_0}A_{i_{l_0}j_{l_0}}w_{i_{l_0}})$ because we chose $l_0$ as the smallest maximizer of
$\ell_D(\lambda_lA_{i_lj_l}w_{i_l})$.  So we in fact have \[
\ell_D\left(\sum_l\lambda_{l}A_{i_{l_0}j_l}w_{i_{l_0}}\right)  =  \ell_D(\lambda_{l_0}A_{i_{l_0}j_{l_0}}w_{i_{l_0}}).  \]
 
By the orthogonality of the ordered basis $(w_1,\ldots,w_m)$ we therefore have \begin{align*}
\ell_D\left(\sum_{l=1}^{k}\lambda_lAv_{j_l}\right)&=\ell_D\left(\sum_{l=1}^{k}\sum_{i=1}^{m}\lambda_lA_{ij_l}w_i\right) \\ 
&= \max_{1\leq i\leq m}\ell_D\left(\sum_{l=1}^{k}\lambda_{l}A_{ij_l}w_i\right)\geq \ell_D(\lambda_{l_0}A_{i_{l_0}j_{l_0}}w_{i_{l_0}}) \\
& = \max_{l}\ell_D(\lambda_{l}A_{i_lj_l}w_{i_l})=\max_l\ell_D(\lambda_l Av_{j_l})
\end{align*} where in the first equality in the third line we use the defining property of $l_0$ and in the last equality we use (\ref{av}).  Since the reverse inequality $\ell_D({\textstyle \sum_l\lambda_lAv_{j_l}})\leq \max_l\ell_D(\lambda_l Av_{j_l})$ is trivial this completes the proof of the orthogonality of $(Av_{j_1},\ldots,Av_{j_k})$.\end{proof} 
As noted earlier, Claims \ref{claim1} and \ref{claim2} directly imply that the basis for $C$ obtained at the termination of the while loop satisfies the required properties, thus completing the proof of Theorem \ref{algsvd}.
\end{proof}

\begin{proof}[Proof of Theorem \ref{existsvd}] First reorder the elements $v'_i$ produced by the Theorem \ref{algsvd} so that $Av'_i\neq 0$ if and only if $i\in \{1,\ldots,r\}$ where $r$ is the rank of $A$, and such that $\ell_C(v'_1)-\ell_D(Av'_1)\geq\cdots\geq \ell_C(v'_r)-\ell_D(Av'_r)$. If $A$ is surjective, then $\left((v'_1,\ldots,v'_n),(Av'_1,\ldots,Av'_r)\right)$ will immediately be a singular value decomposition for $A$.
More generally, we may use Corollary \ref{orcomple} to find an orthogonal complement of $\mathrm{Im}(A)$ in $D$, and by Corollary  \ref{orsub}  this orthogonal complement has some orthogonal ordered basis $(x_{r+1},\ldots,x_m)$. Then \\ 
$\left((v'_1,\ldots,v'_n),(Av'_1,\ldots,Av'_r,x_{r+1},\ldots,x_m)\right)$ is a singular value decomposition for $A$.  \end{proof}

\subsection{Duality and coefficient extension for singular value decompositions}

Proposition \ref{dualor} allows us to easily convert a singular value decomposition for a map $A\co C\to D$ to one for the adjoint map $A^*\co D^*\to C^*$. Explicitly: 

\begin{prop}\label{dualsvd} Let $(C,\ell_C)$ and $(D, \ell_D)$ be two orthogonalizable $\Lambda$-spaces and $A: C \rightarrow D$ be a $\Lambda$-linear map with rank $r$. Suppose $((y_1, ..., y_n), (x_1, ..., x_m))$ is a singular value decomposition for $A$. Then $((x_1^*, ..., x_m^*), (y_1^*, ..., y_n^*))$ is a singular value decomposition for its adjoint map $A^*: D^* \rightarrow C^*$.\end{prop}

\begin{proof} By the first assertion of Proposition \ref{dualor}, $(x_1^*, ..., x_m^*)$ is an orthogonal ordered basis for $D^*$ and $(y_1^*,..., y_n^*)$ is an orthogonal ordered basis for $C^*$. By the definition of a singular value decomposition, $Ay_i = x_i$ for $i \in \{1, ..., r\}$ and $Ay_i=0$ for $i \in \{r+1, ... ,n\}$, so $A^* x_i^* = y_i^*$ for $i \in \{1, ..., r\}$ and $A^*x_{i}^{*}=0$ for $i \in \{r+1, ..., m\}$. Therefore $(x_{r+1}^*, ..., x_{m}^*)$ is an orthogonal ordered basis for $\ker A^*$ and $\{y_1, ..., y_r\} = \{A^*x_1^*, ..., A^* x_r^*\}$ is an orthogonal ordered basis for ${\rm Im} A^*$. Finally, for $i \in \{1, ..., r\}$, by the second assertion of Proposition \ref{dualor}, we have 
\[ 
\ell^*_{D^*}(x_i^*) - \ell^*_{C^*}(y_i^*) = - \ell_D(x_i) + \ell_C(y_i) = \ell_C(y_i) - \ell_D(x_i). \]
So the ordering of $\ell_C(y_i) - \ell_D(x_i)$ implies the desired ordering for $\ell^*_{D^*}(x_i^*) - \ell^*_{C^*}(y_i^*)$.  \end{proof}

Similarly, Proposition \ref{coextorth} implies that singular value decompositions 
are well-behaved under coefficient extension.

\begin{prop} \label{coextsvd}
Consider two subgroups $\Gamma\leq \Gamma'\leq \R$, and write $\Lambda=\Lambda^{\mathcal{K},\Gamma}$ and $\Lambda'=\Lambda^{\mathcal{K},\Gamma'}$.  Let $(C,\ell_C)$ and $(D,\ell_D)$ be orthogonalizable $\Lambda$-spaces and let $A\co C\to D$ be a $\Lambda$-linear map, with singular value decomposition $\left((y_1,\ldots,y_n),(x_1,\ldots,x_m)\right)$.  Then if $C\otimes_{\Lambda}\Lambda'$ and $D\otimes_{\Lambda}\Lambda'$ are endowed with the filtration functions $\ell'_C$ and $\ell'_D$ as in Section \ref{coefsect}, the map $A\otimes 1\co C\otimes_{\Lambda}\Lambda'\to D\otimes_{\Lambda}\Lambda'$ has singular value decomposition given by $\left((y_1\otimes 1,\ldots,y_n\otimes 1),(x_1\otimes 1,\ldots,x_m\otimes 1)\right)$.
\end{prop}

\begin{proof}
Proposition \ref{coextorth} implies that the ordered sets $(y_1\otimes 1,\ldots,y_n\otimes 1)$ and $(x_1\otimes 1,\ldots,x_m\otimes 1)$ are orthogonal.  Moreover by definition of the relevant filtration functions we have $\ell'_C(y_i\otimes 1)=\ell_C(y_i)$ and $\ell'_D(x_i\otimes 1)=\ell_D(x_i)$ for all $i$ such that these are defined.  Once these facts are known it is a trivial matter to check each of the conditions (i)-(iv) in the definition of a singular value decomposition.
\end{proof}

\section{Boundary depth and torsion exponents via singular value decompositions} \label{beta}

The \emph{boundary depth} as defined in \cite{U11} or \cite{U13} is a numerical invariant of a filtered chain complex that, in the case of the Hamiltonian and Lagrangian Floer complexes, has been effectively used to obtain applications in symplectic topology. 
A closely related notion is that of the \emph{torsion threshold} and more generally the \emph{torsion exponents} that were introduced in \cite[Section 6.1]{FOOO09} for the Lagrangian Floer complex over the universal Novikov ring and were used in \cite{FOOO13} to obtain lower bounds for the displacement energies of polydisks.  We will see in this section that, for complexes like those that arise in Floer theory, both of these notions are naturally encoded in the (non-Archimedean) singular value decomposition of the boundary operator of the chain complex.  In particular our discussion will show that the boundary depth coincides with the torsion threshold when both are defined, and that certain natural generalizations of the boundary depth likewise coincide with the rest of the torsion exponents.  This implies new restrictions on the values that the torsion exponents can take.  Our generalized boundary depths will be part of the data that comprise the concise barcode of a Floer-type complex, our main invariant to be introduced in Section \ref{barsect}.

For the rest of the paper, we will always work with what we call a {\bf Floer-type complex} over a Novikov field $\Lambda$, defined as follows:

\begin{dfn} \label{dfnfcc} A {\bf Floer-type complex} $(C_{\ast}, \partial_C, \ell_C)$ over a Novikov field $\Lambda=\Lambda^{\mathcal{K},\Gamma}$ is a chain complex $(C_*=\oplus_{k\in\Z}C_k,\partial_C)$ over $\Lambda$ together with a function $\ell_C\co C_*\to \R\cup\{-\infty\}$  such that each $(C_k,\ell|_{C_k})$ is an orthogonalizable $\Lambda$-space, and for each $x\in C_k$ we have $\partial_Cx\in C_{k-1}$ with $\ell_C(\partial_C x)\leq \ell_C(x)$.\end{dfn}

\begin{ex} According to Example \ref{geneorvs}, the Morse, Novikov, and Hamiltonian Floer chain complexes are all Floer-type complexes. In each case the boundary operator is defined by counting connecting trajectories between two critical points for some function, which satisfy a certain  differential equation (see, e.g., \cite[Section 1.5]{Sal97} for the Hamiltonian Floer case).\end{ex}

\begin{remark}\label{nonstrict}
In fact in many Floer-type complexes including the Morse, Novikov, and Hamiltonian Floer complexes one has the strict inequality $\ell_C(\partial_C x)<\ell_C(x)$.  However it is also often useful in Morse and Floer theory to consider complexes where the inequality is not  necessarily strict; for instance the Biran-Cornea pearl complex \cite{BC} with appropriate coefficients can be described in this way, as can the Morse-Bott complex built from moduli spaces of ``cascades'' in \cite[Appendix A]{Fr}.  Also our definition allows other, non-Floer-theoretic, constructions such as the Rips complex (see Example \ref{RipsComplex}), and the mapping cylinders which  play a crucial role in the proofs of Theorem B and Theorem \ref{mainstab},  to be described as Floer-type complexes, whereas requiring $\ell_C(\partial_C x)<\ell_C(x)$ would rule these out.  In the case that one does have a strict inequality for the effect of the boundary operator on the filtration, the verbose and concise barcodes that we define later are easily seen to be equal to each other.
\end{remark}

\begin{dfn}\label{dfnfiso} Given two Floer-type complexes $(C_*, \partial_C, \ell_C)$ and $(D_*, \partial_D, \ell_D)$, a {\bf filtered chain isomorphism} between these two complexes is a chain isomorphism $\Phi: C_* \rightarrow D_*$  such that $\ell_D(\Phi(x)) = \ell_C(x)$ for all $x\in C_*$.\end{dfn}

\begin{dfn}\label{dfnfht} Given two Floer-type complexes $(C_*, \partial_C, \ell_C)$ and $(D_*, \partial_D, \ell_D)$, two chain maps $\Phi,\Psi\co C_*\to D_*$ are called {\bf filtered chain homotopic} if there exists $K: C_* \rightarrow D_{*+1}$ such that $\Phi - \Psi = \partial_D K+ K \partial_D$ and $K$ {\it preserves} filtration, i.e. $\ell_D(K(x)) \leq \ell_C(x)$ for all $x$, and both $\Phi$ and $\Psi$ preserve filtration as well. 

We say that $(C_*, \partial_C, \ell_C)$ is {\bf filtered homotopy equivalent} to $(D_*, \partial_D, \ell_D)$ if there exist chain maps $\Phi\co C_*\to D_*$ and $\Psi\co D_*\to C_*$ which both preserve filtration such that $\Psi\circ\Phi$ is filtered chain homotopic to identity $I_C$ while $\Phi\circ\Psi$ is filtered chain homotopic to the $I_D$. \end{dfn} 

In order to cut down on the number of indices that appear in our formulas, we will sometimes work in the following setting:

\begin{dfn} \label{dfntwo} A {\bf two-term Floer-type complex} $(C_1 \xrightarrow{\partial} C_0)$ is a Floer-type complex of the following form 
\[ \begin{array}{l}
\cdots \rightarrow 0 \rightarrow C_1 \xrightarrow{\partial} C_0 \rightarrow 0 \rightarrow \cdots\,.
\end{array} \]
Given any  Floer-type complex $(C_*, \partial_C, \ell_C)$, fixing a degree $k$, we can consider the following two-term Floer-type complex: 
\[ (\tilde{C}^{(k)}_1 \xrightarrow{\partial|_{C_k}} \tilde{C}^{(k)}_0)\]
where $\tilde{C}^{(k)}_1  = C_k$ and $\tilde{C}^{(k)}_0 = \ker (\partial|_{C_{k-1}}) (\leq C_{k-1})$.\end{dfn}

For the rest of this section, we will focus mainly on two-term  Floer-type complexes; consistently with the above discussion this roughly corresponds to focusing on a given degree in one of the multi-term chain complexes that we are ultimately interested in. For a  two-term Floer-type complex $(C_1 \xrightarrow{\partial} C_0)$, by Theorem \ref{existsvd} we may fix a singular value decomposition $((y_1, ..., y_n), (x_1, ..., x_m))$ for the boundary map $\partial\co C_1\to C_0$. Denote the rank of $\partial$ by $r$.  We will see soon that the numbers $\{\ell(y_i) - \ell(x_i)\}$ for $i \in \{1, ..., r\}$ (which have earlier been described as the negative logarithms of the singular values of $\partial$) can be characterized in terms of the following notion of {\it robustness} of the boundary operator. 

\begin{dfn} \label{dfn-robust} Let $\delta\in\R$.  An element $x\in C_0$ is said to be {\bf $\delta$-robust} if for all $y\in C_1$ such that $\partial y=x$ it holds that $\ell(y)>\ell(x)+\delta$. A subspace $V\leq C_0$ is said to be $\delta$-robust if every $x\in V\setminus\{0\}$ is $\delta$-robust.\end{dfn}

\begin{ex}  When $(C_1 \xrightarrow{\partial} C_0)$ is the two-term Floer-type complex $\widetilde{CM}^{(k)}_{*}(f)$ induced by the degree-$k$ and degree-$(k-1)$ parts of the Morse complex $CM_{*}(f)$ of a Morse function on a compact manifold, the reader may verify that each nonzero element of $C_0$ is $\delta$-robust for all $\delta<\delta_k$, where $\delta_k$ is the minimal positive difference between a critical value of an index-$k$ critical point and a critical value of an index-$(k-1)$ critical point.  Because a strict inequality is required in the definition of robustness, there may be elements of $C_0$ which are not $\delta_k$-robust.
\end{ex}

In the presence of our singular value decomposition $\left((y_1,\ldots,y_n),(x_1,\ldots,x_m)\right)$, the following simple observation is useful for checking $\delta$-robustness:

\begin{lemma}\label{checkingrobust} Let $x={\textstyle \sum_{i=1}^{r}}\lambda_ix_i$ be any element of ${\rm Im\partial}$, and suppose $y\in C_1$ obeys $\partial y=x$.  Then \[ \ell(y)\geq \ell\left(\sum_{i=1}^{r}\lambda_i y_i\right) = \max\{\ell(y_i)-\nu(\lambda_i)|1\leq i\leq r\}. \] \end{lemma}

\begin{proof} Since $\partial y_i=x_i$ for $1\leq i\leq r$ and $\partial y_i=0$ for $i>r$, and since the $x_i$ are linearly independent, the elements $y\in C_1$ such that $\partial y=x$ are precisely those of form ${\textstyle \sum_{i=1}^{r}}\lambda_i y_i +{\textstyle \sum_{i=r+1}^{n}}\mu_iy_i$ for arbitrary $\mu_{r+1},\ldots,\mu_n\in \Lambda$.  The proposition then follows directly from the fact that $(y_1,\ldots,y_n)$ is an orthogonal ordered basis for $C_1$.\end{proof}

\begin{dfn} \label{barlength} Given a two-term chain complex $(C_1\xrightarrow{\partial} C_0)$ and a positive integer $k$, let 
\[\begin{array}{l}
\beta_k(\partial) = \sup(\{0\} \cup \{\delta \geq 0\,|\,$There exists a\,\,$\delta$-robust subspace$\,V \leq {\rm {\rm Im} \partial}\, \,$with$\,\, \dim(V) =k\}).
\end{array} \]\end{dfn}

Note that  $\beta_k(\partial)=0$ if $\partial$ is the zero map or if $k > \dim({\rm {\rm Im} \partial})$. It is easy to see that, when $k\leq\dim({\rm {\rm Im} \partial})$, $\beta_k(\partial)$ can be rephrased as 
\[ \begin{array}{l}
\beta_k(\partial)= \displaystyle \sup_{\begin{tiny}\begin{array}{c}V\leq {\rm {\rm Im} \partial}\\ \dim(V) =k\end{array}\end{tiny}} \inf_{x\in V\setminus\{0\}}\{\ell(y) - \ell(x) \,| \, \partial y = x\}. 
\end{array} \]
When $k=1$, this is exactly the definition of boundary depth in \cite{U13} (see \cite[(24)]{U13}), and so we can view the $\beta_k(\partial)$ as generalizations of the boundary depth.  Clearly one has \[ \beta_1(\partial)\geq\beta_2(\partial)\geq \cdots\beta_k(\partial)\geq 0 \] for all $k$. We will prove the following theorem which relates the $\beta_k(\partial)$'s to singular value decompositions.

\begin{theorem} \label{genebd} Given a singular value decomposition $((y_1, ..., y_n), (x_1, ..., x_m))$ for a two-term chain complex $(C_1 \xrightarrow{\partial} C_0)$, the numbers $\beta_k(\partial)$ are given by \[ \beta_k(\partial) = \left\{\begin{array}{ll}\ell(y_k)-\ell(x_k) & 1\leq k\leq r \\ 0 & k>r \end{array}\right. \] where $r$ is the rank of $\partial$.\end{theorem}

\begin{proof} For each $k\in \{1, ..., r\}$, we will show that there exists a $k$-dimensional $\delta$-robust subspace of $\Img\partial$ for any $\delta< \ell(y_k) - \ell(x_k)$, but that no $k$-dimensional subspace is $(\ell(y_k) - \ell(x_k))$-robust. This clearly implies the result by the definition of $\beta_k(\partial)$.

Considering the subspace $V_k=span_{\Lambda}\{x_1, ..., x_k\}$, let $x={\textstyle \sum_{i=1}^{k}}\lambda_ix_i$ be any nonzero element in $V_k$.  Let $i_0\in \{1,\ldots,k\}$ maximize the quantity $\ell(x_i)-\nu(\lambda_i)$ over all $i\in\{1,\ldots,k\}$, so that by the orthogonality of the $x_i$ we have $\ell(x)=\ell(x_{i_0})-\nu(\lambda_{i_0})$.  Then, using the orthogonality of the $y_i$, \begin{align*}
\ell\left(\sum_{i=1}^{k}\lambda_iy_i\right)-\ell(x)&=\max_i(\ell(y_i)-\nu(\lambda_i))-(\ell(x_{i_0})-\nu(\lambda_{i_0}))
\\ &\geq (\ell(y_{i_0})-\nu(\lambda_{i_0}))-(\ell(x_{i_0})-\nu(\lambda_{i_0}))=\ell(y_{i_0})-\ell(x_{i_0})
\\ &\geq \ell(y_k)-\ell(x_k) 
\end{align*}
where the last inequality follows from our ordering convention for the $x_i$. But then by Lemma \ref{checkingrobust}, it follows that whenever $\partial y=x$ we have $\ell(y)-\ell(x)\geq \ell(y_k)-\ell(x_k)$.  Since this holds for an arbitrary  element $x\in span_{\Lambda}\{x_1,\ldots,x_k\}\setminus\{0\}$ we obtain that  $span\{x_1,\ldots,x_k\}$ is $\delta$-robust for all $\delta<\ell(y_k)-\ell(x_k)$.

Next, for any $k$-dimensional subspace $V\leq \Img\partial$, let $W=span_{\Lambda}\{x_k,x_{k+1},\ldots,x_r\}$. Since $W$ has codimension $k-1$ in ${\rm {\rm Im} \partial}$, the intersection $V\cap W$ contains some nonzero element $x$. Since $x\in W$ we can write $x={\textstyle \sum_{i=k}^{r}}\lambda_ix_i$ where not all $\lambda_i$ are zero. Choose $i_0\in \{k,\ldots,r\}$ to maximize the quantity $\ell(y_i)-\nu(\lambda_i)$ over $i\in\{k,\ldots,r\}$. Let $y={\textstyle \sum_{i=k}^{r}}\lambda_iy_i$.  Then we have $\partial y=x$, and 
\begin{align*}
\ell(y)-\ell(x) &= (\ell(y_{i_0})-\nu(\lambda_{i_0}))-\max_i(\ell(x_{i})-\nu(\lambda_i))\leq (\ell(y_{i_0})-\nu(\lambda_{i_0}))-(\ell(x_{i_0})-\nu(\lambda_{i_0})) \\& = \ell(y_{i_0})-\ell(x_{i_0})\leq \ell(y_k)-\ell(x_k)
\end{align*}
by our ordering convention for the $x_i$. So since $x\in V\setminus\{0\}$ (and since the inequality required in the definition of $\delta$-robustness is strict) this proves that  $V$ is not $(\ell(y_k)-\ell(x_k))$-robust.

Finally, when $k >r$, there is no $V \leq {\rm {\rm Im} \partial}$ such that $\dim(V) = k$ (since $\dim({\rm {\rm Im} \partial}) = k$). Then by definition of $\beta_k(\partial)$, it is zero. \end{proof}

Note that Definition \ref{barlength} makes clear that $\beta_k(\partial)$ is independent of the choice of singular value decomposition; thus we deduce the non-obvious fact that the difference $\ell(y_k)-\ell(x_k)$ is likewise independent of the choice of singular value decomposition for each $k \in \{1, ..., r\}$. Note also that any filtration-preserving $\Lambda$-linear map $A$ between two orthogonalizable $\Lambda$-spaces $C$ and $D$ can just as well be viewed as a two-term chain complex $(C \xrightarrow{A} D)$, and so we obtain generalized boundary depths $\beta_k(A)$. Theorem \ref{existsvd} or Theorem \ref{algsvd} provides a systematic way to compute $\beta_k(A)$.   It is also clear from the definition that if $A\co C\to D$ has image contained in some subspace $D'\leq D$ then $\beta_k(A)$ is the same regardless of whether we regard $A$ as a map $C\to D$ or as a map $C\to D'$.  For instance if $(C_*,\partial_C,\ell_C)$ is a Floer-type complex, for any $i$ we could consider either of the two-term complexes $(C_{i} \xrightarrow{\partial|_{C_i}} C_{i-1})$ or $(C_i\xrightarrow{\partial|_{C_i}} \ker (\partial_C|_{C_{i-1}}))$ and obtain the same values of $\beta_k$.

We conclude this section by phrasing the torsion exponents of \cite{FOOO09}, \cite{FOOO13} in our terms and proving that these torsion exponents coincide with our generalized boundary depths $\beta_k$.  We will explain this just for two-term Floer-type complexes $(C_1 \xrightarrow{\partial} C_0)$; this represents no loss of generality, as for a general Floer-type complex $(C_*,\partial_C,\ell_C)$ one may apply the discussion below to the various two-term Floer-type complexes $(C_{i+1}\xrightarrow{\partial|_{C_{i+1}}} \ker (\partial_C|_{C_{i}}))$ in order to relate the torsion exponents and generalized boundary depths in any degree $i\in\Z$.

So let $(C_1 \xrightarrow{\partial} C_0)$ be a two-term Floer-type complex over $\Lambda = \Lambda^{\mathcal{K},\Gamma}$.  We first define the torsion exponents (in degree zero) in our language, leaving it to readers familiar with \cite{FOOO09} to verify that our definition is consistent with theirs.  Write $\Lambda^{univ}=\Lambda^{\mathcal{K},\R}$ for the ``universal'' Novikov field, so named because regardless of the choice of $\Gamma$ we have a field extension $\Lambda^{\mathcal{K},\Gamma}\hookrightarrow \Lambda^{univ}$.  Also define \[ \Lambda^{univ}_{0}=\{\lambda\in \Lambda^{univ}\,|\,\nu(\lambda)\geq 0\}; \]  thus $\Lambda^{univ}_{0}$ is the subring of $\Lambda^{univ}$ consisting of formal sums $\sum_g a_gT^g$ with each $g\geq 0$.

As in Section \ref{coefsect}, for $j=0,1$ let $C'_{j}=C_j\otimes_{\Lambda} \Lambda^{univ}$, and endow $C'_j$ with the filtration function obtained by choosing an orthogonal ordered basis $(w_1,\ldots,w_a)$ for $C_j$ and putting $\ell'\left(\sum_i \lambda'_i w_i\otimes 1\right)=\max_i(\ell(w_i)-\nu(\lambda'_i))$ for any $\lambda'_1,\ldots,\lambda'_a\in \Lambda^{univ}$.  By Proposition \ref{coextorth} this definition is independent of the choice of orthogonal basis $(w_1,\ldots,w_a)$.

Now, for $j=0,1$, define \[ \bar{C}'_j=\{c\in C'_j\,|\,\ell'(c)\leq 0\} \]  and observe that $\bar{C}_j$ is a module over the subring $\Lambda^{univ}_{0}$ of $\Lambda^{univ}$.  Moreover, again taking Proposition \ref{coextorth} into account, it is easy to see that if $(w_1,\ldots,w_a)$ is \emph{any} orthogonal ordered basis for $C_j$, then the elements $\bar{w}_i=w_i\otimes T^{\ell(w_i)}$ form a basis for $\bar{C}'_j$ as a $\Lambda^{univ}_{0}$-module.

The fact that $\ell(\partial c)\leq \ell(c)$ implies that the coefficient extension $\partial\otimes 1\co C'_1\to C'_0$ restricts to $\bar{C}'_1$ as a map to $\bar{C}'_0$.  So we have a (two-term) chain complex of $\Lambda^{univ}_{0}$-modules 
$(\bar{C}'_1 \xrightarrow{\partial\otimes 1} \bar{C}'_0)$.  
Fukaya, Oh, Ohta, and Ono show \cite[Theorem 6.1.20]{FOOO09} that the zeroth homology of this complex (i.e., the quotient $\bar{C}'_0/(\partial\otimes 1)\bar{C}'_1$) is isomorphic to \begin{equation}\label{tors} (\Lambda^{univ}_{0})^q\oplus \bigoplus_{k=1}^{s}(\Lambda^{univ}_{0}/T^{\lambda_k}\Lambda^{univ}_{0}) \end{equation} for some natural numbers $q,s$ and positive real numbers $\lambda_i,\ldots,\lambda_s$.

\begin{dfn}[\cite{FOOO09}] Order the summands in the decomposition (\ref{tors}) of $\bar{C}'_0/(\partial\otimes 1)\bar{C}'_1$ so that $\lambda_1\geq\cdots\geq\lambda_s$.  For a positive integer $k$, the \textbf{$k$th torsion exponent} of the two-term Floer-type complex 
$(C_1 \xrightarrow{\partial} C_0)$ is $\lambda_k$ if $k\leq s$ and $0$ otherwise.  The first torsion exponent is also called the \textbf{torsion threshold}.
\end{dfn}

\begin{theorem}\label{torsion} For each positive integer $k$ the $k$th torsion exponent of $(C_1 \xrightarrow{\partial} C_0)$ is equal to the generalized boundary depth $\beta_k(\partial)$.
\end{theorem}

\begin{proof}
Let $\left((y_1,\ldots,y_n),(x_1,\ldots,x_m)\right)$ be a singular value decomposition for $\partial\co C_1\to C_0$.  By Proposition \ref{coextsvd}, $\left((y_1\otimes 1,\ldots,y_n\otimes 1),(x_1\otimes 1,\ldots,x_m\otimes 1)\right)$ is a singular value decomposition for $\partial\otimes 1\co C'_1\to C'_0$.  Let $r$ denote the rank of $\partial$ (equivalently, that of $\partial\otimes 1$).

Let us determine the image $(\partial\otimes 1)(\bar{C}'_1)\subset C'_0$.  A general element $x$ of $C'_0$ can be written as $x=\sum_{i=1}^{m}\lambda_i x_i\otimes 1$ where $\lambda_i\in \Lambda^{univ}$.  By the definition of a singular value decomposition, in order for $x$ to be in the image of $\partial\otimes 1$ we evidently must have $\lambda_i=0$ for $i>r$.  Given that this holds, we will have $(\partial\otimes 1)\left(\sum_{i=1}^{r}\lambda_iy_i\otimes 1\right)=x$, and moreover by Lemma \ref{checkingrobust}, $\sum_{i=1}^{r}\lambda_iy_i\otimes 1$ has the lowest filtration level among all preimages of $x$ under $\partial \otimes 1$.  Now \[ \ell'\left(\sum_{i=1}^{r}\lambda_iy_i\otimes 1\right)=\max_i (\ell(y_i)-\nu(\lambda_i)), \] so we conclude that $x=\sum_{i=1}^{m}\lambda_i x_i\otimes 1$ belongs to 
$(\partial\otimes 1)(\bar{C}'_1)$ if and only if both $\lambda_i=0$ for $i>r$ and $\nu(\lambda_i)\geq \ell(y_i)$ for $i=1,\ldots,r$.

Recall that the elements $\bar{x_i}=x_i\otimes T^{\ell(x_i)}$ form a $\Lambda^{univ}_{0}$-basis for $\bar{C}'_0$.  Letting $\mu_i=T^{-\ell(x_i)}\lambda_i$,   the conclusion of the above paragraph can be rephrased as saying that $(\partial\otimes 1)(\bar{C}'_1)$ consists precisely of elements $\sum_{i=1}^{m}\mu_i\bar{x}_i$ such that $\mu_i=0$ for $i>r$ and $\nu(\mu_i)\geq \ell(y_i)-\ell(x_i)$ for $i=1,\ldots,r$.  Now for any $\mu\in \Lambda^{univ}$ and $c\in \R$, one has $\nu(\mu)\geq c$ if and only if $\mu\in T^c\Lambda^{univ}_{0}$.  So we conclude that 
\[ (\partial\otimes 1)(\bar{C}'_1)=span_{\Lambda^{univ}_{0}}\{T^{\ell(y_1)-\ell(x_1)}\bar{x}_1,\ldots,T^{\ell(y_r)-\ell(x_r)}\bar{x}_r\}, \] while as mentioned earlier \[
\bar{C}'_0=span_{\Lambda^{univ}_{0}}\{\bar{x}_1,\ldots,\bar{x}_m\}.
\]
These facts immediately imply that \[ \frac{\bar{C}'_0}{(\partial\otimes 1)(\bar{C}'_1)} = (\Lambda^{univ}_{0})^{m-r}\oplus \bigoplus_{k=1}^{r}(\Lambda^{univ}_{0}/T^{\ell(y_k)-\ell(x_k)}\Lambda^{univ}_{0}). \]  Comparing with (\ref{tors}) we see that the numbers that we have denoted by $s$ and $r$ are equal to each other, and that the $k$th torsion exponent is equal to $\ell(y_k)-\ell(x_k)$ for $1\leq k\leq r$ and to zero otherwise. By Theorem \ref{genebd} this is the same as $\beta_k(\partial)$.
\end{proof}

\section{Filtration spectrum}

The filtration spectrum of an orthogonalizable $\Lambda$-space is an algebraic abstraction of the set of critical values of a Morse function or the action spectrum of a Hamiltonian diffeomorphism (cf. \cite{Sc00}). 

In the definition below and elsewhere, our convention is that $\mathbb{N}$ is the set of nonnegative integers (so includes zero).

\begin{dfn} A {\bf multiset} $M$ is a pair $(S, \mu)$ where $S$ is a set and $\mu\co S \rightarrow \mathbb N\cup\{\infty\}$ is a function, called the {\bf multiplicity function} of $M$.  If $T$ is some other set, a \textbf{multiset of elements of $T$} is a multiset $(S,\mu)$ such that $S\subset T$.\end{dfn}

For $s\in S$, the value $\mu(s)$ should be interpreted as ``the number of times that $s$  appears'' in the multiset $M$.    By abuse of notation we will sometimes denote multisets in set-theoretic notation with elements repeated: for instance $\{1,3,1,2,3\}$ denotes a multiset with $\mu(1)=\mu(3)=2$ and $\mu(2)=1$.  The cardinality of the multiset $(S,\mu)$ is by definition $\sum_{s\in S}\mu(S)$. (For notational simplicity we are not distinguishing between different infinite cardinalities in our definition; in fact, for nearly all of the multisets that appear in this paper the multiplicity function will only take finite values.)

Also, if $S\subset T$ and $\mu\co T\to \N\cup\{\infty\}$ is a function with $\mu|_{T\setminus S}\equiv 0$ then we will not distinguish between the multisets $(T,\mu)$ and $(S,\mu|_S)$.

\begin{dfn}\label{dfnfs} Let $(C, \ell)$ be an orthogonalizable $\Lambda$-space with a fixed orthogonal ordered basis $(v_1, ..., v_n)$. The {\bf filtration spectrum} of $(C, \ell)$ is the multiset $(\mathbb R/\Gamma, \mu)$ where
\[ \begin{array}{l}
\mu(s) = \#\{v_i \in \{v_1, ..., v_n\} \,| \, \ell(v_i) \equiv s\,\,$mod$\,\Gamma\}.
\end{array} \]
\end{dfn}

\begin{remark} When $\Gamma$ is trivial, the filtration spectrum is just the set $\{\ell(v_1), ..., \ell(v_n)\}$ and multiplicity function is just defined by setting $\mu(s)$ equal to the number of $i$ such that $\ell(v_i)=s$.\end{remark}

\begin{ex} Let $\Gamma = \mathbb Z$ and $C = span_{\Lambda} \{v_1, v_2\}$ where $v_1,v_2$ are orthogonal with $\ell(v_1) = 2.5$ and $\ell(v_2) = 0.5$. Then for $[0.5] \in \mathbb R/\Gamma$, $\mu([0.5]) = 2$, while for $[0.7] \in \mathbb R/\Gamma$, $\mu([0.7]) = 0$. The filtration spectrum is then the multiset $\{[0.5], [0.5]\}$.\end{ex}

While Definition \ref{dfnfs} relies on a choice of an orthogonal basis for $(C,\ell)$, the following proposition shows that the filtration spectrum can be reformulated in a way that is manifestly independent of the choice of orthogonal basis, and so is in fact an invariant of the orthogonalizable $\Lambda$-space $(C,\ell)$. 

\begin{prop} \label{multi} Let $(C,\ell)$ be an orthogonalizable $\Lambda^{\mathcal{K},\Gamma}$-space and let $(\R/\Gamma,\mu)$ be the filtration spectrum of $(C,\ell)$ (as determined by an arbitrary orthogonal basis).  Then for any $s\in\R/\Gamma$, 
\[ \begin{array}{l}
\mu(s) = \max\left\{k \in \N\,\left|\,(\exists V \leq C) \left(\dim(V) =k\,\,\mbox{and}\, (\forall v\in V\backslash\{0\}) (\ell(v)\equiv s \,\mbox{mod}\,\Gamma)\right)\right.\right\}.
\end{array}\] \end{prop}

\begin{proof} Let $(v_1,\ldots,v_n)$ be an orthogonal ordered basis of $C$ and let $\mu$ be the multiplicity of some element $s \in \R/\Gamma$ in the filtration spectrum of $C$.  So by definition there are precisely $\mu$ elements $i_1,\ldots,i_{\mu} \in \{1,\ldots, n\}$ such that each $\ell(v_{i_j})\equiv s\,\,$mod$\,\Gamma$ for $j=1,\ldots,\mu$.  Any nonzero element $u$ in the $\mu$-dimensional subspace spanned by the $v_{i_j}$ can be written as $u= {\textstyle \sum_j}\lambda_jv_{i_j}$ where $\lambda_j\in\Lambda$ are not all zero, and then $\ell(u)={\textstyle \max_j}\{\ell(v_{i_j})- \nu(\lambda_j)\} \equiv s$\,mod $\Gamma$ since $\nu(\lambda_j)$ all belong to $\Gamma$.  This proves that $\mu$ is less than or equal to right hand side in the statement of the proposition.

For the reverse inequality, suppose that $V\leq C$ has dimension greater than $\mu$. For $i_1,\ldots,i_{\mu}$ as in the previous paragraph, let $W=span_{\Lambda}\{v_i\,|\,i\notin\{i_1,\ldots,i_{\mu}\}\}$.  Since $W$ has codimension $\mu$ and $\dim V>\mu$, $V$ and $W$ intersect non-trivially.  So there is some nonzero element $v={\textstyle \sum_{i\notin\{i_1,\ldots,i_{\mu}\}}}\lambda_i v_i\in V\cap W$.  Since the $v_i$'s are orthogonal, $\ell(v)$ has the same reduction modulo $\Gamma$ as one of the $v_i$ with $i\notin \{i_1,\ldots,i_{\mu}\}$, and so this reduction is not equal to $s$.  Thus no subspace of dimension greater than $\mu$ can have the property indicated in the statement of the proposition.\end{proof}

\begin{remark}\label{brmk}
Let us now relate our singular value decompositions to the Morse-Barannikov complex $\mathcal{C}(f)$ of an excellent Morse function $f\co M\to\R$ on a Riemannian manifold as described in \cite[Section 2]{LNV}, where the term ``excellent'' means in particular that the restriction of $f$ to its set of critical points is injective.  

This latter assumption means, in our language, that  the filtration spectrum of the orthogonalizable $\mathcal{K}$-space $(CM_{*}(f),\ell)$ consists of the index-$k$ critical values of $f$, each occurring with multiplicity one, since (essentially by definition)  $(CM_{*}(f),\ell)$ has an orthogonal basis given by the critical points of $f$, with filtrations given by their corresponding critical values.   So in view of Proposition \ref{multi}, the filtration function $\ell$ will restrict to any other orthogonal basis of $(CM_{*}(f),\ell)$ as a bijection to the set of critical values of $f$.  

Denoting by $\partial$ the boundary operator on $CM_{*}(f)$, Theorem \ref{existsvd} allows us to construct an orthogonal ordered basis $(x_1,\ldots,x_r,y_1,\ldots,y_r,z_1,\ldots,z_h)$ for $CM_{*}(f)$ such that $span\{x_1,\ldots,x_r\}=\Img(\partial)$, $span\{x_1,\ldots,x_r,z_1,\ldots,z_h\}=\ker(\partial)$, and $\partial y_i=x_i$. By the previous paragraph, then, each critical value $c$ of $f$ can then be written in exactly one way as $c=\ell(x_i)$ or $c=\ell(y_i)$ or $c=\ell(z_i)$.  

For $\lambda\in \R$, let $C^{\lambda}_{*}$ denote the subcomplex of $CM_{*}(f)$ spanned by the critical points with critical value at most $\lambda$.  Observe that $C^{\lambda}_{*}$ is equal to the subcomplex of $CM_{*}(f)$ spanned by the $x_i,y_i,z_i$ having $\ell\leq \lambda$ (indeed the latter is clearly a subspace of $C^{\lambda}_{*}$, but Proposition \ref{multi} implies that their dimensions are the same).  Now the treatment of the Barannikov complex in \cite{LNV} involves separating the critical values $c$ of $f$ into three types, where $\ep$ represents a small positive number:
\begin{itemize}
\item The \emph{lower} critical values, for which the natural map $H_{*}(C^{c+\ep}_{*}/C^{c-\ep}_{*})\to H_{*}(CM_{*}(f)/C^{c-\ep}_{*})$ vanishes;
\item The \emph{upper} critical values, for which the natural map $H_{*}(C^{c+\ep}_{*})\to H_{*}(C^{c+\ep}_{*},C^{c-\ep}_{*})$ vanishes (equivalently,  $H_{*}(C^{c-\ep}_{*})\to H_{*}(C^{c+\ep}_{*})$ is surjective);
\item All other critical values, called \emph{homological} critical values.
\end{itemize}

If $w$ is any of $x_i,y_i,$ or $z_i$ and if $\ell(w)=c$, one has $C^{c+\ep}_{*}=C^{c-\ep}_{*}\oplus\langle w\rangle$.  Consequently it is easy to see that $c$ is a lower critical value if and only if $c=\ell(x_i)$ for some $i$, that $c$ is an upper critical value if and only if $c=\ell(y_i)$ for some $i$, and that $c$ is a homological critical value if and only if $c=\ell(z_i)$ for some $i$.  Moreover, in the case that $c$ is an upper critical value so that $c=\ell(y_i)$ for some $i$, the natural map $H_*(C^{c+\ep}_{*}/C^{\lambda}_{*})\to H_*(C^{c+\ep}_{*}/C^{c-\ep}_{*})$ vanishes precisely for $\lambda\leq \ell(x_i)$.  

In \cite[Definition 2.9]{LNV}, the Morse-Barannikov complex $(\mathcal{C}(f),\partial_B)$ is described as the chain complex generated by the critical values of $f$, with boundary operator given by $\partial_Bc=0$ if $c$ is a lower critical value or a homological critical value, and \[ \partial_Bc = \sup\{\lambda|H_*(C^{c+\ep}_{*}/C^{\lambda}_{*})\to H_*(C^{c+\ep}_{*}/C^{c-\ep}_{*})\mbox{ is the zero map}\} \] if $c$ is an upper critical value.  The foregoing discussion shows that the unique linear map $(CM_{*}(f),\partial)\to (\mathcal{C}(f),\partial_B)$ that sends the basis elements $x_i,y_i,z_i$ to their respective filtration levels $\ell(x_i),\ell(y_i),\ell(z_i)$ defines an isomorphism of chain complexes.  In particular, the Morse-Barannikov complex can be recovered quite directly from a singular value decomposition.
\end{remark}

\section{Barcodes}\label{barsect}

Recall from the introduction that a persistence module $\mathbb{V}=\{V_t\}_{t\in \R}$ over the field $\mathcal{K}$ is a system of $\mathcal{K}$-vector spaces $V_t$ with suitably compatible maps $V_s\to V_t$ whenever $s\leq t$.  

A special case of a persistence module is obtained by choosing an interval $I\subset \R$ and defining 
\[ (\mathcal{K}_{I})_{t}=\left\{\begin{array}{ll}\mathcal{K} & t\in I \\ 0 & t\notin I \end{array}\right. \] with the maps $(\mathcal{K}_{I})_{s}\to (\mathcal{K}_{I})_{t}$ defined to be the identity when $s,t\in I$ and to be zero otherwise.

A persistence module $\mathbb{V}$ is called \emph{pointwise finite-dimensional} if each $V_t$ is finite-dimensional.  Such persistence modules obey the following structure theorem. 

\begin{theorem} \label{strthm}(\cite{ZC}, \cite{CB})  Every pointwise finite-dimensional persistence module $\mathbb V$ can be uniquely decomposed into the following normal form 
\begin{equation} \label{normalform}\begin{array}{l}
\mathbb V \cong \bigoplus_{\alpha} \mathcal{K}_{I_{\alpha}}
\end{array} \end{equation}
for certain intervals $I_{\alpha}\subset \R$
\end{theorem}

The (persistent homology) \textbf{barcode} of $\mathbb{V}$ is then by definition the multiset $(S,\mu)$ where $S$ is the set of intervals $I$ for which $\mathcal{K}_I$ appears in (\ref{normalform}) and $\mu(I)$ is the number of times that $\mathcal{K}_I$ appears.  As follows from the discussion at the end of the introduction in \cite{CB}, the barcode is a complete invariant of a finite-dimenisonal persistence module.

In classical persistent homology, where the persistence module is constructed from the filtered homologies of the \v{C}ech or Rips complexes associated to a point cloud,
\cite{ZC} provides an algorithm computing the resulting barcode (cf. Theorem \ref{algsvd} below).  In this case the intervals in the barcode are all half-open intervals $[a,b)$ (with possibly $b=\infty$). See, e.g., \cite[Figure 4]{Ghr08}, \cite[p. 278]{CBull} for some nice illustrations of barcodes. 

Returning to the context of the Floer-type complexes $(C_*, \partial, \ell)$ considered in this paper, for any $t\in\R$, if we let $C_{k}^{t}=\{c\in C_k|\ell(c)\leq t\}$ the assumption on the effect of $\partial$ on $\ell$ shows that we have a subcomplex $C_{*}^{t}$; just as discussed in the introduction for any $k$ the degree-$k$ homologies $H_{k}^{t}(C_*)$ of these complexes yield a persistence module over the base field $\mathcal K$.  Typically $H_{k}^{t}(C_*)$ can be infinite-dimensional (and also may not satisfy the weaker descending chain condition which appears in \cite{CB}), so Theorem \ref{strthm} usually does not apply to these persistence modules. The exception to this is when the subgroup $\Gamma\leq\R$ used in the Novikov field $\Lambda=\Lambda^{\mathcal K,\Gamma}$ is the trivial group, in which case we just have $\Lambda=\mathcal K$ and the chain groups $C_k$ (and so also the homologies) are finite-dimensional over $\mathcal{K}$.  So when $\Gamma=\{0\}$, Theorem \ref{strthm} does apply to show that the persistence module $\{H_{k}^{t}(C_*)\}_{t\in\R}$ decomposes as a direct sum of interval modules $\mathcal{K}_I$; by definition the degree-$k$ part of the barcode of $C_*$ is then the multiset of intervals appearing in this direct sum decomposition.  We have:

\begin{theorem}\label{classical}
Assume that $\Gamma=\{0\}$ and let $(C_*,\partial,\ell)$ be a Floer-type complex over $\Lambda^{\mathcal{K},\{0\}}=\mathcal{K}$. For each $k\in \Z$ write $\partial_{k+1}\co C_{k+1}\to C_{k}$ for the degree-$(k+1)$ part of the boundary operator $\partial$, and write $Z_k=\ker \partial_k$, so that $\partial_{k+1}$ has image contained  in $Z_{k}$.  Let $\left((y_1,\ldots,y_{n}),(x_1,\ldots,x_m)\right)$ be a singular value decomposition for  $\partial_{k+1}\co C_{k+1}\to Z_k$. Then if $r=rank(\partial_{k+1})$, the degree-$k$ part of the barcode of $C_*$ consists precisely of: 
\begin{itemize}\item an interval $[\ell(x_i),\ell(y_i))$ for each $i\in \{1,\ldots,r\}$ such that $\ell(y_i)>\ell(x_i)$; and 
\item an interval $[\ell(x_i),\infty)$ for each $i\in \{r+1,\ldots,m\}$.
\end{itemize}
\end{theorem}

\begin{proof} As explained earlier, $\{H^{t}_{k}(C_*)\}_{t \in \mathbb R}$ is a pointwise-finite-dimensional persistence module. Therefore by Theorem \ref{strthm}, we have a normal form ${\textstyle \bigoplus_{\alpha}}  \mathcal{K}_{I_{\alpha}}$. Given a singular value decomposition $((y_1, ..., y_n),(x_1, ..., x_m))$ as in the hypothesis, we first claim that, for all $t \in \mathbb R$,
\begin{equation} \label{svd-fil-homo}
H_k^t(C_*) = span_{\mathcal K} \left\{ [x_i]  \,\bigg|\,
\begin{array}{rcl}
          \ell(x_i) \leq t < \ell(y_i) & \mbox{if}
         & i \in \{1, \ldots, r\} \\ \ell(x_i) \leq t  & \mbox{if} & i \in \{r+1, \ldots, m\} \\
         \end{array} \right\}.
\end{equation}
In fact, $(x_1, ..., x_m)$ is an orthogonal ordered basis for $\ker \partial_k$, so $\{x_i \,| \, \ell(x_i) \leq t\}$ is an orthogonal basis for $\ker(\partial_k|_{C_k^t})$. Meanwhile, by Lemma \ref{checkingrobust} when $\Gamma =\{0\}$ (so that $\nu$ vanishes on all nonzero elements of $\Lambda$), an element $x=\sum_{i=1}^{m}\lambda_i x_i$ lies in $\partial_{k+1}(C_{k+1}^{t})$ if and only if it holds both that $\lambda_i=0$ for all $i>r$ and that $\ell\left(\sum_{i=1}^{r}\lambda_i y_i\right)\leq t$, i.e. if and only if  $x\in span_{\mathcal{K}}\{x_i|1\leq i\leq r,\,\ell(y_i)\leq t\}$.
So we have bases $\{x_i\,|\,\ell(x_i)\leq t\}$ for $Z_k\cap C_{k}^{t}$ and $\{x_i|1\leq i\leq r,\,\ell(y_i)\leq t\}$ for $\partial_{k+1}(C_{k+1}^{t})$, from which the expression (\ref{svd-fil-homo}) for $H_{k}^{t}(C_*)$ immediately follows.

Write $V_t$ for the right hand side of (\ref{svd-fil-homo}). For $s \leq t$, the inclusion-induced map $\sigma_{st}: H_k^s(C_*) \rightarrow H_k^t(C_*)$ is identified with the map $\sigma_{st}: V_s \rightarrow V_t$ defined as follows, for any generator $[x_i]$ of $V_s$, 
\begin{equation} \label{trans-map}
\sigma_{st} ([x_i]) =     \left\{ \begin{array}{rcl}
         [x_i]& \mbox{if}
         & \ell(y_i) >t \,\,\,$or$\,\,\, i \in \{r+1, ..., s\}   \\ 
         0 & \mbox{if} & \ell(y_i) \leq t
                \end{array}\right. . 
\end{equation}
Clearly, this is a $\mathcal K$-linear homomorphism. It is easy to check that $\sigma_{ss} = I_{V_s}$ and for $s \leq t \leq u$, $\sigma_{su} = \sigma_{tu} \circ \sigma_{st}$. Therefore, $\mathbb V = \{V_{t}\}_{t \in \mathbb R}$ is a persistence module, which is (tautologically) {\it isomorphic}, in the sense of persistence modules, to $\{H_k^t(C_*)\}_{t \in \mathbb R}$. 

On the other hand, the normal form of $\mathbb V$ can be explicitly written out as follows: 
\begin{equation} \label{expli-nor-form}
\mathbb V \cong \bigoplus_{1\leq i\leq r} \mathcal{K}_{[\ell(x_i), \ell(y_i))} \oplus \bigoplus_{r+1 \leq j \leq m} \mathcal{K}_{[\ell(x_j), \infty)}.
\end{equation}

Indeed the indicated isomorphism of persistence modules can be obtained by simply mapping  $1\in(\mathcal{K}_{[\ell(x_i),\ell(y_i))]})_{t} =\mathcal{K}$ to the class $[x_i]$ for $t\in [\ell(x_i),\ell(y_i))$ and $i=1,\ldots,r$, and similarly for the $\mathcal{K}_{[\ell(x_i), \infty)}$ for $i>r$.
\end{proof}

Thus in the ``classical'' $\Gamma=\{0\}$ case the barcode can be read off directly from the filtration levels of the elements involved in a singular value decomposition; in particular, these filtration levels are independent of the choice of singular value decomposition, consistently with Theorem \ref{indep} below. For nontrivial $\Gamma$ there is clearly some amount of arbitrariness of the filtration levels of the elements of a singular value decomposition: if $\left((y_1,\ldots,y_n),(x_1,\ldots,x_m)\right)$ is a singular value decomposition, then $\left((T^{g_1}y_1,\ldots,T^{g_r}y_r,y_{r+1},\ldots,y_n),(T^{g_1}x_1,\ldots,T^{g_m}x_m)\right)$ is also a singular value decomposition for any $g_1,\ldots,g_m\in \Gamma$; based on Theorem \ref{classical} one would expect this to result in a change of the positions of each of the intervals in the barcode.  Note that this change moves the endpoints of the intervals but does not alter their lengths.  This suggests the following definition, related to the ideas of boundary depth and filtration spectrum:
 
\begin{dfn} \label{dfnbc} Let  $(C_*, \partial, \ell)$ be a Floer-type complex over $\Lambda=\Lambda^{\mathcal{K},\Gamma}$ and for each $k\in \Z$ write $\partial_k=\partial|_{C_k}$ and $Z_k=\ker \partial_k$. Given any $k\in \Z$ choose a singular value decomposition $((y_1, ..., y_n), (x_1, ..., x_m))$ for the $\Lambda$-linear map  $\partial_{k+1}\co C_{k+1}\to Z_k$ and let $r$ denote the rank of $\partial_{k+1}$. Then the degree-$k$ {\bf verbose barcode} of $(C_*,\partial,\ell)$ is the multiset of elements of $(\mathbb R/\Gamma) \times [0, \infty]$ consisting of \begin{itemize}
\item [(i)] a pair $(\ell(x_i)$ mod $\Gamma, \ell(y_i) - \ell(x_i))$ for $i = 1, ..., r$;
\item[(ii)] a pair $(\ell(x_i)$ mod $\Gamma, \infty)$ for $i = r+1, ...,m$.\end{itemize} 
The {\bf concise barcode} is the submultiset of the verbose barcode consisting of those elements whose second element is positive.\end{dfn}

Thus in the case that $\Gamma=\{0\}$ elements $[a,b)$ of the persistent homology barcode correspond according to Theorem \ref{classical} to elements $(a,b-a)$ of the concise barcode.  In general we think of an element $([a],L)$ of the (verbose or concise) barcode as corresponding to an interval with left endpoint $a$ and length $L$, with the understanding that the left endpoint is only specified up to the additive action of $\Gamma$.

Definition \ref{dfnbc} appears to depend on a choice of singular value decomposition, but we will see in Theorem \ref{indep} that different choices of singular value decompositions yield the same verbose (and hence also concise) barcodes. Of course in the case that $\Gamma=\{0\}$ this already follows from Theorem \ref{classical}; in the opposite extreme case that $\Gamma=\R$ (in which case the first coordinates of the pairs in the verbose and concise barcodes carry no information) it can easily be inferred from Theorem \ref{torsion}.

\begin{remark} Our reduction modulo $\Gamma$ in Definition \ref{dfnbc} (i) and (ii) is easily seen to be necessary if there is to be any hope of the verbose and concise barcodes being independent of the choice of singular value decomposition, for the reason indicated in the paragraph before Definition \ref{dfnbc}.  Namely, acting on the elements involved in the singular value decompositon by appropriate elements of $\Lambda$ could change the various quantities $\ell(x_i)$ involved in the barcode by  arbitrary elements of $\Gamma$.
\end{remark}
%indeed we could multiply the various $x_i$ and $y_i$ by elements $T^{g_i}$ for $g_i\in \Gamma$ without changing the fact that  $((y_1, ..., y_n), (x_1, ..., x_m))$ is a singular value decompositon.  While such an alteration will of course change the individual numbers $\ell(x_i)$ and $\ell(y_i)$, it will not change the pairs $ (\ell(x_i)$ mod $\Gamma, \ell(y_i) - \ell(x_i))$ that appear in Definition \ref{dfnbc}.
%\end{remark}

\begin{remark} In the spirit of Theorem \ref{algsvd}, we outline the procedure for computing the degree-$k$ verbose barcode for a Floer-type complex $(C_*,\partial,\ell)$:
\begin{itemize}
\item First, by applying the algorithm in Theorem \ref{algsvd} to $\partial_k\co C_k\to C_{k-1}$ or otherwise, obtain an orthogonal ordered basis $(w_1,\ldots,w_m)$ for $\ker\partial_k$.
\item Express $\partial_{k+1}\co C_{k+1}\to \ker\partial_k$ in terms of an orthogonal basis for $C_{k+1}$ and the basis $(w_1,\ldots,w_m)$ for $\ker\partial_k$, and apply Theorem \ref{algsvd} to obtain data $(v'_1,\ldots,v'_n)$ and $\mathcal{U}$ as in the statement of that theorem.
\item The degree-$k$ verbose barcode consists of one element $([\ell(Av'_i)],\ell(v'_i)-\ell(Av'_i))$ for each $i\in \mathcal{U}$, and one element $([a],\infty)$ for each $[a]$ lying in the multiset complement $\{[\ell(w_1)],\ldots,[\ell(w_m)]\}\setminus\{[\ell(Av'_i)]|i\in \mathcal{U}\}$.
\end{itemize}
\end{remark}
\subsection{Relation to spectral invariants} \label{specrel}

Following a construction that is found in \cite{Sc00}, \cite{Oh05} in the context of Hamiltonian Floer theory (and which is closely related to classical minimax-type arguments in Morse theory), we may describe the \textbf{spectral invariants} associated to a Floer-type complex $(C_*,\partial,\ell)$: letting $H_k(C_*)$ denote the degree-$k$ homology of $C_*$,  these invariants take the form of a map $\rho\co H_k(C_*)\to \R\cup\{-\infty\}$ defined by, for $\alpha\in H_k(C_*)$, \[ \rho(\alpha)=\inf\{\ell(c)\,|\,c\in C_k,\,[c]=\alpha\} \] (where $[c]$ denotes the homology class of $c$).  In a more general context the main result of \cite{Ush08} shows that the infimum in the definition of $\rho(\alpha)$ is always attained.

The spectral invariants are reflected in the concise barcode in the following way.

\begin{prop}\label{spectral} Let $\mathcal{B}_{C,k}$ denote the degree-$k$ part of the concise barcode of a Floer-type complex $(C_*,\partial,\ell)$, obtained from a singular value decomposition of $\partial_{k+1}\co C_{k+1}\to \ker\partial_k$. Then:
\begin{itemize} 
\item[(i)] There is a basis $\{\alpha_1,\ldots,\alpha_h\}$ for $H_k(C_*)$ over $\Lambda$ such that the submultiset of $\mathcal{B}_{C,k}$ consisting of elements with second coordinate equal to $\infty$ is equal to $\{([\rho(\alpha_1)],\infty),\ldots,([\rho(\alpha_h)],\infty)\}$ where for each $i$, $[\rho(\alpha_i)]$ denotes the reduction of $\rho(\alpha_i)$ modulo $\Gamma$. 
\item[(ii)] For any class $\alpha\in H_k(C_*)$, if we write $\alpha=\sum_{i=1}^{h}\lambda_i\alpha_i$ where $\lambda_i\in \Lambda$ and $\{\alpha_1,\ldots,\alpha_h\}$ is the basis from (i), then $\rho(\alpha)=\max_i(\rho(\alpha_i)-\nu(\lambda_i))$.  In particular, if $\alpha\neq 0$, then the concise barcode $\mathcal{B}_{C,k}$ contains an element of the form $([\rho(\alpha)],\infty)$.
\end{itemize}
\end{prop}

\begin{proof}
Let $((y_1,\ldots,y_m),(x_1,\ldots,x_n))$ be a singular value decomposition of $\partial_{k+1}\co C_{k+1}\to \ker\partial_k$.  In particular, if $r$ is the rank of $\partial_{k+1}$, then $span_{\Lambda}\{x_{r+1},\ldots,x_m\}$ is an orthogonal complement to ${\rm Im} \partial_{k+1}$.
Hence the classes $\alpha_i=[x_{r+i}]$ (for $1\leq i\leq m-r$) form a basis for $H_k(C_*)$, and the dimension of the $H_k(C_*)$ over $\Lambda$ is $h=m-r$.  By definition,   the submultiset of $\mathcal{B}_{C,k}$ consisting of elements with second coordinate equal to $\infty$ is $\{([\ell(x_{r+1})],\infty),\ldots,([\ell(x_{m})],\infty)\}$, so both part (i) and the first sentence of part (ii) of the proposition will follow if we show that, for any $\lambda_1,\ldots,\lambda_{m-r}\in \Lambda$ we have \begin{equation} \label{rholambda}
\rho\left(\sum_{i=1}^{m-r}\lambda_i\alpha_i\right)=\max_i(\ell(x_{r+i})-\nu(\lambda_i))
 \end{equation} (indeed the special case of (\ref{rholambda}) in which $\lambda_i=\delta_{ij}$ implies that $\rho(\alpha_j)=\ell(x_{r+j})$).
 
To prove (\ref{rholambda}), simply note that any class $\alpha=\sum_i\lambda_i\alpha_i\in H_k(C_*)$ is represented by the chain $\sum_{i}\lambda_ix_{r+i}$, and that the general representative of $\alpha$ is given by $x=y+\sum_{i}\lambda_ix_{r+i}$ for $y\in {\rm Im}\partial_{k+1}$.  So since $\{x_{r+1},\ldots,x_{m}\}$ is an orthogonal basis for an orthogonal complement to ${\rm Im}\partial_{k+1}$ it follows that \[ \ell(x)=\max\left\{\ell(y),\ell\left(\sum_i\lambda_ix_{r+i}\right)\right\}\geq \ell\left(\sum_i\lambda_ix_{r+i}\right)=\max_i(\ell(x_{r+i})-\nu(\lambda_i)),\] with equality if $y=0$.  Thus the minimal value of $\ell$ on any representative $x$ of $\sum_{i=1}^{m-r}\lambda_i\alpha_i$ is equal to $\max_i(\ell(x_{r+i})-\nu(\lambda_i))$, proving (\ref{rholambda}).

As noted earlier, (\ref{rholambda}) directly implies (i) and the first sentence of (ii).  But then the second sentence of (ii) also follows immediately, since each $\lambda\in \Lambda\setminus\{0\}$ has $\nu(\lambda)\in \Lambda$, and so if $\alpha=\sum_{i}\lambda_i\alpha_i\neq 0$ it follows from (\ref{rholambda}) that $\rho(\alpha)$ is congruent mod $\Gamma$ to one of the $\rho(\alpha_i)$.
\end{proof}

\subsection{Duality and coefficient extension for barcodes} \label{dualcodes}

Given a Floer-type complex $(C_*,\partial,\ell)$ over $\Lambda=\Lambda^{\mathcal{K},\Gamma}$ one obtains a dual complex $(C_{*}^{\vee},\delta,\ell^*)$ by taking $C_{k}^{\vee}$ to be the dual over $\Lambda$ of $C_{-k}$, $\delta\co C^{\vee}_{k}\to C^{\vee}_{k-1}$ to be the adjoint of $\partial\co C_{-k+1}\to C_{-k}$ and defining $\ell^{*}$ as in 
Section \ref{dualss}.  The following can be seen as a generalization both of \cite[Corollary 1.6]{U10} and of \cite[Proposition 2.4]{SMV}

\begin{prop}
For all $k$, denote by $\tilde{\mathcal{B}}_{C,k}$ the degree-$k$ verbose barcode of $(C_*,\partial,\ell)$. Then the degree-$k$ verbose barcode of $(C_{*}^{\vee},\delta,\ell^{*})$ is given by \begin{equation}\label{dualize} \tilde{\mathcal{B}}_{C^{\vee},k}=\{([-a],\infty)\,|\,([a],\infty)\in\tilde{\mathcal{B}}_{C,-k}\}\cup\{([-a-L],L)\,|\,L<\infty\mbox{ and }([a],L)\in\tilde{\mathcal{B}}_{C,-k-1}\}. \end{equation}
\end{prop}

\begin{proof}
 Suppose that $r$ is the rank of $\partial_{-k}\co C_{-k}\to C_{-k-1}$, $s$ is the rank of $\partial_{-k+1}\co C_{-k+1}\to C_{-k}$, and $t\geq r$ is the dimension of the kernel of $\partial_{-k-1}\co C_{-k-1}\to C_{-k-2}$. It is straightforward (by using the Gram-Schmidt process in Theorem \ref{GSp} if necessary) to modify a singular value decomposition  $\left((y_1,\ldots,y_n),(x_1,\ldots,x_m)\right)$ of $\partial_{-k}\co C_{-k}\to C_{-k-1}$ so that it has the additional properties that:  \begin{itemize} \item[(i)] $(x_1,\ldots,x_{t})$ is an orthogonal ordered basis for $\ker\partial_{-k-1}$, so that in particular $\left((y_1,\ldots,y_n),(x_1,\ldots,x_t)\right)$ is a singular value decomposition for $\partial_{-k}\co C_{-k}\to\ker\partial_{-k-1}$.
\item[(ii)] $(y_{n-s+1},\ldots,y_n)$ is an orthogonal ordered basis for $\mathrm{Im}\partial_{-k+1}$, so that the elements $([a],L)$ of $\tilde{\mathcal{B}}_{C,-k}$ having $L=\infty$ are precisely the $([\ell(y_i)],\infty)$ for $i\in\{r+1,\ldots,n-s\}$.
\end{itemize}

By Proposition \ref{dualor}, a singular value decomposition for $\delta_{k+1}\co C^{\vee}_{k+1}\to C^{\vee}_{k}$ is given by $\left((x_{1}^{*},\ldots,x_{m}^{*}),(y_{1}^{*},\ldots,y_{n}^{*})\right)$, where the $x_{i}^{*}$ and $y_{j}^{*}$ form dual bases for the bases $(x_1,\ldots,x_m)$ and $(y_1,\ldots,y_n)$, respectively.  Moreover by (ii) above, the kernel of $\delta_{k}\co C^{\vee}_{k}\to C^{\vee}_{k-1}$ (i.e., the annihilator of the image of $\partial_{-k+1}$) is precisely the span of $y_{1}^{*},\ldots,y_{n-s}^{*}$, and so  $\left((x_{1}^{*},\ldots,x_{m}^{*}),(y_{1}^{*},\ldots,y_{n-s}^{*})\right)$
is a singular value decomposition for $\delta_{k+1}\co C^{\vee}_{k+1}\to \ker\delta_k$.  Since by (\ref{flip}) we have $\ell^{*}(x_{i}^{*})=-\ell(x_i)$ and $\ell^{*}(y^{*}_{i})=-\ell(y_i)$ it follows that
 \[ \tilde{\mathcal{B}}_{C^{\vee},k}=\{([-\ell(y_i)],\ell(y_i)-\ell(x_i))\,|\,i=1,\ldots,r\}\cup\{([-\ell(y_i)],\infty)\,|\,i=r+1,\ldots,n-s\} \] which precisely equals the right hand side of (\ref{dualize}).
 
\end{proof}

The effect on the verbose barcode of extending the coefficient field of a Floer-type complex by enlarging the value group $\Gamma$ is even easier to work out, given our earlier results.

\begin{prop}  Let $(C_*,\partial,\ell)$ be a Floer-type complex over $\Lambda=\Lambda^{\mathcal{K},\Gamma}$, let $\Gamma'\leq \R$ be a subgroup containing $\Gamma$, and consider the Floer-type complex $(C'_{*},\partial\otimes 1,\ell')$ over $\Lambda^{\mathcal{K},\Gamma'}$ given by letting $C'_k=C_k\otimes_{\Lambda}\Lambda^{\mathcal{K},\Gamma'}$ and defining $\ell'$ as in Section \ref{coefsect}. Let $\tilde{\mathcal{B}}_{C,k}$ be the verbose barcode of $(C_*,\partial,\ell)$ in degree $k$ and let $\pi\co \R/\Gamma\to\R/\Gamma'$ be the projection. Then the verbose barcode of $(C'_{*},\partial\otimes 1,\ell')$ in degree $k$ is \[ \{(\pi([a]),L)\,|\,([a],L)\in\tilde{\mathcal{B}}_{C,k}\}. \]
\end{prop}

\begin{proof}
This follows directly from Proposition \ref{coextsvd} and the definitions.
\end{proof}

\section{Classification theorems} \label{classsect}

In the spirit of the structure theorem (Theorem \ref{strthm}) for pointwise finite-dimensional persistence modules, we will use the verbose and concise barcodes to classify Floer-type complexes up to filtered chain isomorphism and filtered homotopy equivalence. Specifically, we will prove the following two key theorems, stated earlier in the introduction. 

\begin{theorema} Two Floer-type complexes $(C_*, \partial_C, \ell_C)$ and $(D_*, \partial_D, \ell_D)$ are filtered chain isomorphic to each other if and only if they have identical verbose barcodes in all degrees.\end{theorema}

\begin{theoremb} Two Floer-type complexes $(C_*, \partial_C, \ell_C)$ and $(D_*, \partial_D, \ell_D)$ are filtered  homotopy equivalent to each other if and only if they have identical concise barcodes in all degrees.
\end{theoremb}

\subsection{Classification up to filtered isomorphism}

We will assume the following important theorem first, and then the proof of Theorem A will follow quickly.

\begin{theorem} \label{indep} For any $k\in \Z$, the degree-$k$ verbose barcode of any Floer-type complex is independent of the choice of singular value decomposition for $\partial_{k+1}\co C_{k+1}\to Z_k$.\end{theorem}

\begin{proof}[Proof of Theorem A]
 On the one hand, a filtered chain isomorphism $C_*\to D_*$ maps a singular value decomposition for $(\partial_C)_{k+1}\co C_{k+1}\to \ker(\partial_C)_k$ to a singular value decomposition for $(\partial_D)_{k+1}\co D_{k+1}\to \ker(\partial_D)_k$, while keeping all filtration levels the same.    Therefore, the ``only if'' part of Theorem A is a direct consequence of Theorem \ref{indep}.

To prove the ``if'' part of Theorem A we begin by introducing some notation that will also be useful to us later.
Given a collection of Floer-type complexes $\mathcal{C}_{\alpha}=(C_{\alpha *},\partial_{\alpha},\ell_{\alpha})$ we define $\oplus_{\alpha}\mathcal{C}_{\alpha}$ to be the triple $(\oplus_{\alpha}C_{\alpha *},\oplus_{\alpha}\partial_{\alpha},\tilde{\ell})$ where $\tilde{\ell}\left((c_{\alpha})\right)=\max_{\alpha}\ell_{\alpha}(c_{\alpha})$.  Provided that, for each $k\in \Z$, only finitely many of the $C_{\alpha k}$ are nontrivial, $\oplus_{\alpha}\mathcal{C}_{\alpha}$ is also a Floer-type complex.

\begin{dfn}\label{elemdfn} Fix $\Gamma\leq \R$ and the associated Novikov field $\Lambda=\Lambda^{\mathcal{K},\Gamma}$.  For $a\in \R$, $L\in [0,\infty]$, and $k\in \Z$ define the \textbf{elementary Floer-type complex} $\mathcal{E}(a,L,k)$ to be the Floer-type complex $(E_*,\partial_E,\ell_E)$ given as follows: \begin{itemize}
\item If $L=\infty$ then $E_m=\left\{\begin{array}{ll}\Lambda & m=k \\ 0 & {\rm otherwise}\end{array}\right.$, $\partial_E=0$, and, for $\lambda\in E_m=\Lambda$, $\ell(\lambda)=a-\nu(\lambda)$.
\item If $L\in [0,\infty)$, then $E_{k}$ is the one-dimensional $\Lambda$-vector space generated by a symbol $x$, $E_{k+1}$ is the one-dimensional $\Lambda$-vector space generated by a symbol $y$, and $E_m=\{0\}$ for $m\notin \{k,k+1\}$. Also, $\partial_E\co E_*\to E_*$ is defined by $\partial_E(\lambda x+\mu y)=\mu x$, and $\ell_E(\lambda x+\mu y)=\max\{a-\nu(\lambda),(a+L)-\nu(\mu)\}$.  
\end{itemize}
\end{dfn}

\begin{remark}\label{redundantE}
If $b-a\in \Gamma$, then there is a filtered chain isomorphism $\mathcal{E}(a,L,k)\to \mathcal{E}(b,L,k)$ given by scalar multiplication by the element $T^{b-a}\in \Lambda$.
\end{remark}

\begin{prop} \label{ournormalform}  Let $(C_*,\partial,\ell)$ be a Floer-type complex and denote by $\tilde{\mathcal{B}}_{C,k}$ the degree-$k$ verbose barcode of $(C_*,\partial, \ell)$.  Then there is a filtered chain isomorphism \[ (C_*,\partial,\ell)\cong \bigoplus_{k\in\Z}\bigoplus_{([a],L)\in \tilde{\mathcal{B}}_{C,k}}\mathcal{E}(a,L,k) \] (where for each $([a],L)\in \tilde{\mathcal{B}}_{C,k}$ we choose an arbitrary representative $a\in \R$ of the coset $[a]\in \R/\Gamma$).
\end{prop}

\begin{proof}[Proof of Proposition \ref{ournormalform}]
For each $k$ let \[ \left((y_{1}^{k},\ldots,y_{r_{k}}^{k},\ldots,y_{r_k+m_{k+1}}^{k}),(x_{1}^{k},\ldots,x_{m_k}^{k})\right) \] be an arbitrary singular value decomposition for $(\partial_C)_{k+1}\co C_{k+1}\to \ker(\partial_C)_k$, where $r_k$ is the rank of $(\partial_{C})_{k+1}$ and $m_k = \dim(\ker(\partial_C)_k)$ for each degree $k \in \Z$.  We will first modify these singular value decompositions for various $k$ to be related to each other in a convenient way. Specifically, since $(x_{1}^{k+1},\ldots,x_{m_{k+1}}^{k+1})$ is an orthogonal ordered basis for $\ker(\partial_C)_{k+1}$, the tuple \[  \left((y_{1}^{k},\ldots,y_{r_{k}}^{k},x_{1}^{k+1},\ldots,x_{m_{k+1}}^{k+1}),(x_{1}^{k},\ldots,x_{m_k}^{k})\right) \] is also a singular value decomposition for $(\partial_C)_{k+1}\co C_{k+1}\to \ker(\partial_C)_k$.  So letting \[ (a_{i}^{k},L_{i}^{k})=\left\{\begin{array}{ll} (\ell(x_{i}^{k}),\ell(y_{i}^{k})-\ell(x_{i}^{k})) & 1\leq i\leq r_k\\ (\ell(x_{i}^{k}),\infty) & r_{k}+1\leq i\leq m_k \end{array}\right. , \] we have $\mathcal{B}_{C,k}=\{([a_{i}^{k}],L_{i}^{k})|1\leq i\leq m_k\}$ and the proposition states that $(C_*,\partial,\ell)$ is filtered chain isomorphic to $\oplus_k\oplus_{i=1}^{m_k}\mathcal{E}(a_{i}^{k},L_{i}^{k},k)$.  Now for each $i$ and $k$ there is an obvious embedding $\phi_{i,k}\co \mathcal{E}(a_{i}^{k},L_{i}^{k},k)\to C_*$ defined by: 
\begin{itemize}
\item when $L_i^k = \infty$, $\phi_{i,k}(\lambda) = \lambda x_{i}^k$;
\item when $L_i^k< \infty$, $\phi_{i,k}(\lambda x+\mu y)=\lambda x_{i}^{k}+\mu y_{i}^{k}$.
\end{itemize}
%For instance, for degree $k$, 
%\[ {\rm Im}(\phi_{i,k}) = span_{\Lambda} \{y^k_1, ..., y^k_{r_k}, x_1^k, ..., x_{m_k}^k \} \]
%and for degree $k+1$,
%\[ {\rm Im}(\phi_{i,k+1}) = span_{\Lambda} \{y^{k+1}_1, ..., y^{k+1}_{r_{k+1}}, x_1^{k+1}, ..., x_{m_{k+1}}^{k+1} \}. \]
%Note that, by construction, $(y^k_1, ..., y^k_{r_k}, x_1^{k+1}, ..., x_{m_{k+1}}^{k+1})$ spans $C_{k+1}$. 
From the definition of the filtration and boundary operator on $\mathcal{E}(a_{i}^{k},L_{i}^{k},k)$ this embedding is a chain map which exactly preserves filtration levels.  Then \[ \oplus_{i,k}\phi_{i,k}\co \oplus_{i,k}\mathcal{E}(a_{i}^{k},L_{i}^{k},k)\to C_* \] is likewise a chain map. Finally, for each $k$, the fact that $(y_{1}^{k},\ldots,y_{r_{k}}^{k},x_{1}^{k+1},\ldots,x_{m_{k+1}}^{k+1})$ is an orthogonal ordered basis for $C_{k+1}$ readily implies that $\oplus_{i,k}\phi_{i,k}$ is in fact a filtered chain isomorphism.
\end{proof}

Since, by Remark \ref{redundantE}, the filtered isomorphism type of $\mathcal{E}(a,L,k)$ only depends on $[a],L,k$, and since quite generally filtered chain isomorphisms $\Phi_{\alpha}\co\mathcal{C}_{\alpha}\to \mathcal{D}_{\alpha}$ between Floer-type complexes induce a filtered chain isomorphism $\oplus_{\alpha}\co \oplus_{\alpha}\mathcal{C}_{\alpha}\to \oplus_{\alpha}\mathcal{D}_{\alpha}$, Proposition \ref{ournormalform} shows that the filtered chain isomorphism type of a Floer-type complex is determined by its verbose barcode, proving the ``if part'' of Theorem A.
\end{proof}

The remainder of this subsection is directed toward the proof of Theorem \ref{indep}.  We will repeatedly apply the following criterion for testing whether a subspace is an orthogonal complement of a given subspace.

\begin{lemma}\label{or&pr} Let $(C,\ell)$ be an orthogonalizable $\Lambda$-space, and let $U,U',V\leq C$ be subspaces such that $U$ is an orthogonal complement to $V$ and $\dim U'=\dim U$.  Consider the projection $\pi_U: C \rightarrow U$ associated to the direct sum decomposition $C=U\oplus V$. Then $U'$ is an orthogonal complement of $V$ if and only if $\ell(\pi_U x)=\ell(x)$ for all $x \in U'$.\end{lemma}

\begin{proof}
Assume that $U'$ is an orthogonal complement to $V$.  Then for $x\in U'$, we of course have 
\[ \begin{array}{l}
x=\pi_Ux + (x-\pi_Ux)
\end{array} \]
where $\pi_Ux\in U$ and  $x-\pi_Ux\in V$. Because $U$ and $V$ are orthogonal, it follows that $\ell(x)=\max\{\ell(\pi_Ux),\ell(x-\pi_Ux)\}$. In particular, 
\begin{equation} \label{1}
\ell(x) \geq \ell(\pi_U x).
\end{equation}
Meanwhile since 
\[ \begin{array}{l}
\pi_Ux = x-(x-\pi_Ux)
\end{array} \]
where $x\in U'$, $x-\pi_Ux\in V$, and $U'$ and $V$ are orthogonal, we have $\ell(\pi_Ux)=\max\{\ell(x),\ell(x-\pi_Ux)\}$. In particular, $\ell(\pi_U x) \geq \ell(x)$. Combined with (\ref{1}), this shows $\ell(x)=\ell(\pi_Ux)$.

Conversely, suppose that $\ell(\pi_U x)=\ell(x)$ for all $x\in U'$.  To show that $U'$ is an orthogonal complement to $V$ we just need to show that $U'$ and $V$ are orthogonal, that is, for any $x\in U'$ and $v\in V$ we have $\ell(x+v)=\max\{\ell(x),\ell(v)\}$ (indeed if we show this, then by Lemma \ref{bsorprop} (i)  $U'$ and $V$ will have trivial intersection and  so  dimensional considerations will imply that $C=U'\oplus V$).  Now write $x\in U'$ as 
\[ \begin{array}{l}
x=\pi_U x + (x-\pi_U x)
\end{array} \]
where $\pi_Ux\in U$ and $x-\pi_Ux\in V$.  Because $U$ and $V$ are orthogonal, our assumption shows that $\ell(x)=\ell(\pi_Ux)\geq \ell(x-\pi_Ux)$. Now 
\[ \begin{array}{l}
x+v=\pi_Ux +\left(v+(x-\pi_Ux)\right)
\end{array} \]
where $\pi_Ux$ in $U$ and $v+(x-\pi_Ux)\in V$. Again, $U$ and $V$ are orthogonal, so we have 
\begin{align*} 
\ell(x+v) &= \max\{\ell(\pi_Ux),\ell\left(v+(x-\pi_Ux)\right)\} \\
& = \max\left\{\ell(x),\ell\left(v+(x-\pi_Ux)\right)\right\}. 
\end{align*}
Now if $\ell(v)>\ell(x)$ then $\ell(x+v) = \ell(v)=\max\{\ell(x),\ell(v)\}$, as desired.  On the other hand if $\ell(v)\leq\ell(x)$ then $\ell\left(v+(x-\pi_Ux)\right)\leq \max\{\ell(v),\ell(x-\pi_Ux)\}\leq \ell(x)$, and so $\ell(x+v)=\ell(x)=\max\{\ell(x),\ell(v)\}$.  So in any case we indeed have $\ell(x+v)=\max\{\ell(x),\ell(v)\}$ for any $x\in U',v\in V$, and so $U'$ and $V$ are orthogonal. \end{proof}

\begin{notation} \label{notation} Let $\left((y_1, \ldots, y_n),(x_1,\ldots, x_m)\right)$ be a singular value decomposition for a two-term Floer-type complex $(C_1 \xrightarrow{\partial} C_0)$, and let $r$ be the rank of $\partial$. Denote $k_1,\ldots,k_p\in\{1,\ldots,r\}$ to be the increasing finite sequence of integers defined by the property that $k_1=1$ and, for $i\in\{1,\ldots,p\}$, either $\beta_{k_i}(\partial)=\beta_{k_i+1}(\partial)=\cdots=\beta_r(\partial)$ (in which case $p=i$) or else $\beta_{k_i}(\partial)=\cdots=\beta_{k_{i+1}-1}(\partial)>\beta_{k_{i+1}}(\partial)$.  Also let $k_{p+1}=r+1$. We emphasize that the numbers $k_i$ are independent of choice of singular value decomposition (since the $\beta_k(\partial)$ are likewise independent thereof, see Definition \ref{barlength}). \end{notation}

The proof of Theorem \ref{indep} inductively uses the following lemma, which is an application of Lemma \ref{or&pr}.

\begin{lemma}\label{induc1} Let $\left((y_1,\ldots,y_n),(x_1,\ldots,x_m)\right)$ be a singular value decomposition for $(C_1 \xrightarrow{\partial} C_0)$ and $r=rank(\partial)$, and let $k_1,\ldots,k_{p+1}$ be the integers in  Notation \ref{notation}. Let $i\in\{1,\ldots,p\}$, and suppose that $V,W\leq {\rm Im} \partial\leq C_0$ obey:
\begin{itemize}
\item[(i)] $\dim V=k_{i}-1$, $V$ is $\delta$-robust for all $\delta<\beta_{k_{i-1}}(\partial)$, and $V$ is orthogonal to  $span_{\Lambda}\{x_{k_i},\ldots,x_{m}\}$.  (If $i=1$ these conditions mean $V=\{0\}$.)
\item[(ii)] $\dim W = k_{i+1}-k_{i}$, $W$ is orthogonal to $V$, and $V\oplus W$ is $\delta$-robust for all $\delta < \beta_{k_i}(\partial)$.\end{itemize}
Now let $X=span_{\Lambda}\{x_{k_i},\ldots,x_{k_{i+1}-1}\}$ and $X'=span_{\Lambda}\{x_{k_{i+1}},\ldots,x_m\}$. Then $V\oplus W$ is orthogonal to $X'$, and there is an isomorphism of filtered vector spaces $W\cong X$.\end{lemma}

\begin{proof}  Since $V$ is orthogonal to $X\oplus X'$ and $X$ is orthogonal to $X'$,  by Lemma \ref{bsorprop}, we have an orthogonal direct sum decomposition $C_0=X\oplus(X'\oplus V)$.  We will first show that the projection $\pi_X: C_0 \rightarrow X$ associated to this direct sum decomposition has the property that $\pi_X|_W$ exactly preserves filtration levels.

Let $w\in W$, and write $w=v+x+x'$ where $v\in V$, $x\in X$, and $x'\in X'$, so our goal is to show that $\ell(w)=\ell(x)$. Of course this is trivial if $w=0$, so assume $w\neq 0$. Now \begin{equation}\ell(w)=\max\{\ell(x+x'),\ell(v)\} \nonumber\end{equation}  since $V$ is orthogonal to $X\oplus X'$.  Meanwhile since $x+x'=w-v$ and $V$ and $W$ are orthogonal we have $\ell(x+x')=\max\{\ell(v),\ell(w)\}\geq \ell(v)$.  So $\ell(w)=\ell(x+x')=\max\{\ell(x),\ell(x')\}$.  (In particular $x$ and $x'$ are not both zero.)
Now expand $w-v=x+x'$ in terms of the basis $\{x_j\}$ as \[ w-v=\sum_{j=k_i}^{r}\lambda_j x_j. \]  The fact that we can take the sum to start at $k_i$ follows from the definitions of $X$ and $X'$, and the sum terminates at $r$ because $w-v\in V\oplus W\leq {\rm Im} \partial$.  Then $\ell(w-v)=\max\{\ell(\lambda_j x_j)|j\in\{k_i,\ldots,r\}\}$.  By Lemma \ref{checkingrobust}, the infimal filtration level of any $\tilde{y}\in C_1$ such that $\partial \tilde{y}=x+x'$ is attained by $\tilde{y}=y+y'$ where $y={\textstyle \sum_{j=k_i}^{k_{i+1}-1}}\lambda_j y_j$ and $y'={\textstyle \sum_{j=k_{i+1}}^{r}}\lambda_j y_j$; by the assumption that $V\oplus W$ is $\delta$-robust for all $\delta<\beta_{k_i}(\partial)$, we will have  \begin{equation}\ell(y+y')\geq\ell(w-v)+\beta_{k_i}(\partial)=\ell(x+x')+\beta_{k_i}(\partial).\nonumber\end{equation}

Thus by the orthogonality of the bases $\{x_j\}$ and $\{y_j\}$, \begin{equation}\label{yyxx} 
\beta_{k_i}(\partial)\leq \ell(y+y')-\ell(x+x') = \max\{\ell(y),\ell(y')\}-\max\{\ell(x),\ell(x')\}.  \end{equation}

Now if we choose $j_0$ to maximize the quantity $\ell(\lambda_j y_j)$ over all $j\in\{k_{i+1},\ldots,r\}$ we will have \[ \ell(y')=\ell(\lambda_{j_0}y_{j_0})=\ell(\lambda_{j_0}x_{j_0})+\beta_{j_0}(\partial)\leq \ell(x')+\beta_{j_0}(\partial). \] So \[ \ell(y')-\max\{\ell(x),\ell(x')\}\leq \ell(y')-\ell(x')\leq \beta_{j_0}(\partial)<\beta_{k_i}(\partial) \] since $j_0\geq k_{i+1}$.  Thus in view of (\ref{yyxx}) we must have $\ell(y)>\ell(y')$ and so by Proposition \ref{strict} $\ell(y+y') =\ell(y)$.  Similarly, choose $i_0 \in \{k_i, \ldots, k_{i+1}-1\}$ to maximize the quantity $\ell(\lambda_jx_j)$, so that $\ell(x) = \ell(\lambda_{i_0} x_{i_0})$. Then 
\[ 
\ell(y) - \ell(x) \geq \ell(\lambda_{i_0} y_{i_0}) - \ell(\lambda_{i_0} x_{i_0}) = \beta_{i_0}(\partial). \]
Symmetrically, choose $i_1 \in \{k_i, \ldots, k_{i+1}-1\}$ to maximize the quantity $\ell(y) ={\textstyle \sum_{k_i}^{k_{i+1}-1}} \lambda_i y_i$, that is $\ell(y) = \ell(\lambda_{i_1} y_{i_1})$. Then 
\[ 
\ell(y) - \ell(x) \leq \ell(\lambda_{i_1} y_{i_1}) - \ell(\lambda_{i_1} x_{i_1}) = \beta_{i_1}(\partial).
 \]
Because $\beta_{k_i}(\partial) = \cdots = \beta_{k_{i+1}-1}(\partial)$ and $i_0,i_1\in\{k_i,\ldots,k_{i+1}-1\}$, the above inequalities imply that $ \beta_{i_0}(\partial) =  \beta_{i_i}(\partial) = \beta_{k_i}(\partial)$. Thus we necessarily have $\ell(y)-\ell(x)=\beta_{k_i}(\partial)$. So we cannot have $\ell(x')>\ell(x)$, since if this were the case then $\ell(y+y')-\ell(x+x')=\ell(y)-\max\{\ell(x),\ell(x')\}$ would be strictly smaller than $\beta_{k_i}(\partial)$, a contradiction to condition (ii).  Thus $\ell(x)\geq \ell(x')$.  So since we have seen that $\ell(w)=\max\{\ell(x),\ell(x')\}$ this proves that $\ell(w)=\ell(x)$.

Thus the projection $\pi_X: C_0 \rightarrow X$ associated to the direct sum decomposition $X\oplus(V\oplus X')$ has $\ell(\pi_Xw)=\ell(w)$ for all $w\in W$, and in particular it is injective because $0$ is the only element with filtration level $-\infty$. So dimensional considerations prove the last statement of the lemma.  By Lemma \ref{or&pr}, this also implies that $W$ is an orthogonal complement to $V\oplus X'$. Since $X'$ is orthogonal to $V$ and $V\oplus X'$ is orthogonal to $W$ it follows from Lemma \ref{bsorprop} (ii) that $V\oplus W$ is orthogonal to $X'$, which is precisely the remaining conclusion of the lemma. \end{proof}

\begin{cor} \label{switch} Let $\left((z_1,\ldots,z_n),(w_1,\ldots,w_m)\right)$ and $\left((y_1,\ldots,y_n),(x_1,\ldots,x_m)\right))$ be two singular value decompositions for $(C_1 \xrightarrow{\partial} C_0)$. Then for each $i\in\{1,\ldots,p\}$ there is a commutative diagram 
\begin{equation}\label{diag} \xymatrix{ span_{\Lambda}\{z_{k_i},\ldots,z_{k_{i+1}-1}\}\ar[r] \ar[d]^{\partial} & span_{\Lambda}\{y_{k_i},\ldots,y_{k_{i+1}-1}\} \ar[d]^{\partial} \\ span_{\Lambda}\{w_{k_i},\ldots,w_{k_{i+1}-1}\} \ar[r] & span_{\Lambda}\{x_{k_i},\ldots,x_{k_{i+1}-1}\} } \nonumber\end{equation} where the horizontal arrows are isomorphisms of filtered vector spaces. \end{cor}
 
\begin{proof} Consider the following ascending sequence of subspaces of $\rm {Im}(\partial)$:
\[ \begin{array}{l}
\{0\} = V_0 \leq V_1 \leq V_2 \leq \ldots \leq V_p = {\rm{Im}}(\partial)
\end{array} \]
where $V_i = span\{w_1, \ldots, w_{k_{i+1}-1}\}$. Each $V_i$ is $\delta$-robust for all $\delta< \beta_{k_i}(\partial)$ by Lemma \ref{checkingrobust}.  Also let $W_i=span_{\Lambda}\{w_{k_i},\ldots,w_{k_{i+1}-1}\}$, so we have an orthogonal direct sum decomposition $V_i=V_{i-1}\oplus W_i$.  

We claim by induction on $i$ that
$V_i$ is orthogonal to $span_{\Lambda}\{x_{k_{i+1}},\ldots,x_m\}$.  Indeed for $i=0$ this is trivial, and assuming that it holds for the value $i-1$ then applying Lemma \ref{induc1} with $V=V_{i-1}$ and $W=W_i$ proves the claim for the value $i$. 
Given this fact, for any $i$ we may again apply Lemma \ref{induc1} to obtain a filtered isomorphism $W_i\to span_{\Lambda}\{x_{k_i},\ldots,x_{k_{i+1}-1}\}$, which serves as the bottom arrow in the diagram in the statement of the Corollary.

 Since the side arrows and the bottom arrow are all linear isomorphisms, there is a unique top arrow that makes the diagram commute.  Moreover the bottom arrow exactly preserves filtration, and the side arrows both decrease the filtration levels of all nonzero elements by \emph{exactly} $\beta_{k_i}(\partial)$, so it follows that the top arrow is an isomorphism of filtered vector spaces as well. \end{proof}

\begin{proof}[Proof of Theorem \ref{indep}] Let $\left((z_1,\ldots,z_n),(w_1,\ldots,w_m)\right)$, $\left((y_1,\ldots,y_n),(x_1,\ldots,x_m)\right)$ be two singular value decompositions. Both of  $span_{\Lambda}\{w_{r+1},\ldots,w_{m}\}$ and $span_{\Lambda}\{x_{r+1},\ldots,x_m\}$ are orthogonal complements to ${\rm Im} \partial$, where $r=rank(\partial)$, so they are filtered isomorphic by Lemma \ref{or&pr} and so they have the same filtration spectra by Proposition \ref{multi}. Meanwhile, the subspaces $span_{\Lambda}\{w_{k_i},\ldots,w_{k_{i+1}-1}\}$ and $span_{\Lambda}\{x_{k_i},\ldots,x_{k_{i+1}-1}\}$ are filtered isomorphic for each $i \in \{1, ..., p\}$ by Corollary \ref{switch}, so they likewise have the same filtration spectra. The conclusion now follows immediately from the description of verbose barcode, using Theorem \ref{genebd}. \end{proof}

\subsection{Classification up to filtered homotopy equivalence}

Now we move on to the classification of the filtered chain homotopy equivalence class of a Floer-type complex. First, we will prove the ``if part'', which is the easier direction. 

\begin{prop} \label{if-part-THB} For any Floer-type complex $(C_*,\partial_C,\ell_C)$, let $\mathcal{B}_{C,k}$ denote the degree-$k$ concise barcode of $(C_*,\partial_C,\ell_C)$. For each $([a],L)\in\mathcal{B}_{C,k}$, choose a representative $a$ of the coset $[a]\in\R/\Gamma$. Then $(C_*,\partial_C,\ell_C)$ is filtered homotopy equivalent to \[ \bigoplus_{k\in \Z}\bigoplus_{([a],L)\in\mathcal{B}_{C,k}}\mathcal{E}(a,L,k). \]
\end{prop}

\begin{proof}
For each $k$ let $\tilde{\mathcal{B}}_{C,k}$ denote the degree-$k$ verbose barcode of $(C_*,\partial_{C},\ell_C)$ and $\mathcal{B}_{C,k}$ the degree-$k$ concise barcode, so $\mathcal{B}_{C,k}=\{([a],L)\in\tilde{\mathcal{B}}_{C,k}\,|\,L>0\}$  

By Proposition \ref{ournormalform}, if for each $([a],L)\in \tilde{\mathcal{B}}_{C,k}$ we choose a representative $a$ of the coset $[a]\in\R/\Gamma$,  $(C_*,\partial_C,\ell_C)$ is filtered chain isomorphic to \begin{equation}\label{sep} \left(\bigoplus_{k}\bigoplus_{([a],L)\in \mathcal{B}_{C,k}}\mathcal{E}(a,L,k)\right)\oplus \left(\bigoplus_{k}\bigoplus_{([a],0)\in \tilde{\mathcal{B}}_{C,k}\setminus\mathcal{B}_{C,k}}\mathcal{E}(a,0,k)\right). \end{equation}
Recall the definition of $\mathcal{E}(a,0,k)$ as the triple $(E_*,\partial_E,\ell_E)$ where $E_*$ is spanned over $\Lambda$ by elements $y\in E_{k+1}$ and $x\in E_k$ with $\partial_Ey =x$ and $\ell_E(y)=\ell_E(x)=a$.  If we define $K\co E_*\to E_{*+1}$ to be the $\Lambda$-linear map defined by $Kx=-y$ and $K|_{E_m}=0$ for $m\neq k$, we see that $\ell_E(Ke)\leq \ell_E(e)$ for all $e\in E_*$, that $(\partial_E K + K\partial_E)x=-\partial_Ey=-x$, and that $(\partial_EK+K\partial_Ey)=Kx=-y$.  So $K$ defines a filtered chain homotopy between $0$ and the identity, in view of which $\mathcal{E}(a,0,k)$ is filtered homotopy equivalent to the zero chain complex.  Since a direct sum of filtered homotopy equivalences is a filtered homotopy  equivalence, the Floer-type complex in (\ref{sep}) (and hence also $(C_*,\partial_C,\ell_C)$) is filtered homotopy equivalent to $ \bigoplus_{k\in \Z}\bigoplus_{([a],L)\in\mathcal{B}_{C,k}}\mathcal{E}(a,L,k)$.
\end{proof}

Recalling from Remark \ref{redundantE} that the filtered isomorphism type of $\mathcal{E}(a,L,k)$ only depends on $([a],L,k)$, so that up to filtered chain isomorphism  $ \oplus_{k\in \Z}\oplus_{([a],L)\in\mathcal{B}_{C,k}}\mathcal{E}(a,L,k)$ is independent of the choices $a$ of representatives of the cosets $[a]$, the ``if'' part of Theorem B follows directly from Proposition \ref{if-part-THB}.

\subsubsection{Mapping cylinders}\label{cylsect}

We review here the standard homological algebra construction of the mapping cylinder of a chain map between two chain complexes; the special case where the chain map is a homotopy equivalence will be used both in the proof of the ``only if'' part of Theorem B and in the proof of the stability theorem.

For a chain complex $(C_*,\partial_C)$ we use $(C[1]_*,\partial_C)$ to denote the chain complex obtained by shifting the degree of $C_*$ by $1$: $C[1]_k=C_{k-1}$, with  boundary operator given tautologically by the boundary operator of $C_*$.

\begin{dfn} \label{dfnmpcyl0}  Let $(C_*,\partial_C)$ and $(D_*,\partial_D)$ be two chain complexes over an arbitrary ring, and let $\Phi\co C_*\to D_*$ be a chain map.  The \textbf{mapping cylinder} of $\Phi$ is the chain complex $(Cyl(\Phi)_*,\partial_{cyl})$ defined by $Cyl(\Phi)_* = C_* \oplus D_* \oplus C[1]_*$ and, for $(c,d,e)\in Cyl(\Phi)_*$, $\partial_{cyl}(c,d,e) = (\partial_Cc-e,\partial_D d+\Phi e, -\partial_Ce)$. Thus, in block form, \[ \partial_{cyl} = \left(\begin{array}{ccc}\partial_C & 0 & -I_{C_*} \\ 0 &\partial_D & \Phi \\ 0 & 0 & -\partial_C   \end{array}\right).
\]\end{dfn}

It is a routine matter to check that $\partial_{cyl}^{2}=0$, so $(Cyl(\Phi)_*,\partial_{cyl})$ as defined above is indeed a chain complex.  

For the moment we will work at the level of chain complexes, not of filtered chain complexes, the reason being that we will later use Lemma \ref{cylhtopy} below under a variety of different kinds of assumptions about filtration levels.

\begin{dfn}\label{htdef} Given two chain complexes $(C_*,\partial_C)$ and $(D_*,\partial_D)$, a homotopy equivalence between $(C_*,\partial_C)$ and $(D_*,\partial_D)$ is a quadruple $(\Phi,\Psi,K_C,K_D)$ such that $K_C: C_* \to C_{*+1}$, $K_D: D_* \to D_{*+1}$ are linear maps shifting degree by $+1$ and $\Phi\co C_*\to D_*$, $\Psi\co D_*\to C_*$ are chain maps, obeying $\Psi\Phi-I_{C_*}=\partial_CK_C+K_C\partial_C$ and $\Phi\Psi-I_{D_*}=\partial_DK_D+K_D\partial_D$.
\end{dfn}

(In particular our convention is to consider the homotopies part of the data of a homotopy equivalence.)

\begin{lemma}\label{cylhtopy} Let $(\Phi,\Psi,K_C,K_D)$ be a homotopy equivalence between $(C_*,\partial_C)$ and $(D_*,\partial_D)$.  Then:
\begin{itemize} \item[(i)] Suppose that $i_D\co D_*\to Cyl(\Phi)_*$ is the inclusion, $\alpha\co Cyl(\Phi)_*\to D_*$ is defined by $\alpha(c,d,e)=\Phi c+d$, and $K\co Cyl(\Phi)_*\to Cyl(\Phi)_{*+1}$ is defined by $K(c,d,e)=(0,0,c)$. Then the quadruple $(i_D,\alpha,0,K)$ is a homotopy equivalence between $(D_*,\partial_D)$ and $(Cyl(\Phi)_*,\partial_{cyl})$.
\item[(ii)] Suppose that $i_C\co C_*\to Cyl(\Phi)_*$ is the inclusion, $\beta\co Cyl(\Phi)_*\to C_*$ is defined by $\beta(c,d,e)=c+\Psi d+K_Ce$, and $L\co Cyl(\Phi)_*\to Cyl(\Phi)_{*+1}$ is defined by \[ L(c,d,e)=(-K_Cc,K_D(\Phi c+d),c-\Psi(\Phi c+d)). \]
Then the quadruple $(i_C,\beta,0,L)$ is a homotopy equivalence between $(C_*,\partial_C)$ and $(Cyl(\Phi)_*,\partial_{cyl})$.\end{itemize}
\end{lemma}

\begin{proof} The proof requires only a series of routine computations to show that $i_D,\alpha,i_C,\beta$ are all chain maps and that the various chain homotopy equations hold.  We will do only the most nontrivial of these, namely the proof of the identity 
$i_C\beta-I_{Cyl(\Phi)_*}=\partial_{Cyl}L+L\partial_{Cyl}$, leaving the rest to the reader.  We see that, for $(c,d,e)\in Cyl(\Phi)_*$, 
\[ (i_C\beta-I_{Cyl(\Phi)_*})(c,d,e)=\left(\Psi d+K_Ce,-d,-e\right)\] while \begin{align*} \partial_{cyl}L(c,d,e)&= \partial_{cyl}\left(-K_Cc,K_D(\Phi c+d),c-\Psi(\Phi c+d)\right) 
\\ &= \left(-\partial_CK_Cc-c+\Psi\Phi c+\Psi d,\partial_DK_D(\Phi c+d)+\Phi c-\Phi\Psi(\Phi c+d),-\partial_C c+\partial_C\Psi(\Phi c+d)\right)
\\ &= \left(K_C\partial_C c+\Psi d,-K_D\partial_D\Phi c+(\partial_DK_D-\Phi\Psi)d, -\partial_C c+\partial_C\Psi(\Phi c+d)\right)\end{align*}
where we have used the facts that $\Psi\Phi-I_{C_*}=\partial_CK_C+K_C\partial_C$ and $\Phi\Psi-I_{D_*}=\partial_DK_D+K_D\partial_D$.  Meanwhile \begin{align*}
L\partial_{cyl}(c,d,e)&=L\left(\partial_Cc-e,\partial_Dd+\Phi e,-\partial_Ce\right)
\\& = \left(-K_C\partial_Cc+K_Ce,K_D(\Phi\partial_Cc+\partial_Dd),\partial_C c-e-\Psi(\Phi\partial_C c+\partial_Dd)\right).
\end{align*}
So \begin{align*}
\left(\partial_{cyl}L+L\partial_{cyl}\right)&(c,d,e) = \left(\Psi d+K_Ce,(\partial_DK_D-\Phi\Psi+K_D\partial_D)d,-e\right)
\\&= (\Psi d+K_Ce,-d,-e) = (i_C\beta-I_{Cyl(\Phi)_*})(c,d,e)
\end{align*} where in the first equation we have used the fact that $\Phi$ and $\Psi$ are chain maps and in the second equation we have again used that $\Phi\Psi-I_{D_*}=\partial_DK_D+K_D\partial_D$.  So indeed $i_C\beta-I_{Cyl(\Phi)_*}=\partial_{Cyl}L+L\partial_{Cyl}$; as mentioned earlier the remaining identities are easier to prove and so are left to the reader.
\end{proof}

We can now fill in the last part of our proofs of the main classification results.
\begin{proof}[Proof of Theorem B]
One implication has already been proven in Proposition \ref{if-part-THB}.  For the other direction, let $(C_*,\partial_C,\ell_C)$ and $(D_*,\partial_D,\ell_D)$ be two filtered homotopy equivalent Floer-type complexes.  Thus there is a homotopy equivalence $(\Phi,\Psi,K_C,K_D)$ satisfying the additional properties that, for all $c\in C_*$ and $d\in D_*$, we have \begin{equation}\label{filtsame} \ell_D(\Phi c)\leq \ell_C(c)\quad \ell_C(\Psi d)\leq \ell_D(d) \quad \ell_C(K_Cc)\leq \ell_C(c)\quad \ell_D(K_Dd)\leq \ell_D(d). \end{equation}  Now form the mapping cylinder $(Cyl(\Phi)_*,\partial_{cyl})$ as described earlier, and define $\ell_{cyl}\co Cyl(\Phi)_{*}\to \R\cup\{-\infty\}$ by \[ \ell_{cyl}(c,d,e)=\max\{\ell_C(c),\ell_D(d),\ell_C(e)\} \] It is easy to see that $(Cyl(\Phi)_{*},\partial_{cyl},\ell_{cyl})$ is then a Floer-type complex.\footnote{For comparison with what we do later it is worth noting that the fact that $\ell_{cyl}(\partial_{cyl} x)\leq \ell_{cyl}(x)$ for all $x$ is crucially dependent on the first inequality of (\ref{filtsame}).} Now  $(Cyl(\Phi)_*,\partial_{cyl},\ell_{cyl})$ has a concise barcode in each degree; we will show that this concise barcode is both the same as that of $(C_*,\partial_C,\ell_C)$ and the same as that of $(D_*,\partial_D,\ell_D)$, which will suffice to prove the result.

Using the notation of Lemma \ref{cylhtopy}, since $\alpha\co Cyl(\Phi)_*\to D_*$ is a chain map with $\alpha i_D=I_{D_*}$, we have a direct sum decomposition \emph{of chain complexes} $Cyl(\Phi)_*=D_*\oplus\ker\alpha$.  We claim that $D_*$ and $\ker\alpha$ are orthogonal (with respect to the filtration function $\ell_{cyl}$).  Now \[ \ker\alpha=\{(c,d,e)\in Cyl(\Phi)_*\,|\,d=-\Phi c\}=\left\{(c,-\Phi c,e)\,|\,(c,e)\in C_*\oplus C[1]_*\right\}. \]  Since $D_*$ is  an orthogonal complement to $C_*\oplus C[1]_{*}$ in $Cyl(\Phi)_*$, and since in each grading $k$ the dimensions of the degree-$k$ part of $\ker\alpha$ and of $C_k\oplus C[1]_k$ are the same, by Lemma \ref{or&pr} in order to show that $\ker\alpha$ is orthogonal to   $D_*$ it suffices to show that, writing $\pi\co Cyl(\Phi)_*\to C_*\oplus C[1]_*$ for the orthogonal projection $(c,d,e)\mapsto (c,e)$, one has $\ell_{cyl}(\pi x)=\ell_{cyl}(x)$ for all $x\in \ker\alpha$.  But any $x\in \ker\alpha$ has $x=(c,-\Phi c,e)$ for some $(c,e)\in C_*\oplus C[1]_*$, and $\ell_D(- \Phi c)\leq \ell_C(c)$, so we indeed have $\ell_{cyl}(\pi x)=\max\{\ell_C(c),\ell_C(e)\}=\ell_{cyl}(x)$.  So indeed $D_*$ and $\ker\alpha$ are orthogonal. 

In view of the orthogonal direct sum decomposition of chain complexes $Cyl(\Phi)_*=D_*\oplus \ker\alpha$, for every degree $k$ we can obtain a singular value decomposition for $(\partial_{cyl})_{k+1}\co Cyl(\Phi)_{k+1}\to \ker(\partial_{cyl})_k$ by simply combining singular value decompositions for the restrictions of $(\partial_{cyl})_{k+1}$ to $D_{k+1}$ and to $(\ker\alpha)_{k+1}$.  Then by Theorem \ref{indep}, the verbose barcode of $Cyl(\Phi)_{*}$ is the union of the verbose barcodes of $D_*$ and of $\ker\alpha$.

To describe the latter of these, we will show presently that every element in $\ker(\partial_{cyl}|_{\ker\alpha})$ is the boundary of an element having the same filtration level. In fact, for any $x \in \ker(\partial_{cyl}|_{\ker\alpha})$, the equation $i_D\alpha-I_{Cyl(\Phi)_{*}}=\partial_{cyl}K+K\partial_{cyl}$ shows that $x=\partial_{cyl}(-Kx)$. Moreover, 
\[ \ell_{cyl}(x) = \ell_{cyl} (\partial_{cyl}(-Kx)) \leq \ell_{cyl}(-Kx) \leq \ell_{cyl}(x), \]
where the last inequality comes from the formula for $K$ in Lemma \ref{cylhtopy}. Therefore $\ell_{cyl}(x) = \ell_{cyl}(-Kx)$. 

Consequently, every element $([a],s)$ of the verbose barcode of $\ker\alpha$  has $s=0$ (or, said differently, the concise barcode of $\ker\alpha$ is empty in every degree).  Thus the verbose barcode of $Cyl(\Phi)_{*}$ may be obtained from the verbose barcode of $D_*$ by adding elements with second coordinate equal to zero; consequently the concise barcodes of $Cyl(\Phi)_{*}$ and of $D_*$ are equal.

The proof that the concise barcodes of $Cyl(\Phi)_{*}$ and $C_*$ are likewise equal is very similar.  We have a direct sum decomposition of chain complexes $Cyl(\Phi)_{*}=C_*\oplus \ker\beta$, where $\ker\beta = \{(-\Psi d-K_Ce,d,e)|(d,e)\in D_{*}\oplus C[1]_{*}\}$.  Let $\pi'\co Cyl(\Phi)_{*}\to D_{*}\oplus C[1]_{*}$ be the projection associated to the orthogonal direct sum decomposition $Cyl(\Phi)_{*}=C_*\oplus(D_*\oplus C[1]_{*})$. The inequalities (\ref{filtsame}) imply that $\ell_{cyl}(\pi'x)=\ell_{cyl}(x)$ for all $x\in \ker\beta$.  Hence by applying Lemma \ref{or&pr} degree-by-degree we see that $Cyl(\Phi)_{*}=C_*\oplus \ker\beta$ is an \emph{orthogonal} direct sum decomposition of chain complexes, and hence that in any degree $k$ the verbose barcode of $Cyl(\Phi)_{*}$ is the union of the degree-$k$ verbose barcodes of $C_*$ and of $\ker\beta$.  Any cycle $x$ in $\ker\beta$ obeys $x=-\partial_{cyl} Lx$, where the formula for $L$ (together with (\ref{filtsame})) shows that $\ell_{cyl}(-Lx)\leq\ell_{cyl}(x)$.  While $Lx$ might not be an element of $\ker\beta$, the orthogonality of $C_{*}$ and $\ker\beta$ together with 
Lemma \ref{checkingrobust} allow one to find $y\in\ker\beta$ with $\partial y=x$ and $\ell_{cyl}(y)\leq \ell_{cyl}(-Lx)\leq \ell_{cyl}(x)$.  Just as above, this proves that all elements $([a],s)$ of the verbose barcode of $\ker\beta$ have second coordinate $s$ equal to zero, and so once again the concise barcode of $Cyl(\Phi)_{*}$ coincides with that of $C_*$. 
\end{proof}

\section{The Stability theorem}

The Stability Theorem (or a closely related statement sometimes called the  Isometry Theorem) is the one of the most important theorems in the theory of persistent homology. It successfully transfers the problem of relating the filtered homology groups constructed by different methods
(e.g., different Morse functions on a given manifold) to a combinatorial problem based on the associated barcodes. The result was originally established for the persistence modules associated to ``tame'' functions on topological spaces in \cite{CEH}; since then a variety of different proofs and generalizations have appeared (see e.g. \cite{CCGGO}, \cite{BL}), and it now generally understood as an algebraic statement in the abstract context of persistence modules. In this section, we will introduce  some basic notations and definitions in order to state our version of the stability theorem, which unlike previous versions applies to Floer-type complexes over general Novikov fields $\Lambda^{\mathcal{K},\Gamma}$.  In the special case that $\Gamma=\{0\}$ the result follows from recent more algebraic formulations of the stability theorem like that in \cite{BL}, though we would say that our proof is conceptually rather different.

The following is an abstraction of the filtration-theoretic properties satisfied by the ``continuation maps''
in Hamiltonian Floer theory that relate the Floer-type complexes associated to different Hamiltonian functions; namely such maps are homotopy equivalences which shift the filtration by a certain amount which is related to an appropriate distance (the Hofer distance) between the Hamiltonians  (see \cite[Propositions 5.1, 5.3 and 6.1]{U13}).

\begin{dfn} \label{dfnqusie} Let $(C_*,\partial_C,\ell_C)$ and $(D_*,\partial_D,\ell_D)$ be two Floer-type complexes over $\Lambda$, and $\delta\geq 0$.  A $\delta$-{\bf quasiequivalence} between $C_*$ and $D_*$ is a quadruple $(\Phi,\Psi,K_C,K_D)$ where:
\begin{itemize}
\item[(i)] $(\Phi,\Psi,K_C,K_D)$ is a homotopy equivalence (see Definition \ref{htdef}).
\item[(ii)] For all $c\in C_*$ and $d\in D_*$ we have \begin{equation}\label{qeshift} \ell_D(\Phi c)\leq \ell_C(c)+\delta \quad  \ell_C(\Psi d)\leq \ell_D(d)+\delta \quad  \ell_C(K_Cc)\leq \ell_C(c)+2\delta \quad \ell_D(K_Dd)\leq \ell_D(d)+2\delta. \end{equation}
\end{itemize}
The {\bf quasiequivalence distance} between   $(C_*,\partial_C,\ell_C)$ and $(D_*,\partial_D,\ell_D)$ is then defined to be \[ d_Q((C_*,\partial_C,\ell_C),(D_*,\partial_D,\ell_D))=\inf\left\{\delta\geq 0\left|\begin{array}{cc}\mbox{There exists a $\delta$-quasiequivalence between }\\(C_*,\partial_C,\ell_C)\mbox{ and }(D_*,\partial_D,\ell_D)\end{array}\right.\right\}. \]\end{dfn}

Of course, $(C_*,\partial_C,\ell_C)$ and $(D_*,\partial_D,\ell_D)$ are said to be $\delta$-quasiequivalent provided that there exists a $\delta$-quasiequivalence between them.  Note that a $0$-quasiequivalence is the same thing as a filtered homotopy equivalence.

\begin{remark}\label{dqtri}
It is easy to see that if $(C_*,\partial_C,\ell_C)$ and $(D_*,\partial_D,\ell_D)$ are $\delta_0$-quasiequivalent and $(D_*,\partial_D,\ell_D)$ and $(E_*,\partial_E,\ell_E)$ are $\delta_1$-quasiequivalent then $(C_*,\partial_C,\ell_C)$ and $(E_*,\partial_E,\ell_E)$ are $(\delta_0+\delta_1)$-quasiequivalent.  Thus $d_Q$ satisfies the triangle inequality.  In particular, if 
$(C_*,\partial_C,\ell_C)$ and $(D_*,\partial_D,\ell_D)$ are $\delta$-quasiequivalent then $(C_*,\partial_C,\ell_C)$ is also $\delta$-quasiequivalent to any Floer-type complex that is filtered homotopy equivalent to $(D_*,\partial_D,\ell_D)$.  
\end{remark}

\begin{ex} \label{morqe} Take $(F_1, g_1)$ and $(F_2, g_2)$ to be two  Morse functions together with suitably generic Riemannian metrics on a closed manifold $X$. Let $\delta = \|F_1-F_2\|_{L^{\infty}}$. Then it is well-known (and can be deduced from constructions in \cite{Sc93}, for instance) that the associated Morse chain complexes, over the ground field $\mathcal K=\Lambda^{\mathcal{K},\{0\}}$, $CM_*(X; F_1, g_1)$ and $CM_*(X; F_2, g_2)$ are $\delta$-quasiequivalent.\end{ex}

\begin{ex} \label{hfqe} Take $(H_1, J_1)$ and $(H_2, J_2)$ to be two generic Hamiltonian functions together with compatible almost complex structures on a closed symplectic manifold $(M, \omega)$. Then, as is recalled in greater detail at the start of Section \ref{hamsect}, one has Hamiltonian Floer complexes $CF_*(M; H_1, J_1)$ and $CF_*(M; H_2, J_2)$ over the Novikov field $\Lambda^{\mathcal K, \Gamma}$ where  $\Gamma\leq\R$ is defined in (\ref{hamgamma}). Define 
\[ \begin{array}{l}
E_{+}(H) = \int_0^1 \max_M H(t, \cdot) dt \,\,\,\,\,\,\,$and$\,\,\,\,\,\,\,\,E_{-}(H) = - \int_0^1 \min_M H(t, \cdot) dt \
\end{array} \]
and let $\delta = \max\{E_{+}(H_2 - H_1), E_{-}(H_2 - H_1)\}$.  Then $(CF_*(M; H_1, J_1))$ and $(CF_*(M; H_2, J_2))$ are $\delta$-quasiequivalent. The maps in the corresponding quadruple $(\Phi, \Psi, K_1, K_2)$ are constructed by counting solutions of certain partial differential equations (see \cite[Chapter 11]{AD}). \end{ex}

\begin{remark} One could more generally define, for $\delta_1,\delta_2\in \R$, a $(\delta_1,\delta_2)$-quasiequivalence by replacing (\ref{qeshift}) by the conditions $\ell_D(\Phi c)\leq \ell_C(c)+\delta_1$, $\ell_C(\Psi d)\leq \ell_D(d)+\delta_2$, $\ell_C(K_Cc)\leq \ell_C(c)+\delta_1+\delta_2$, and $\ell_D(K_Dd)\leq \ell_D(d)+\delta_1+\delta_2$.  (So in this language a $\delta$-quasiequivalence is the same thing as a $(\delta,\delta)$-quasiequivalence.)  Then in Example \ref{hfqe} one has the somewhat sharper statement that
$(CF_*(M; H_1, J_1))$ and $(CF_*(M; H_2, J_2))$ are $(E_+(H_2-H_1),E_-(H_2-H_1))$-quasiequivalent.  However since adding a suitable constant to $H_1$ has the effect of reducing to the case that $E_+(H_2-H_1)$ and $E_-(H_2-H_1)$ are equal to each other while changing the filtration on the Floer complex (and hence changing the barcode) by a simple uniform shift, for ease of exposition we will restrict attention to the more symmetric case of a $\delta$-quasiequivalence.
\end{remark}

\begin{remark} We will explain in Appendix \ref{intsect} that quasiequivalence is closely related with the notion of {\it interleaving} of persistent homology from \cite{BL}. In particular, the quasiequivalence distance $d_Q$ is equal to a natural chain-level version of the interleaving distance from \cite{BL}.\end{remark}

Our first step toward the stability theorem will be a continuity result for 
the quantities $\beta_k$ from Definition \ref{barlength}.  Recall that for $i\in\Z$ the degree-$i$ part of the (verbose or concise) barcode of $(C_*,\partial_C,\ell_C)$ is obtained from a singular value decomposition of the map $(\partial_C)_{i+1}\co C_{i+1}\to \ker(\partial_C)_i$. 

\begin{lemma}\label{qepre} Let $(\Phi,\Psi,K_C,K_D)$ be a $\delta$-quasiequivalence and let $\eta\geq 2\delta$.  If $V\leq \ker(\partial_C)_i$ is $\eta$-robust then $\Phi|_V$ is injective and $\Phi(V)$ is $(\eta-2\delta)$-robust. 
\end{lemma}

\begin{proof}
If $v\in V$ and $\Phi v=0$ then \[ v=v-\Psi\Phi v=\partial_C(-K_Cv) \] where $\ell_C(-K_Cv)\leq \ell_C(v)+2\delta$; by the definition of $\eta$-robustness  (see Definition \ref{dfn-robust}) this implies that  $v=0$ since $\eta\geq 2\delta$.  So indeed $\Phi|_V$ is injective.

Now suppose that $0\neq w=\Phi v\in \Phi(V)$ with $\partial_D y=w$.  Then \[ \partial_C\Psi y= \Psi \partial_D y = \Psi\Phi v=v+\partial_CK_Cv \] (where we've used the fact that $V\leq \ker\partial_C$).  So $v=\partial_C(\Psi y-K_Cv)$.  By the definition of $\eta$-robustness we have $\ell_C(\Psi y-K_Cv)> \ell_C(v)+\eta$.  Since $\ell_C(K_C v)\leq \ell_C(v)+2\delta\leq \ell_C(v)+\eta$ this implies that \[ \ell_C(\Psi y)> \ell_C(v)+\eta.\] But   
$\ell_D(y)\geq \ell_C(\Psi y)-\delta$, and $\ell_D(w)=\ell_D(\Phi v)\leq \ell_C(v)+\delta$, which combined with the displayed inequality above shows that $\ell_D(y)>\ell_D(w)+(\eta-2\delta)$.  Since $w$ was an arbitrary nonzero element of $\Phi(V)$ this proves that $\Phi(V)$ is $(\eta-2\delta)$-robust.
\end{proof}

\begin{cor} \label{qegenebd} Suppose that $(C_*, \partial_C, \ell_C)$ and $(D_*, \partial_D, \ell_D)$ are $\delta$-quasiequivalent.  Then for all $i\in \Z$ and $k \in \mathbb N$, we have $|\beta_k((\partial_C)_{i+1}) - \beta_k((\partial_D)_{i+1})| \leq 2 \delta$.\end{cor}

\begin{proof}
By definition $\beta_k((\partial_C)_{i+1})$ is the supremal $\eta\geq 0$ such that there exists a $k$-dimensional $\eta$-robust subspace of $\mathrm{Im}((\partial_D)_{i+1})$, or is zero if no such subspace exists for any $\eta$.  If $\beta_k((\partial_C)_{i+1})>2\delta$, then given $\ep>0$ there is a $k$-dimensional subspace $V\leq {\rm Im} (\partial_C)_{i+1}$ which is $(\beta_k((\partial_C)_{i+1})-\ep)$-robust, and then (for small enough $\ep$) Lemma \ref{qepre} shows that $\Phi(V)\leq \mathrm{Im}((\partial_D)_{i+1})$ is $k$-dimensional and $(\beta_k((\partial_C)_{i+1})-\ep-2\delta)$-robust.  Since this construction applies for all sufficiently small $\ep>0$ it follows that \begin{equation}\label{onesidedbeta} \beta_k((\partial_D)_{i+1})\geq \beta_k((\partial_C)_{i+1})-2\delta\end{equation} provided that $\beta_k((\partial_C)_{i+1})>2\delta$.  But of course if $\beta_k((\partial_C)_{i+1})\leq 2\delta$ then (\ref{onesidedbeta}) still holds for the trivial reason that $\beta_k((\partial_D)_{i+1})$ is by definition nonnegative.  So (\ref{onesidedbeta}) holds in any case.  But this argument may equally well be applied with the roles of the complexes $(C_*,\partial_C,\ell_C)$ and $(D_*,\partial_D,\ell_D)$ reversed (as the relation of $\delta$-quasiequivalence is symmetric), yielding $\beta_k((\partial_C)_{i+1})\geq \beta_k((\partial_D)_{i+1})-2\delta$, which together with (\ref{onesidedbeta}) directly implies the corollary.
\end{proof}

In order to state our stability theorem we must explain the bottleneck distance, which is a measurement of the distance between two barcodes in common use at least since \cite{CEH}.  First we will define some notions related to matchings between multisets, similar to what can be found in, e.g., \cite{CSGO}.  We initially express this in rather general terms in order to make clear that our notion of a partial matching can be identified with corresponding notions found elsewhere in the literature. Recall below that a pseudometric space is a generalization of a metric space in which two distinct points are allowed to be a distance zero away from each other, and an extended pseudometric space is a generalization of a pseudometric space in which the distance between two points is allowed to take the value $\infty$.

\begin{dfn} Let $(X,d)$ be an extended pseudometric space equipped with a ``length function'' $\lambda\co X\to [0,\infty]$, and let $\mathcal{S}$ and $\mathcal{T}$ be two multisets of elements of $X$.  
\begin{itemize} 
\item A \textbf{partial matching} between $\mathcal{S}$ and $\mathcal{T}$ is a triple $\mathfrak{m}=(\mathcal{S}_{short},\mathcal{T}_{short},\sigma)$ where $\mathcal{S}_{short}$ and $\mathcal{T}_{short}$ are submultisets of $\mathcal{S}$ and $\mathcal{T}$, respectively, and $\sigma\co\mathcal{S}\setminus\mathcal{S}_{short}\to \mathcal{T}\setminus\mathcal{T}_{short}$ is a bijection.  (The elements of $\mathcal{S}_{short}$ and $\mathcal{T}_{short}$ will sometimes be called ``unmatched.'')
\item For $\delta\in [0,\infty]$, a \textbf{$\delta$-matching} between $\mathcal{S}$ and $\mathcal{T}$ is a partial matching $(\mathcal{S}_{short},\mathcal{T}_{short},\sigma)$ such that for all $x\in\mathcal{S}_{short}\cup\mathcal{T}_{short}$ we have $\lambda(x)\leq \delta$
and for all $x$ in $\mathcal{S}\setminus\mathcal{S}_{short}$ we have $d(\sigma(x),x)\leq \delta$.

\item If $\mathfrak{m}$ is a partial matching between $\mathcal{S}$ and $\mathcal{T}$, the \textbf{defect} of $\mathfrak{m}$ is \[ \delta(\mathfrak{m})=\inf\{\delta\geq 0\,|\,\mathfrak{m}\mbox{ is a $\delta$-matching}\}. \]
\end{itemize}
\end{dfn}

%\begin{remark}\label{pushfwd}
%If $\mathcal{S}$ and $\mathcal{T}$ are multisets of elements of $X$ and $f\co X\to Y$ is any function then we have obvious multisets $f_{*}\mathcal{S}=\{f(s)\,|\,s\in\mathcal{S}\}$ and $f_{*}\mathcal{T}=\{f(t)\,|\,t\in\mathcal{T}\}$ of elements of $Y$ (allowing repetitions when $f$ is not injective, so that these multisets have the same cardinality as $\mathcal{S}$ and $\mathcal{T}$, respectively).  Then a partial matching $\mathfrak{m}$ between $\mathcal{S}$ and $\mathcal{T}$ induces in obvious fashion a partial matching $f_*\mathfrak{m}$ between $f_*\mathcal{S}$ and $f_*\mathcal{T}$.   Likewise, a partial matching $\mathfrak{n}$ between $f_*\mathcal{S}$ and $f_*\mathcal{T}$ induces a partial matching $f^*\mathfrak{n}$ between $\mathcal{S}$ and $\mathcal{T}$; we have $f_*f^*\mathfrak{n}=\mathfrak{n}$ and $f^*f_*\mathfrak{m}=\mathfrak{m}$.
%\end{remark}

\begin{ex}
Let $\mathcal{H}=\{(x,y)\in (-\infty,\infty]^2\,|\,x< y\}$ with extended metric $d_{\mathcal{H}}((a,b),(c,d))=\max\{|c-a|,|d-b|\}$ and $\lambda_{\mathcal{H}}((a,b))=\frac{b-a}{2}$.  Then our notion of a $\delta$-matching between multisets of elements of $\mathcal{H}$ is readily verified to be the same as that used in \cite[Section 4]{CSGO} or \cite[Section 3.2]{BL}.
\end{ex}

\begin{ex}\label{G0match} Consider $\R\times (0,\infty]$ with the extended metric $d((a,L),(a', L'))=\max\{|a-a'|,|(a+L)-(a'+L')|\}$ and the length function $\lambda(a,L)=L/2$.  
Then the bijection $f\co \R\times (0,\infty]\to\mathcal{H}$ defined by $f(a,L)=(a,a+L)$ pulls back $d_{\mathcal{H}}$ and $\lambda_{\mathcal{H}}$ from the previous example to $d$ and  $\lambda$, respectively, so giving a $\delta$-matching $\frak{m}$ between multisets of elements of $\R\times (0,\infty]$ is equivalent to giving a $\delta$-matching $f_*\frak{m}$ between the corresponding multisets of elements of $\mathcal{H}$.
\end{ex}

\begin{ex}\label{Gmatch} Our main concern will be $\delta$-matchings between concise barcodes of Floer-type complexes, which are by definition multisets of elements of $(\R/\Gamma)\times (0,\infty]$ for a subgroup $\Gamma\leq \R$.  For this purpose we use the length function $\lambda\co (\R/\Gamma)\times (0,\infty]\to \R$ defined by $\lambda([a],L)=\frac{L}{2}$ and the extended pseudometric \[ d\left(([a],L),([a'],L')\right)=\inf_{g\in \Gamma}\max\{|a+g-a'|,|(a+g+L)-(a'+L')|\}. \]  In the case that $\Gamma=\{0\}$ this evidently reduces to Example \ref{G0match}.
\end{ex}

For convenience, we rephrase the definition of a $\delta$-matching between concise barcodes:

\begin{dfn} \label{dfnmatching} Consider two concise barcodes $\mathcal{S}$ and $\mathcal{T}$ (viewed as multisets of elements of $(\mathbb R/\Gamma) \times (0, \infty]$). A {\bf $\delta$-matching} between $\mathcal{S}$ and $\mathcal{T}$ consists of the following data:
\begin{itemize} \item[(i)] submultisets $\mathcal{S}_{short}$ and $\mathcal{T}_{short}$ such that the second coordinate $L$ of every element $([a],L)\in \mathcal{S}_{short}\cup\mathcal{T}_{short}$ obeys $L\leq 2\delta$.
\item[(ii)] A bijection $\sigma\co \mathcal{S}\setminus\mathcal{S}_{short}\to \mathcal{T}\setminus\mathcal{T}_{short}$ such that, for each $([a],L)\in \mathcal{S}\setminus\mathcal{S}_{short}$ (where $a\in \R$, $L\in [0,\infty]$) we have $\sigma([a],L)=([a'],L')$ where for all $\ep>0$ the representative $a'$ of the coset $[a']\in \R/\Gamma$ can be chosen such that both $|a'-a|\leq \delta+\ep$ and either $L=L'=\infty$ or $|(a'+L')-(a+L)|\leq \delta+\ep$.
\end{itemize} \end{dfn}

 It follows from the discussion in Example \ref{G0match} that our definition agrees  in the case that $\Gamma=\{0\}$ (via the map $(a,L)\mapsto (a,a+L)$) to the definitions in, for example, \cite{CSGO} or \cite{BL}.

\begin{dfn} \label{dfnbnd} If $\mathcal{S}$ and $\mathcal{T}$ are two multisets of elements of $(\R/\Gamma)\times (0,\infty]$ then the {\bf bottleneck distance} between $\mathcal{S}$ and $\mathcal{T}$ is \[ d_{B}(\mathcal{S},\mathcal{T})=\inf\{\delta\geq 0\,|\,\mbox{There exists a $\delta$-matching between }\mathcal{S}\mbox{ and }\mathcal{T}\}. \]
\end{dfn}

Our constructions associate to a Floer-type complex a concise barcode \emph{for every $k\in \Z$}, so the appropriate notion of distance for this entire collection of data is:

\begin{dfn}\label{dfnbndseq}
Let $\mathcal{S}=\{\mathcal{S}_k\}_{k\in \Z}$ and $\mathcal{T}=\{\mathcal{T}_k\}_{k\in \Z}$ be two families of multisets of elements of $(\R/\Gamma)\times(0,\infty]$.  The bottleneck distance between $\mathcal{S}$ and $\mathcal{T}$ is then \[ d_{B}(\mathcal{S},\mathcal{T})=\sup_{k\in \Z}d_B(\mathcal{S}_k,\mathcal{T}_k).\]
\end{dfn}

\begin{remark}\label{bndtri} It is routine to check that $d_B$ is indeed an extended pseudometric. In particular, it satisfies the triangle inequality.\end{remark}

We can now formulate another of this paper's main results, the Stability Theorem.

\begin{theorem} \label{stabthm} (Stability Theorem). Given a Floer-type complex $(C_*,\partial_C,\ell_C)$ and $k\in \Z$, denote its degree-$k$ concise barcode by $\mathcal{B}_{C,k}$; moreover let $\mathcal{B}_C=\{\mathcal{B}_{C,k}\}_{k\in\Z}$ denote the indexed family of concise barcodes for all gradings $k$. Then the bottleneck and quasiequivalence distances obey, for any two Floer-type complexes $(C_*,\partial_C,\ell_C)$ and $(D_*,\partial_D,\ell_D)$:
\begin{equation}\label{mainstability} \begin{array}{l}
 d_B(\mathcal{B}_C,\mathcal{B}_D)\leq 2d_Q((C_*,\partial_C,\ell_C),(D_*,\partial_D,\ell_D)).
\end{array} \end{equation}
Moreover, for any $k\in\Z$, if we let $\Delta_{D,k}>0$ denote the smallest second coordinate $L$ of all of the elements of $\mathcal{B}_{D,k}$, and if $d_Q((C_*,\partial_C,\ell_C),(D_*,\partial_D,\ell_D))<{\textstyle \frac{\Delta_{D,k}}{4}}$, then 
\begin{equation}\label{strongstability} \begin{array}{l}
d_B(\mathcal{B}_{C,k},\mathcal{B}_{D,k})\leq  d_Q((C_*,\partial_C,\ell_C),(D_*,\partial_D,\ell_D)).
\end{array} \end{equation}
\end{theorem}

We will also prove an inequality in the other direction, analogous to \cite[(4.11'')]{CSGO}.

\begin{theorem} \label{convstab} (Converse Stability Theorem)  With the same notation as in Theorem \ref{stabthm}, we have an inequality \[
d_Q((C_*,\partial_C,\ell_C),(D_*,\partial_D,\ell_D))\leq d_B(\mathcal{B}_C,\mathcal{B}_D). \]
\end{theorem}

Thus, with respect to the quasiequivalence and bottleneck distances, the map from Floer-type complexes to concise barcodes is globally at least bi-Lipschitz, and moreover is a local isometry (at least among complexes having a uniform positive lower bound on the parameters $\Delta_{D,k}$ as $k$ varies through $\Z$; for instance this is true for the Hamiltonian Floer complexes).  We expect that the factor of two in (\ref{mainstability}) is unnecessary so that the map is always a global isometry (as is the case when $\Gamma$ in trivial by \cite[Theorem 4.11]{CSGO}). In Section \ref{interp}, we will see this becomes true if the quasiequivalence distance $d_{Q}$ is replaced by more complicated distance called the {\bf interpolating distance}.

We prove the Stability Theorem in the following section, and the (easier) Converse Stability Theorem in Section \ref{convsect}.

\section{Proof of the Stability Theorem} \label{stabproof}

\subsection{Varying the filtration}
The proof of the stability theorem will involve first estimating the bottleneck distance between two Floer-type complexes having the same underlying chain complex but different filtration functions, and then using a mapping cylinder construction to reduce the general case to this special case.  
We begin  with  a simple combinatorial lemma:

\begin{lemma}\label{permute}  Suppose that $A$ and $B$ are finite sets and that $\sigma,\tau\co A\to B$ are bijections and $f\co A\to \R$ and $g\co B\to \R$ are functions such that, for some $\delta\geq 0$, we have $f(a)-g(\sigma(a))\leq \delta$ and $g(\tau(a))-f(a)\leq \delta$ for all $a\in A$.  Then there is a bijection $\eta\co A\to B$ such that $|f(a)-g(\eta(a))|\leq \delta$ for all $a\in A$.
\end{lemma}
\begin{proof}
Denote the elements of $A$ as $a_1,\ldots,a_n$, ordered in such a way that $f(a_1)\leq\cdots\leq f(a_n)$; likewise denote the elements of $B$ as $b_1,\ldots,b_n$, ordered such that $g(b_1)\leq\cdots\leq g(b_n)$.  Our bijection $\eta\co A\to B$ will then be given by $\eta(a_i)=b_i$ for $i=1,\ldots,n$.

Given $i\in \{1,\ldots,n\}$, write $\tau(a_i)=b_m$ and suppose first that $m\geq i$.  Then $g(b_m)\geq g(b_{i})$, so $g(b_i)-f(a_i)\leq g(b_{m})-f(a_i)\leq\delta$ by the hypothesis on  $\tau$.  On the other hand if $m<i$ then there must be some $j\in\{1,\ldots,i-1\}$ such that $\tau(a_j)=b_k$ where $k\geq i$ (for otherwise $\tau$ would give a bijection between $\{a_1,\ldots,a_{i}\}$ and a subset of $\{b_1,\ldots,b_{i-1}\}$).  In this case since $j<i\leq k$ we have \[ g(b_i)-f(a_i)
\leq g(b_k)-f(a_j)=g(\tau(a_j))-f(a_j)\leq \delta. \]  So in any event $g(b_i)-f(a_i)\leq \delta$ for all $i$.  A symmetric argument (using $\sigma^{-1}$ in place of $\tau$) shows that likewise 
$f(a_i)-g(b_i)\leq \delta$ for all $i$.  So indeed our permutation $\eta$ defined by $\eta(a_i)=b_i$ obeys $|f(a)-g(\eta(a))|\leq \delta$ for all $a\in A$.
\end{proof}

\begin{lemma}\label{shiftsvd}
Let $(C,\ell_C)$ and $(D,\ell_D)$ be orthogonalizable $\Lambda$-spaces and let $A\co C\to D$ be a $\Lambda$-linear map with unsorted singular value decomposition $\left((y_1,\ldots,y_n),(x_1,\ldots,x_m)\right)$.  Let $\ell'_D\co D\to \mathbb{R}\cup\{-\infty\}$ be another filtration function such that $(D,\ell'_D)$ is an orthogonalizable $\Lambda$-space, and let $\delta>0$ be such that $|\ell_D(d)-\ell'_D(d)|\leq \delta$ for all $d\in D$.  Then there is an unsorted singular value decomposition $\left((y'_1,\ldots,y'_n),(x'_1,\ldots,x'_m)\right)$ for the map $A$ with respect to $\ell_C$ and the new filtration function $\ell'_D$, such that:
\begin{itemize}
\item[(i)] $\ell_C(y'_i)=\ell_C(y_i)$ for each $i$.
\item[(ii)] $|\ell'_D(x'_i)-\ell_D(x_i)|\leq \delta$ for each $i\leq rank(A)$.
\end{itemize}
\end{lemma}

\begin{proof}
To simplify matters later, we shall assume that: \begin{equation}\label{levelcoset} \mbox{For all $i,j$, if $\ell_C(y_i)\equiv \ell_C(y_j) \mod\Gamma$\quad then $\ell_C(y_i)=\ell_C(y_j)$.}\qquad  \end{equation} There is no loss of generality in this assumption, as it may be arranged to hold by multiplying the various $y_i,x_i$ by appropriate field elements $T^{g_i}$ (and then correspondingly multiplying the elements $y'_i,x'_i$ constructed in the proof of the lemma by $T^{-g_i}$).

Let us first apply the algorithm described in Theorem \ref{algsvd} to $A$, viewed as a map between the non-Archimedean normed vector spaces $(C,\ell_C)$ and $(D,\ell'_D)$.  That algorithm takes as input orthonormal bases for both the domain and the codomain of $A$; for the domain $(C,\ell_C)$ we use the ordered basis $(y_1,\ldots,y_n)$ from the given singular value decomposition (for $A$ as a map from $(C,\ell_C)$ to $(D,\ell_D)$), while we use an arbitrary  orthogonal basis for the codomain. 

Denote the rank of $A$ by $r$.  Since $Ay_i=0$ for $i=r+1,\ldots,n$, inspection of the algorithm in the proof of Theorem \ref{algsvd} shows that, for $i=r+1,\ldots,m$, the element $y_i$ is unchanged throughout the running of the algorithm.  Thus the ordered basis $(y'_1,\ldots,y'_n)$ for $C$ that is output by the algorithm has $y'_i=y_i$ for $i=r+1,\ldots,m$.  So since $r$ is the rank of $A$ and $Ay'_i=Ay_i=0$ for $i>r$, it follows that $Ay'_i\neq 0$ for $i\in \{1,\ldots,r\}$. In fact, setting $x'_i=Ay'_i$ for $i\in\{1,\ldots,r\}$, the tuple $(x'_1,\ldots,x'_r)$ gives an orthogonal ordered basis for $\Img(A)$.  Moreover, according to Theorem \ref{algsvd}, we have $\ell_C(y'_i)=\ell_C(y_i)$ for all $i$, while \begin{equation} \label{prime}\ell'_D(x'_i)\leq \ell'_D(x_i)\mbox{ for } i\in \{1,\ldots,r\}.\end{equation}  Taking $(x'_{r+1},\ldots, x'_{m})$ to be an arbitrary $\ell'_D$-orthogonal basis for an orthogonal complement to $\Img(A)$, it follows that $\left((y'_1,\ldots,y'_n),(x'_1,\ldots,x'_m)\right)$ is an unsorted singular value decomposition for $A$ considered as a map from $(C,\ell_C)$ to $(D,\ell'_D)$, which moreover satisfies property (i) in the statement of the lemma.

We will show that, possibly after replacing $y'_i,x'_i$ by $y'_{\eta(i)},x'_{\eta(i)}$ for some permutation $\eta$ of $\{1,\ldots,r\}$ having $\ell_{C}(y_i)=\ell_C(y_{\eta(i)})$ for each $i$, this singular value decomposition also satisfies property (ii).  In this direction, symmetrically to the previous paragraph, apply the algorithm from Theorem \ref{algsvd} to $A$ as a map from $(C,\ell_C)$ to $(D,\ell_D)$, using as input the basis $(y'_1,\ldots,y'_n)$ for $C$ that we obtained above.  This yields a new unsorted singular value decompositon $\left((y''_1,\ldots,y''_n),(x''_1,\ldots,x''_m)\right)$ for $A$ as a map from $(C,\ell_C)$ to $(D,\ell_D)$, having \[ \ell_C(y''_i)=\ell_C(y'_i)=\ell_C(y_i)\mbox{ for all } i\] and \begin{equation}\label{2prime} \ell_D(x''_i)\leq\ell_D(x'_i)\mbox{ for }i\in\{1,\ldots,r\}.\end{equation}

Now by Theorem \ref{indep} and our assumption (\ref{levelcoset}), there is an equality of multisets of elements of $\R^2$: \begin{equation}\label{equalprime} \{(\ell_C(y_i),\ell_D(x_i))|i=1,\ldots,r\} = \{(\ell_C(y''_i),\ell_D(x''_i))|i=1,\ldots,r\}. \end{equation} Indeed, each of these multisets corresponds to the finite-length bars in the verbose barcode of the two-term Floer-type complex $(C\xrightarrow{A} D)$, and the condition (\ref{levelcoset}) and the fact that $\ell_C(y''_i)=\ell_C(y_i)$ ensure that an equality of some $\ell_C(y_i)$ and $\ell_C(y''_j)$ modulo $\Gamma$ implies an equality in $\R$.  For any $z\in \{\ell_C(y_1),\ldots,\ell_C(y_r)\}$, let \[ I_z = \left\{i\in\{1,\ldots,r\}|\ell_C(y_i)=z\right\} \]  and define functions $f,g\co I_z\to \R$ by $f(i)=\ell'_D(x'_i)$ and $g(i)=\ell_D(x_i)$.  Using (\ref{prime}), for each $i\in I_z$ we then have,  \[ f(i)\leq \ell'_D(x_i)\leq \ell_D(x_i)+\delta=g(i)+\delta.\]  Meanwhile by (\ref{equalprime}) there is a permutation $\tau$ of $I_z$ such that $\ell_D(x_{\tau(i)})=\ell_D(x''_{i})$ for all $i\in I_z$, and so by (\ref{2prime}) \[ g(\tau(i))=\ell_D(x_{\tau(i)})=\ell_D(x''_i)\leq \ell_D(x'_i)\leq \ell'_D(x'_i)+\delta=f(i)+\delta. \]

So we can apply Lemma \ref{permute} to obtain a permutation $\eta_z$ of $I_z$ such that \[ |\ell'_{D}(x_i)-\ell_D(x_{\eta_z(i)})|=|f(i)-g(\eta_z(i))|\leq \delta \] for all $i$.  Repeating this process for each $z\in\{\ell_C(y_1),\ldots,\ell_C(y_r)\}$, and reordering the tuples $(y'_1,\ldots,y'_r)$ and $(x'_i,\ldots,x'_r)$ using the permutation $\eta$ of $\{1,\ldots,r\}$ that restricts to each $I_z$ as  $\eta_z$, we obtain a singular value decomposition for $A$ as a map $(C,\ell_C)\to (D,\ell'_D)$ satisfying the desired properties.  

% Combining
%(\ref{prime}) and (\ref{2prime}) along with the hypothesis that $|\ell'_D(d)-\ell_D(d)|\leq \delta$ for all $d\in D$ shows that \[ \ell_{D}(x''_i)\leq \ell_D(x'_i)\leq \ell_{D'}(x'_i)+\delta\leq \ell_{D'}(x_i)+\delta\leq \ell_D(x_i)+2\delta \] for $i\in\{1,\ldots,r\}$. 
\end{proof}

We now prove a version of the stability theorem in the case that the Floer-type complexes in question arise from the same underlying chain complex, with different filtration functions.

\begin{prop}\label{twofilt} Let $(C_{*},\partial)$ be a chain complex of $\Lambda$-vector spaces and let $\ell_0,\ell_1\co C_*\to \R\cup\{-\infty\}$ be two filtration functions such that both $(C_*,\partial,\ell_0)$ and $(C_*,\partial,\ell_1)$ are Floer-type complexes. Assume that $\delta\geq 0$ is such that $|\ell_1(c)-\ell_0(c)|\leq \delta$ for all $c\in C_*$.   Then denoting by $\mathcal{B}_{C}^{0}$ and $\mathcal{B}_{C}^{1}$ the concise barcodes of $(C_*,\partial,\ell_0)$ and $(C_*,\partial,\ell_1)$, respectively, we have $d_{B}(\mathcal{B}_{C}^{0},\mathcal{B}_{C}^{1})\leq \delta$.
\end{prop}

\begin{proof}
Fix a grading $k$, let $r$ denote the rank of $\partial|_{C_{k+1}}$, and let $((y_1,\ldots,y_n),(x_1,\ldots,x_m))$ be a singular value decomposition for $\partial|_{C_{k+1}}$, considered as a map $(C_{k+1},\ell_0)\to (C_k,\ell_0)$.  In particular, the finite-length bars of the degree-$k$ part of $\mathcal{B}_{C}^{0}$ are given by $([\ell_0(x_i)],\ell_0(y_i)-\ell_0(x_i))$ for $1\leq i\leq r$, and the infinite-length bars of the degree-$(k+1)$ part of $\mathcal{B}_{C}^{0}$ are given by $([\ell_0(y_i)],\infty)$ for $r+1\leq i\leq n$.

We may then apply Lemma \ref{shiftsvd} to obtain an unsorted singular value decomposition $((y'_1,\ldots,y'_n),(x'_1,\ldots,x'_m))$ for $\partial|_{C_{k+1}}$, considered as a map $(C_{k+1},\ell_0)\to (C_{k},\ell_1)$, such that $\ell_0(y'_i)=\ell_0(y_i)$ for all $i$ and $|\ell_1(x'_i)-\ell_0(x_i)|\leq \delta$.  

Now consider the adjoint $\partial^*\co (C_{k})^{*}\to (C_{k+1})^{*}$ and the dual filtration functions $\ell_{0}^{*},\ell_{1}^{*}$ as defined in Section \ref{dualss}.  It follows immediately from the definitions of $\ell_{0}^{*},\ell_{1}^{*}$ and the assumption that $|\ell_1(c)-\ell_0(c)|\leq \delta$ for all $c\in C_*$ that, likewise, $|\ell_{1}^{*}-\ell_{0}^{*}|$ is uniformly bounded above by $\delta$.  Moreover by Proposition \ref{dualsvd}, the collection of dual basis elements $((x'^{*}_{1},\ldots,x'^{*}_{m}), (y'^{*}_{1},\ldots,y'^{*}_{n}))$ gives an unsorted singular value decomposition for $\partial^*$ considered as a map from $((C_{k})^{*},\ell_{1}^{*})$ to $((C_{k+1})^{*},\ell_{0}^{*})$.  Thus Lemma \ref{shiftsvd} applies to give an unsorted singular value decomposition $((\xi_1,\ldots,\xi_m),(\eta_1,\ldots,\eta_n))$ for $\partial^*$ considered as a map $((C_{k})^{*},\ell_{1}^{*})\to ((C_{k+1})^{*},\ell_{1}^{*})$, with $\ell_{1}^{*}(\xi_i)=\ell_{1}^{*}(x'^{*}_{i})$ for all $i$ and $|\ell_{1}^{*}(\eta_i)-\ell_{0}^{*}(y'^{*}_{i})|\leq \delta$ for all $i\in \{1,\ldots,r\}$.  Again using Proposition \ref{dualsvd} (and using the canonical identification of $(C_i)^{**}$ with $C_i$ for $i=k,k+1$),  it follows that $((\eta_{1}^{*},\ldots,\eta_{n}^{*}),(\xi_{1}^{*},\ldots,\xi_{m}^{*}))$ is a singular value decomposition for $\partial$ considered as a map $(C_{k+1},\ell_{1}^{**})\to (C_k,\ell_{1}^{**})$.  It is easy to see (for instance by using (\ref{flip}) twice) that $\ell_{1}^{**}=\ell_1$.  Thus the finite-length bars in the degree-$k$ part of $\mathcal{B}_{C}^{1}$ are given by $([\ell_1(\xi_{i}^{*})],\ell_1(\eta_{i}^{*})-\ell_1(\xi_{i}^{*}))$.

Now  using (\ref{flip}) we have \[ |\ell_1(\xi_{i}^{*})-\ell_0(x_i)|\leq |-\ell_{1}^{*}(\xi_i)-\ell_{1}(x'_{i})| + |\ell_{1}(x'_i)-\ell_0(x_i)|\leq |-\ell_{1}^{*}(\xi_i)+\ell_{1}^{*}(x'^{*}_{i})|+\delta=\delta\] and similarly \[ |\ell_1(\eta_{i}^{*})-\ell_0(y_i)|=|-\ell_{1}^{*}(\eta_{i})-\ell_0(y'_i)|=|-\ell_{1}^{*}(\eta_{i})+\ell_{0}^{*}(y'^{*}_{i})|\leq \delta. \]

Thus we obtain a $\delta$-matching between the finite-length bars in the degree-$k$ parts of  $\mathcal{B}_{C}^{0}$ and $\mathcal{B}_{C}^{1}$ by pairing each $([\ell_0(x_i)],\ell_0(y_i)-\ell_0(x_i))$ with $([\ell_1(\xi_{i}^{*})],\ell_{1}(\eta_{i}^{*})-\ell_{1}(\xi_{i}^{*}))$ for $i=1,\ldots,r$.

It now remains to similarly match the infinite-length bars in the degree-$k$ parts of the $\mathcal{B}_{C}^{i}$.  Let us write \[ \ker(\partial|_{C_k})=\Img(\partial|_{C_{k+1}})\oplus V_0 = \Img(\partial|_{C_{k+1}})\oplus V_1 \] where 
$\Img(\partial|_{C_{k+1}})$ is orthogonal to $V_0$ with respect to $\ell_0$ and $\Img(\partial|_{C_{k+1}})$ is orthogonal to $V_1$ with respect to $\ell_1$.  For $i=0,1$, the infinite-length bars in the degree-$k$ parts of $\mathcal{B}_{C}^{i}$ are then given by $(c,\infty)$ as $c$ varies through the filtration spectrum of $V_i$.  

For $i=0,1$, let $\pi_i\co \ker(\partial|_{C_k})\to V_i$ denote the projections associated to the above direct sum decompositions.  Note that $\pi_1|_{V_0}\co V_0\to V_1$ is a linear isomorphism, with inverse given by $\pi_0|_{V_1}$.  So for $v_0\in V_0$ we obtain \[ \ell_1(\pi_{V_1}v)\leq \ell_1(v)\leq \ell_0(v)+\delta \] while \[ \ell_0(v)=\ell_0(\pi_{V_0}\pi_{V_1}(v))\leq  \ell_0(\pi_{V_1}v)\leq \ell_1(\pi_{V_1}v)+\delta. \]  So the linear isomorphism $\pi_{V_1}|_{V_0}\co V_0\to V_1$ obeys $|\ell_1(\pi_{V_1}v)-\ell_0(v)|\leq \delta$ for all $v\in V$. A singular value decomposition for the map $\pi_{V_1}|_{V_0}\co (V_0,\ell_0|_{V_0})\to (V_1,\ell_1|_{V_1})$ precisely gives orthogonal ordered bases $(w_1,\ldots,w_{m-r})$ and $(\pi_{V_1}w_1,\ldots\pi_{V_1}w_{m-r})$ for $(V_0,\ell_0|_{V_0})$ and $(V_1,\ell_1|_{V_1})$, respectively, and the matching which sends $([\ell_0(w_i)],\infty)$ to $([\ell_1(\pi_{V_1}w_i)],\infty)$ then has defect at most $\delta$.  Combining this matching of the infinite-length bars in the degree-$k$ parts of $\mathcal{B}_{C}^{0}$ and $\mathcal{B}_{C}^{1}$ with the matching of the finite-length bars that we constructed earlier, and letting $k$ vary through $\Z$, we conclude that indeed $d_B(\mathcal{B}_{C}^{0},\mathcal{B}_{C}^{1})\leq \delta$.
\end{proof}

\subsection{Splittings}\label{splitsect}

Our proof of Theorem \ref{stabthm} will involve, given a $\delta$-quasiequivalence $(\Phi,\Psi,K_C,K_D)$,  applying Proposition \ref{twofilt} to a certain pair of filtrations on the mapping cylinder $Cyl(\Phi)_*$.  It turns out that our arguments can be made sharper if we assume that the quasiequivalence $(\Phi,\Psi,K_C,K_D)$ satisfies a certain condition;  in this subsection we introduce this condition and prove that there is no loss of generality in asking for it to be satisfied.

\begin{dfn}
Let $(C_*,\partial,\ell)$ be a Floer-type complex.  A \textbf{splitting} of $C_*$ is a graded vector space $F^{C}_{*}=\oplus_{k\in \Z}F^{C}_{k}$ such that each  $F^{C}_{k}$ is an orthogonal complement in $C_k$ to $\ker \partial_k (=\ker\partial|_{C_k})$.
\end{dfn} 

Clearly splittings always exist, as already follows from Corollary \ref{orcomple}.  One can read off a splitting from  singular value decompositions of the boundary operator in various degrees: if $((y_{1}^{k-1},\ldots,y_{n}^{k-1}),(x_{1}^{k-1},\ldots,x_{m}^{k-1}))$ is a singular value decomposition for $\partial_k\co C_k\to \ker\partial_{k-1}$ and if $r_k$ is the rank of $\partial_k$ then we may take $F_{k}^{C}=span_{\Lambda}\{y_{1}^{k-1},\ldots,y_{r_k}^{k-1}\}$.

\begin{dfn} 
If $(C_*,\partial_C,\ell_C)$ and $(D_*,\partial_D,\ell_D)$ are Floer-type complexes with splittings $F_{*}^{C}$ and $F_{*}^{D}$, respectively, a chain map $\Phi\co C_*\to D_*$ is said to be \textbf{split} provided that $\Phi(F_{*}^{C})\subset F_{*}^{D}$.
\end{dfn}

\begin{lemma} \label{repcomp} Let $\Phi \co C_* \to D_*$ be a chain map between two Floer-type complexes $(C_*, \partial_C, \ell_C)$ and $(D_*, \partial_D, \ell_D)$ having splittings $F_{*}^{C}$ and $F_{*}^{D}$, and let $\pi_C: C_* \rightarrow F_*^C$ and $\pi_D: D_* \rightarrow F_*^D$ be the projections associated to the direct sum decompositions $C_*=F_{*}^{C}\oplus \ker(\partial_C)_*$ and $D_*=F_{*}^{D}\oplus \ker(\partial_D)_*$. Define 
\[ \begin{array}{l}
\Phi^{\pi} = \pi_D \Phi \pi_C + \Phi(I_C - \pi_C).
\end{array} \]
Then this map satisfies following properties:
\begin{itemize}
\item[(i)] $\Phi^{\pi}$ is a chain map;
\item[(ii)] $\Phi^{\pi}$ is split, and $\Phi^{\pi}|_{\ker\partial_C}=\Phi|_{\ker\partial_C}$;
\item[(iii)] If $\delta\geq 0$ and $\ell_D(\Phi(x)) \leq \ell_C(x) + \delta$   for all $x \in C_*$, then likewise $\ell_D(\Phi^{\pi}(x)) \leq \ell_C(x) + \delta$ for all $x \in C_*$. \end{itemize} \end{lemma}

\begin{proof} For (i), since $\partial_C(I_C - \pi_C) = 0$, we see that $\partial_C \pi_C = \partial_C$ and similarly, $\partial_D \pi_D = \partial_D$. Then using that $\Phi$ is a chain map, we get 
\[ \begin{array}{l}
\partial_D \Phi^{\pi} = \partial_D \pi_D \Phi \pi_C + \partial_D \Phi(I_C - \pi_C) = \Phi \partial_C \pi_C + \Phi \partial_C (I_C - \pi_C) = \Phi \partial_C.
\end{array} \]
Moreover, ${\rm Im} \partial_C \leq \ker \partial_C$, so $\pi_C \partial_C =0$, and 
\[ \begin{array}{l}
\Phi^{\pi} \partial_C = \pi_D \Phi \pi_C \partial_C + \Phi(I_C - \pi_C)\partial_C = \Phi \partial_C.
\end{array} \]
So $\Phi^{\pi}$ is a chain map.

For (ii), for $x \in F^C_k$, $\pi_C x =x$ and so $(I_C - \pi_C)x = 0$. So $\Phi^{\pi} x = \pi_D \Phi \pi_C x= \pi_D \Phi x \in F^D_k$, proving that $\Phi^{\pi}$ is split. Meanwhile for $x \in \ker(\partial_C)_k$, we have $\pi_C x =0$ and so $\Phi^{\pi} x= \pi_D \Phi \pi_C x + \Phi ( I_C - \pi_C) x = \Phi x$.

For (iii), note first that since $\pi_D$ (being a projection) obeys $\pi_{D}^{2}=\pi_{D}$, we have \[ \pi_D\Phi^{\pi}=\pi_D\Phi\pi_C+\pi_D\Phi(I_C-\pi_C)=\pi_D\Phi \] while \[ (I_D-\pi_D)\Phi^{\pi}=(I_D-\pi_D)\Phi(I-\pi_C).\] So since $F_k^D$ and $\ker (\partial_D)_k$ are orthogonal, for all $x \in C_k$ we have 
\begin{align*}
\ell_D(\Phi^{\pi} x) &= \max\{ \ell_D(\pi_D \Phi^{\pi} x), \ell_D((I_D - \pi_D) \Phi^{\pi} x\}\\
& = \max\{\ell_D(\pi_D \Phi x), \ell_D((I_D - \pi_D) \Phi (I_C- \pi_C) x\} \leq \max \{\ell_D(\Phi x), \ell_D(\Phi(I_C - \pi_C)x) \}.
\end{align*}

But, assuming that $\ell_D(\Phi x) \leq \ell_C(x) + \delta$ for any $x \in C_k$, the orthogonality of $F_{k}^{C}$ and $\ker(\partial_C)_k$ implies that
\[ \begin{array}{l}
\ell_D(\Phi(I_C - \pi_C)x) \leq \ell_C(x - \pi_C x) + \delta \leq \ell_C(x) + \delta.
\end{array}\]
Thus $\ell_D(\Phi^{\pi}x)\leq \ell_C(x)+\delta$ for all $x\in C_k$.\end{proof}

\begin{prop}\label{splithtopy} Let $(C_*,\partial,\ell)$ be a Floer-type complex with a splitting $F_{*}^{C}$ and let $\pi\co C_*\to F_{*}^{C}$ be the projection associated to the direct sum decomposition $C_*=F_{*}^{C}\oplus\ker\partial_*$.  Suppose that $A,A'\co C_{*}\to C_{*}$ are two chain maps such that:
\begin{itemize} \item[(i)] $A=\partial K+K\partial$ for some $K\co C_*\to C_{*+1}$ such that there is $\ep\geq 0$ with the property that $\ell(Kx)\leq \ell(x)+\ep$ for all $x\in C_*$.
\item[(ii)] $A'$ is split.
\item[(iii)] $A|_{\ker\partial}= A'|_{\ker\partial}$.  
\end{itemize}
Then for $K'=\pi K(I_C-\pi)$, we have $A'=\partial K'+K'\partial$ and $\ell(x)\leq \ell(K'x)+\ep$ for all $x\in C_*$.
\end{prop}
\begin{proof} The statement that $\ell(x)\leq \ell(K'x)+\ep$ follows directly from the corresponding assumption on $K$ and the fact that $\pi$ and $I_C-\pi$ are orthogonal projections.  So we just need to check that $A'=\partial K'+K'\partial$; we will check this separately on elements of $\ker\partial_*$ and elements of $F_{*}^{C}$.

For the first of these, note that just as in the proof of the preceding lemma we have $\partial \pi=\partial$, and if $x\in\ker\partial_*$ then $(I_C-\pi)x=x$.  Hence, by assumption (iii), \[ A'x=Ax=\partial Kx+K\partial x=\partial Kx=\partial \pi Kx=\partial K'x=\partial K'x+K'\partial x,\] as desired.

On the other hand if $x\in F_{*}^{C}$ we first observe that \[ \partial A'x=A'\partial x=A\partial x=\partial Ax=\partial K\partial x \] where the second equality again follows from (iii).  
Now since $\partial \pi=\partial$ and since $I_C-\pi$ is the identity on ${\rm Im} \partial$ we have \[ \partial K\partial x=\partial\pi K(I-\pi)\partial x=\partial K'\partial x.\]  
Thus $\partial A'x=\partial K'\partial x$.  But both $A'$ and $K'$ have image in $F_{*}^{C}$, on which $\partial$ is injective, so $A'x=K'\partial x$.  Meanwhile (since we are assuming in this paragraph that $x\in F_{*}^{C}$) we have $(I_C-\pi)x=0$ and so $K'x=0$.  So indeed $A'x=(\partial K'+K'\partial)x$.

Since $A'$ and $\partial K'+K'\partial$ coincide on both summands $\ker\partial_C$ and $F_{*}^{C}$ of $C_*$ we have shown that they are equal.
\end{proof}

\begin{cor} \label{dqsplit}
Given two Floer-type complexes $(C_*,\partial_C,\ell_C)$ and $(D_*,\partial_D,\ell_D)$ with splittings $F_{*}^{C}$ and $F_{*}^{D}$, the quasiequivalence distance $d_Q((
(C_*,\partial_C,\ell_C), (D_*,\partial_D,\ell_D))$ is equal to 
\[ \inf\left\{\delta\geq 0\left|\begin{array}{c}\mbox{There exists a $\delta$-quasiequivalence $(\Phi,\Psi,K_C,K_D)$}\\\mbox{between   $(C_*,\partial_C,\ell_C)$ and $(D_*,\partial_D,\ell_D)$ such that}\\
\mbox{$\Phi$ and $\Psi$ are split}\end{array}\right.\right\}.
\]
\end{cor}

\begin{proof}
It suffices to show that if $(\Phi,\Psi,K_C,K_D)$ is a $\delta$-quasiequivalence then there is another $\delta$-quasiequivalence $(\Phi',\Psi',K'_C,K'_D)$ such that $\Phi'$ and $\Psi'$ are split.  For this purpose we can take $\Phi'=\Phi^{\pi}$ and $\Psi'=\Psi^{\pi}$ to be the maps
provided by Lemma \ref{repcomp}.  We can then apply 
Proposition \ref{splithtopy} with $A=\Psi\Phi-I_C$ and $A'=\Psi'\Phi'-I_C$ to obtain $K'_C\co C_*\to C_{*+1}$ with $\Psi'\Phi'-I_C=\partial_CK'_C+K'_C\partial_C$ and $\ell_C(K'_Cx)\leq \ell_C(x)+2\delta$.  Similarly applying Proposition \ref{splithtopy} with $A=\Phi\Psi-I_D$ and $A'=\Phi'\Psi'-I_D$ yields a map $K'_D\co D_*\to D_{*+1}$, and the conclusions of Lemma \ref{repcomp} and Proposition \ref{splithtopy} readily imply that $(\Phi',\Psi',K'_C,K'_D)$ is, like $(\Phi,\Psi,K_C,K_D)$, a $\delta$-quasiequivalence.
\end{proof}

Let us briefly describe the strategy of the rest of the proof of Theorem \ref{stabthm}. In the following two subsections we will introduce a filtration function $\ell_{co}$ on the mapping cone $Cone(\Phi)_{*}$ of a $\delta$-quasiequivalence  $\Phi\co C_*\to D_*$, and two filtration functions $\ell_{0},\ell_{1}$ on the mapping cylinder $Cyl(\Phi)_{*}$, with $\ell_0$ and $\ell_1$ obeying a uniform bound $|\ell_1-\ell_0|\leq \delta$.  Moreover $(Cyl(\Phi)_{*},\partial_{cyl},\ell_0)$ will be filtered homotopy equivalent to $D_*$, while $(Cyl(\Phi)_{*},\partial_{cyl},\ell_1)$ will be filtered homotopy equivalent to $C_*\oplus Cone(\Phi)_{*}$.  Combined with Proposition \ref{bdmpc} below which places bounds on the barcode of $Cone(\Phi)_{*}$ when $\Phi$ is split, these constructions will quickly yield Theorem \ref{stabthm} in Section \ref{endproof}.

\subsection{Filtered mapping cones}

  Fix throughout this section a nonnegative real number $\delta$. We will make use of the following algebraic structure, related to the mapping cylinder introduced earlier. 

\begin{dfn} \label{dfnmpc} Given two chain complexes  $(C_*,\partial_C)$ and $(D_*,\partial_D)$ and a chain map $\Phi\co C_*\to D_*$ define the {\bf mapping cone} of $\Phi$,  $(Cone(\Phi)_*, \partial_{co})$ by 
\[ \begin{array}{l}
Cone(\Phi)_*= D_* \oplus C[1]_*
\end{array} \]
with boundary operator $\partial_{co}(d,e)=(\partial_Dd-\Phi e, -\partial_Ce)$ i.e., in block form, \[ \partial_{co} = \left(\begin{array}{cc}\partial_D & -\Phi \\ 0 &- \partial_C   \end{array}\right).
\] 

Assuming additionally that $\ell_D(\Phi x)\leq \ell_C(x)+\delta$ for all $x\in C_*$, define the {\bf filtered mapping cone} $(Cone(\Phi)_*, \partial_{co}, \ell_{co})$ where the filtration function $\ell_{co}$ is given by $\ell_{co}(d,e)=\max\{\ell_D(d) + \delta, \ell_C(e)+2\delta\}$.\footnote{One could equally well define $\ell_{co}(d,e) = \max\{\ell_D(d) +t, \ell_C(e) + t + \delta\}$ for any $t \in \R$ (the  $\delta$  is included to ensure that $\ell_{co}$ does not increase under $\partial_{co}$). Although $t=0$ might seem to be the most natural choice, we use $t = \delta$ here in order to make the proofs of Propositions \ref{bdmpc} and \ref{changefil} more reader-friendly.} \end{dfn}

It is routine to check that $\partial_{co}^2 = 0$ and that $\ell_{co}(\partial_{co}(d,e))\leq \ell_{co}(d,e)$ for all $(d,e)\in Cone_*(\Phi)$. In the case that $\Phi$ is part of a $\delta$-quasiequivalence $(\Phi,\Psi,K_C,K_D)$, we will require some information about the concise barcode of $Cone(\Phi)_*$; we will be able to make an especially strong statement when $\Phi$ is split in the sense of the previous subsection.  Specifically:

\begin{prop} \label{bdmpc} Let  $(C_*, \partial_C, \ell_C)$ and $(D_*, \partial_D, \ell_D)$ be two Floer-type complexes with splittings $F_{*}^{C}$ and $F_{*}^{D}$, and let $(\Phi, \Psi, K_C, K_D)$ be a $\delta$-quasiequivalence such that $\Phi$ and $\Psi$ are split.  Then all elements $([a],L)$ of the concise  barcode of $(Cone(\Phi)_*, \partial_{co}, \ell_{co})$ have second coordinate obeying $L\leq 2 \delta$. \end{prop}

\begin{proof}  The desired conclusion is an easy consequence of the following statement:
\begin{equation}\label{2}
\forall x\in \ker(\partial_{co}),\, \exists y\in Cone(\Phi)_{*} \, \,\mbox{such that}\,\,\partial_{co}y=x\,\, \mbox{and} \,\, \ell_{co}(y)\leq \ell_{co}(x)+2\delta.
\end{equation}  Indeed, by definition, the the elements $([a],L)$ 
of the concise barcode with $L<\infty$ each correspond to pairs $y_i,x_i=\partial y_i$ from a singular value decomposition for $\partial_{co}$, with $a=\ell_{co}(x)$ and $L=\ell_{co}(y_i)-\ell_{co}(x_i)$,  and by Lemma \ref{checkingrobust} any element $y$ with $\partial y=x_i$ has $\ell(y)\geq \ell(y_i)$. Thus  (\ref{2}) implies that $L\leq 2\delta$ provided that $L<\infty$.  Meanwhile there can be no bars with $L=\infty$ since such bars arise from elements of an orthogonal
complement to $\Img(\partial_{co})$ in $\ker(\partial_{co})$ but (\ref{2}) implies that $\Img(\partial_{co})=\ker(\partial_{co})$.

%Since we only need the easier implication I've modified the statement that (2) is equivalent to the proposition to just the statement that (2) implies the proposition. (MJU 9/24)
%In fact, if (\ref{2}) is true, then in particular, the homology of the complex is trivial and so the concise barcode contains no infinite length bars. Moreover, for each corresponding pair $(\partial_{co})_*(y)= x$ where $x \in \ker(\partial_{co})_*$ from construction of singular value decomposition, 
%\[ L = \beta((\partial_{co})_*) = \ell_{co}(y) - \ell_{co}(x) \leq 2 \delta \]
%Conversely, if the conclusion of this proposition holds, again, in particular, no infinite length bar implies there is no homology, therefore, without specifying degree, any $x \in \ker(\partial_{co})_*$ has a preimage $y$. Moreover, each such $x$ can be written as $x= \sum_{i} \lambda_i x_i$ where these $x_i$'s (from a singular value decomposition) span $\ker(\partial_{co})_*$. Then we can take $y = \sum_i \lambda_i y_i$ where these $y_i$'s are from the same singular value decomposition and $(\partial_{co})_* y_i = x_i$. Therefore, if let $i_0$ be the optimal choice of index to realize $\ell(y)$, then
%\[ \ell_{co}(y) - \ell_{co}(x) \leq \ell_{co}(y_{i_0}) - \ell_{co}(x_{i_0}) \leq 2 \delta. \]

We now prove (\ref{2}). Let $x = (d,e) \in \ker (\partial_{co})_*$; thus  $\partial_{co}(d,e) = (\partial_D d- \Phi e, - \partial_C e) =0$. Therefore, 
\[ \begin{array}{l}
\partial_D d = \Phi e \,\,\,\,\,\,\,\,\,$and$\,\,\,\,\,\,\,\,\, \partial_C e =0.
\end{array} \]
Split $d$ according to the direct sum decomposition $D_*=F_{*}^{D}\oplus \ker(\partial_D)_*$ as $d = d_F + d_K$ and let $\lambda = \ell_{co} (x)$. Then $\ell_D(d) \leq \lambda - \delta$ and $\ell_C(e) \leq \lambda - 2\delta$.  So since $F^{D}_{*}$ and $\ker(\partial_D)_*$ are orthogonal, $\ell_D(d_K) \leq \lambda - \delta $ and $\ell_D(d_F) \leq \lambda- \delta $. Moreover, since $\partial_C e=0$, the equation $\Psi\Phi-I_C=\partial_C K_C+K_C\partial_C$ implies that  $\partial (K_C e)=\Psi\Phi e - e$, where $\ell_C(K_Ce)\leq \ell_C(e)+2\delta\leq \lambda$.

Write $K_Ce=a+a'$ where $a\in F_{C}^{*}$ and $a'\in \ker(\partial_C)_*$.  Then by the orthogonality of $F_{C}^{*}$ and $\ker(\partial_C)_*$ we have $\ell_C(a)\leq \ell_C(K_Ce)\leq \lambda$, and $\partial_Ca=\partial_CK_Ce= (\Psi\Phi -I_D) e$.

We then find that \begin{equation}\label{aeq} \partial_D(\Phi\Psi d_F-d_F-\Phi a)=\Phi\Psi\partial_Dd_F-\partial_Dd_F-\Phi\partial_Ca=(\Phi\Psi-I_D)\Phi e -\Phi\partial_Ca=0. \end{equation}

On the other hand, because $\Phi$ and $\Psi$ are split we have $\Phi \Psi d_F - d_F - \Phi a \in F^D_*$, so since $\partial_D|_{F^{D}_{*}}$ is injective (\ref{aeq}) implies that \[ \Phi a = \Phi\Psi d_F - d_F.\] 

Meanwhile since $\partial_D d_K = 0$, the element $b=K_Dd_K\in D_{*+1}$ obeys \[ \partial_D b=(\Phi\Psi-I_D)d_K\] and 
 $\ell_D(b) \leq \ell_D(d_K) + 2 \delta \leq \lambda - \delta + 2 \delta = \lambda + \delta$. Let $y = (-b, a - \Psi d)$. We claim that this $y$ obeys the desired conditions stated at the start of the proof. In fact, 
\begin{align*}
\partial_{co}(y) &= (\partial_D (-b)  - \Phi (a - \Psi d), -\partial_C(a - \Psi d)) \\
& = (-\partial_D b - \Phi a + \Phi\Psi d, - \partial_C a + \partial_C \Psi d)\\
& = (d_K - \Phi\Psi d_K - \Phi a + \Phi \Psi d, e - \Psi \Phi e + \Psi \partial_D d)\\
& = (d_K - \Phi \Psi d_K - \Phi \Psi d_F + d_F + \Phi \Psi d,e) \\
& = (d,e) = x.
\end{align*}
Moreover, the filtration level of $y$ obeys 
\begin{align*}
\ell_{co} (y) &= \ell_{co} ((-b, a - \Psi d)) \\
& = \max\{ \ell_D(-b) + \delta, \ell_C(a - \Psi d) +  2\delta \}\\
& \leq \max\{ \lambda + 2 \delta, \max\{\ell_C(a), \ell_C(d) + \delta\} + 2\delta\}\\
& = \lambda + 2 \delta = \ell_{co}(x) + 2 \delta.
\end{align*}
So $\partial_{co}y=x$ and $\ell_{co}(y)\leq \ell_{co}(x)+2\delta$, as desired.  Since $x$ was an arbitrary element of $\ker(\partial_{co})_*$ this implies the result. \end{proof}

\begin{remark}
If one drops the hypothesis that $\Phi$ and $\Psi$ are split, then it is possible to construct examples showing that the largest second coordinate in an element of the concise barcode of $Cone(\Phi)_*$ can be as large as $4\delta$.
\end{remark}

\subsection{Filtered mapping cylinders}

Recall the definition of the mapping cylinder $Cyl(\Phi)_{*}$ of a chain map $\Phi\co C_*\to D_*$ from Section \ref{cylsect}, and the homotopy equivalences $(i_D,\alpha,0,K)$ between $D_*$ and $Cyl(\Phi)_*$ and $(i_C,\beta,0,L)$ between $C_*$ and $Cyl(\Phi)_*$ from Lemma \ref{cylhtopy} (the first of these exists for any chain map $\Phi$, while the second requires $\Phi$ to be part of a homotopy equivalence, as is indeed the case in our present context). The ``only if'' direction of Theorem B was proven by, in the case that $(\Phi,\Psi,K_C,K_D)$ is a filtered homotopy equivalence, exploiting the behavior of a suitable filtration function on $Cyl(\Phi)_*$ with respect to $(i_D,\alpha,0,K)$ and $(i_C,\beta,0,L)$.  In the case that $(\Phi,\Psi,K_C,K_D)$ is instead a $\delta$-quasiequivalence, we will follow a similar strategy, but using different filtration functions on $Cyl(\Phi)_*$ for the two homotopy equivalences.

\begin{prop} \label{decompcyl} Given two Floer-type complexes $(C_*, \partial_C, \ell_C)$ and $(D_*, \partial_D, \ell_D)$ and a $\delta$-quasiequivalence $(\Phi, \Psi, K_C, K_D)$ between them, define a filtration function $\ell_0\co Cyl(\Phi)_*\to \mathbb{R}\cup\{-\infty\}$ by \[ \ell_0(c,d,e) = \max\{\ell_C(c)+\delta,\ell_D(d),\ell_C(e)+\delta\}.\]  Then:
\begin{itemize}
\item[(i)]$ \ell_0(\partial_{cyl}x)\leq \ell_0(x)$ for all $x\in Cyl(\Phi)_{*}$. Thus $(Cyl(\Phi)_*,\partial_{cyl},\ell_0)$ is a Floer-type complex.

\item[(ii)] Let $(i_D,\alpha,0,K)$ be as defined in Lemma \ref{cylhtopy}. Then   $(i_D,\alpha,0,K)$ is a filtered  homotopy equivalence between $(D_{*},\partial_D,\ell_D)$ and $(Cyl(\Phi)_{*},\partial_{cyl},\ell_0)$.

\end{itemize} \end{prop}

\begin{proof}
For (i), if $(c,d,e)\in Cyl(\Phi)_{*}$ we have \[ \ell_0(\partial_{cyl}(c,d,e))=\max\left\{\ell_C(\partial_Cc-e)+\delta,\ell_D(\partial_D d+\Phi e), \ell_C(\partial_C e)+\delta \right\} \] while $\ell_0(c,d,e) = \max\left\{\ell_C(c)+\delta,\ell_D(d),\ell_C(e)+\delta \right\}$.  So (i) follows from the facts that:
\begin{itemize}
\item $\ell_C(\partial_Cc-e)+\delta \leq \max\{\ell_C(c)+\delta,\ell_C(e)+\delta\}$;
\item $\ell_D(\partial_Dd+\Phi e)\leq \max\{\ell_D(d),\ell_D(\Phi e)\}\leq \max\{\ell_D(d),\ell_C(e)+\delta\} $;
\item $\ell_C(\partial_Ce)+\delta \leq \ell_C(e)+\delta$. \end{itemize}

By Lemma \ref{cylhtopy}, $(i_D,\alpha,0,K)$ is a homotopy equivalence, so to prove (ii) we just need to check that each of the maps perserves filtration. We see that: \begin{itemize}
\item Clearly $\ell_0(i_Dd)=\ell_D(d)$ for all $d\in D_*$, by definition of $\ell_0$;
\item For $(c,d,e)\in Cyl(\Phi)_{*}$, \[ \ell_D(\alpha(c,d,e))=\ell_D(\Phi c+d)\leq\max\{\ell_C(c)+\delta,\ell_D(d)\}\leq \ell_0(c,d,e);\]
\item For $(c,d,e)\in Cyl(\Phi)_{*}$, $\ell_0(K(c,d,e))=\ell_0(0,0,c)=\ell_C(c)+\delta \leq \ell_0(c,d,e)$. \end{itemize}
Thus $(i_D, \alpha,0, K)$ is indeed a filtered  homotopy equivalence.
\end{proof}

\begin{prop} \label{changefil} Given two Floer-type complexes $(C_*, \partial_C, \ell_C)$ and $(D_*, \partial_D, \ell_D)$ having splittings $F_{*}^{C}$ and $F_{*}^{D}$ and a $\delta$-quasiequivalence $(\Phi, \Psi, K_C, K_D)$ where $\Phi$ and $\Psi$ are split, define a new filtration function $\ell_1$ on $Cyl(\Phi)_*$ by \[ \ell_1(c,d,e) = \max\{\ell_C(c),\ell_D(d)+\delta,\ell_C(e)+2\delta\}. \] Then, with notation as in Proposition \ref{decompcyl}:
\begin{itemize}
\item[(i)] $\ell_1(\partial_{cyl}(c,d,e))\leq \ell_1(c,d,e)$ for all $(c,d,e)\in Cyl(\Phi)_{*}$, so $(Cyl(\Phi)_*,\partial_{cyl},\ell_1)$ is a Floer-type complex.
\item[(ii)] $i_C(C_*)$ and $\ker\beta$ are {\it orthogonal complements} with respect to $\ell_1$.
\item[(iii)] The second coordinates of all elements of the concise barcode of $(\ker\beta,\partial_{cyl},\ell_1)$ are at most $2\delta$.
\end{itemize} \end{prop}

\begin{proof} Part (i) follows just as in the proof of Proposition \ref{decompcyl} (i) (which only depended on the fact that the shift $\ell_0(0,0,e)-\ell_C(e)$ in the filtration level of $\ell_C(e)$ in the definition of $\ell_0$ was greater than or equal to both $\ell_0(c,0,0)-\ell_C(c)$ and $\delta+\ell_0(0,d,0)-\ell_D(d)$; this condition also holds with $\ell_1$ in place of $\ell_0$).

For part (ii), first note that $\ker\beta$ consists precisely of elements of the form $(-\Psi d-K_C e,d,e)$ for $(d,e)\in D_*\oplus C[1]_{*}$. We will apply Lemma \ref{or&pr} with $V=i_C(C_*)$, $U=\{0\}\oplus D_*\oplus C[1]_{*}$, and $U'=\ker\beta$. Clearly $U$ and $V$ are orthogonal with respect to $\ell_1$, and the projection $\pi_U\co Cyl(\Phi)_*\to U$ is given by $(c,d,e)\mapsto (0,d,e)$, so \[\ell_1(-\Psi d-K_Ce,d,e) = \max\{\ell_D(d)+ \delta, \ell_C(e)+ 2\delta\} = \ell_1(0, d,e)\]
which shows that $\ell_1(\pi_U x)=\ell_1(x)$ for all $x\in \ker\beta$. Thus $\ker\beta$ is indeed an orthogonal complement to $V=i_C(C_*)$.

For part (iii), define a map $f: \ker\beta \rightarrow Cone_*(-\Phi)$ by 
\[ \begin{array}{l}
f (-\Psi d-K_C e,d,e) = (d,e).
\end{array} \]
We claim that $f$ is a filtered chain isomorphism. By definition, we have $(f \circ \partial_{cyl})(-\Psi d-K_C e,d,e) = (\partial_D d + \Phi e, - \partial_C e)$. Meanwhile, $(\partial_{co} \circ f)(-\Psi d-K_C e,d,e) = (\partial_D d + \Phi e, - \partial_C e)$. Therefore, $f$ is a chain map. As for the filtrations,
\begin{align*}
\ell_{co}(f(-\Psi d-K_Ce,d,e)) &= \ell_{co} (d,e) \\ &=\max\{\ell_D(d)+\delta,\ell_C(c)+2\delta\} = \ell_1(-\Psi d-K_Ce,d,e).
\end{align*} 
Thus $f$ defines an isomorphism between $(\ker\beta,\partial_{cyl},\ell_1)$ and $(Cone_*(-\Phi), \partial_{co}, \ell_{co})$ as Floer-type complexes. Moreover, replacing $(\Phi, \Psi, K_C, K_D)$ by $(-\Phi, -\Psi, K_C, K_D)$ does not change the homotopy equations and also it has no effect on the filtration relations. Therefore, the conclusion follows from Theorem A and Proposition \ref{bdmpc}. \end{proof}

\subsection{End of the proof of Theorem \ref{stabthm}} \label{endproof}

%If $\Gamma$ is dense the theorem already follows from Corollary \ref{densestab}, so we assume that  $\Gamma$ is discrete.

Assume that $\delta\geq 0$ and that $(\Phi,\Psi,K_C,K_D)$ is a $\delta$-quasiequivalence  which is split with respect to splittings $F_{*}^{C}$ and $F_{*}^{D}$ for the Floer-type complexes $(C_*,\partial_C,\ell_C)$ and $(D_*,\partial_D,\ell_D)$.  The preceding subsection gives filtration functions $\ell_0,\ell_1\co Cyl(\Phi)_*\to \R\cup\{-\infty\}$ which evidently satisfy the bound $|\ell_1(x)-\ell_0(x)|\leq \delta$ for all $x\in Cyl(\Phi)_{*}$.  Hence by Proposition \ref{twofilt}, we have a bound \begin{equation}\label{l0l1}
d_B(\mathcal{B}_{Cyl,\ell_0},\mathcal{B}_{Cyl,\ell_1}) \leq \delta
\end{equation} for the bottleneck distance between the concise barcodes of the Floer-type complexes $(Cyl(\Phi)_*,\partial_{cyl},\ell_0)$ and $(Cyl(\Phi)_*,\partial_{cyl},\ell_1)$.

%\begin{lemma}\label{filchangempcyl} The bottleneck distance between $(Cyl(\Phi)_*,\partial_{cyl},\ell_0)$ and $(Cyl(\Phi)_*, \partial_{cyl},\ell_1)$ is less than or equal to $\delta$, that is $d_B(\mathcal B_{Cyl, \ell_1}, \mathcal B_{Cyl, \ell_0}) \leq \delta$.\end{lemma}

%\begin{proof} Define a filtration function $\ell_t$ for each $t \in [0, 1]$ by \[ \ell_t(c,d,e)=\max\{\ell_C(c)+(1-t)\delta,\ell_D(d)+t\delta, \ell_C(e)+(1+t)\delta\} \] for $(c,d,e)\in Cyl(\Phi)_*$; evidently this coincides with our earlier definitions of $\ell_0$ and $\ell_1$.   It is easy to see that $\ell_t(\partial_{cyl}(c,d,e))\leq \ell_t(c,d,e)$ for all $t,c,d,e$, so each $(Cyl(\Phi)_{*},\partial_{cyl},\ell_t)$ is a Floer-type complex.  Using a basis for $Cyl(\Phi)_*$ obtained by combining orthogonal bases for $(C_*,\ell_C),(D_*,\ell_D)$, and $(C[1]_*,\ell_C)$, the $\{\ell_t\}$ give a sliding family of filtrations on $Cyl(\Phi)_*$, with each of the functions $f_i$ (see notation in section 9.1) having either $f'_i(t)=-\delta$ or $f'_i(t)=\delta$ for all $t\in [0,1]$. Since we are assuming that $\Gamma$ is discrete, the indicated upper bound then follows directly from Proposition \ref{fccbnd}. \end{proof}

\begin{cor} \label{leqdic} If two Floer-type complexes $(C_*, \partial_C, \ell_C),(D_*,\partial_D, \ell_D)$,  are $\delta$-quasiequivalent, then we have $d_B(\mathcal B_C, \mathcal B_D) \leq 2 \delta$. Therefore, in particular, 
\[ \begin{array}{l}
d_B(\mathcal B_C, \mathcal B_D) \leq 2 d_Q((C_*, \partial_C, \ell_C), (D_*, \partial_D, \ell_D)).
\end{array} \]\end{cor}
 
\begin{proof} By Corollary \ref{dqsplit}, the assumption implies there is a $\delta$-quasiequivalence $(\Phi,\Psi,K_C,K_D)$ which moreover is split with respect to some splittings for $(C_*,\partial_C,\ell_C)$ and $(D_*,\partial_D,\ell_D)$.

By Proposition \ref{changefil} (ii), $(Cyl(\Phi)_{*},\partial_{cyl},\ell_1)$ decomposes as an orthogonal direct sum of subcomplexes $(i_C(C_*),\partial_{cyl},\ell_1)$ and $(\ker\beta,\partial_{cyl},\ell_1)$, so in any degree a singular value decomposition for   $(Cyl(\Phi)_{*},\partial_{cyl},\ell_1)$ may be obtained by combining singular value decompositions for $(i_C(C_*),\partial_{cyl},\ell_1)$ and $(\ker\beta,\partial_{cyl},\ell_1)$. Thus the concise barcode for $(Cyl(\Phi)_{*},\partial_{cyl},\ell_1)$ is the union of the concise barcodes for these two subcomplexes. 

Now $i_C$ embeds $(C_*,\partial_C,\ell_C)$ filtered isomorphically as $(i_C(C_*),\partial_{cyl},\ell_1)$, so the concise barcode of $(Cyl(\Phi)_{*},\partial_{cyl},\ell_1)$ consists of the concise barcode of $(C_*,\partial_C,\ell_C)$ together with the concise barcode of $(\ker\beta,\partial_{cyl},\ell_1)$. By Proposition \ref{changefil} (iii), all elements $([a],L)$ in the second of these barcodes have $L\leq 2\delta$.  Thus by matching the elements of the concise barcode of $(C_*,\partial_C,\ell_C)$ with themselves and leaving the  elements of the concise barcode $(\ker\beta,\partial_{cyl},\ell_1)$ unmatched, we obtain, in each degree, a partial matching between the concise barcodes of $(Cyl(\Phi)_*,\partial_{cyl},\ell_1)$ and of $(C_*,\partial_C,\ell_C)$ with defect at most $\delta$. Thus, in obvious notation, \[ d_B(\mathcal B_{C}, \mathcal B_{Cyl,\ell_1}) \leq \delta.\] 

Finally, by Proposition \ref{decompcyl} (ii) and Theorem B, we know 
\[ 
\mathcal B_{Cyl, \ell_0} = \mathcal B_{D}. 
\]
Therefore, by the triangle inequality and (\ref{l0l1}), we get 
\[ \begin{array}{l}
d_B(\mathcal B_{C}, \mathcal B_{D}) \leq d_B(\mathcal B_{C}, \mathcal B_{Cyl,\ell_1}) + d_B(\mathcal B_{Cyl, \ell_1}, \mathcal B_{Cyl, \ell_0}) + d_B(\mathcal B_{Cyl, \ell_0}, \mathcal B_{D} ) \leq 2 \delta.
\end{array} \] \end{proof}

We have thus proven the inequality (\ref{mainstability}).

For the last assertion in Theorem \ref{stabthm}, let $\lambda = d_Q((C_*, \partial_C, \ell_C), (D_*, \partial_D, \ell_D))$, so there are arbitrarily small $\ep>0$ such that there exists a (split) $(\lambda + \epsilon)$-quasiequivalence $(\Phi,\Psi,K_C,K_D)$ between $(C_*, \partial_C, \ell_C)$ and $(D_*, \partial_D, \ell_D)$.  So by (\ref{l0l1}) with $\delta=\lambda+\ep$, there is a $\delta$-matching $\mathfrak{m}$ between the concise barcodes of $(Cyl(\Phi)_*,\partial_{cyl},\ell_0)$ and $(Cyl(\Phi)_*,\partial_{cyl},\ell_1)$.  Just as in the proof of Corollary \ref{leqdic}, the first of these concise barcodes is, in any given degree $k$, the same as that of $(D_*,\partial_D,\ell_D)$, while the second of these is the union of the concise barcode of $(C_*,\partial_C,\ell_C)$ with a multiset $\mathcal{S}$ of elements all having second coordinate at most $2(\lambda+\ep)$.  
For a grading $k$ in which $\lambda<\frac{\Delta_{D,k}}{4}$, let us take $\ep$ so small that still $\delta=\lambda+\ep<\frac{\Delta_{D,k}}{4}$.  Now by definition, the image of any element $([a],L)$ which is not unmatched under a $\delta$-matching must have second coordinate at most $L+2\delta$. Meanwhile since $\delta<\frac{\Delta_{D,k}}{4}$, the concise barcode $\mathcal{B}_{D,k}$ has \emph{no} elements with second coordinate at most $4\delta$, all of the elements of our multiset $\mathcal{S}$ (each of which have second coordinate less than or equal to $2\delta$) must be unmatched under $\mathfrak{m}$.  But since all elements of $\mathcal{S}$ are unmatched, we can discard them from the domain of $\mathfrak{m}$ and so  restrict $\mathfrak{m}$ to a matching between the barcodes $\mathcal{B}_{C,k}$ and $\mathcal{B}_{D,k}$, still having defect at most $\delta=\lambda+\ep$.  So $d_B(\mathcal{B}_{C,k},\mathcal{B}_{D,k})\leq \lambda+\ep$, and since $\ep>0$ can be taken arbitrarily small this implies that $d_B(\mathcal{B}_{C,k},\mathcal{B}_{D,k})\leq \lambda=d_Q((C_*,\partial_C,\ell_C),(D_*,\partial_D,\ell_D))$.

\begin{remark}\label{densecase}
In the case that $\Gamma$ is dense, a simpler argument based on Corollary \ref{qegenebd} suffices to prove the stability theorem, in fact with the stronger inequality  $d_B\leq d_Q$.  Indeed, if $\Gamma$ is dense then the extended pseudometric $d$ from Example \ref{Gmatch} is easily seen to simplify to $d(([a],L),([a'],L'))=\frac{1}{2}|L-L'|$.  If two Floer-type complexes $(C_*,\partial_C,\ell_C)$ and $(D_*,\partial_C,\ell_C)$ are $\delta$-quasiequivalent, then we can obtain a partial matching of defect at most $\delta$ between the concise barcodes $\mathcal{B}_C$ and $\mathcal{B}_D$ by first sorting the respective barcodes in descending order by the size of the second coordinate $L$ and then matching elements in corresponding positions on the two sorted lists.  It follows easily from Theorem \ref{genebd} and Corollary \ref{qegenebd} that, when $\Gamma$ is dense, this partial matching has defect at most $\delta$.
\end{remark}

\section{Proof of converse stability}\label{convsect}

Recall the elementary Floer-type complexes $\mathcal{E}(a,L,k)$ from Definition \ref{elemdfn}.

\begin{lemma}\label{elemqe} If $\delta\in [0,\infty)$, $|a-a'|\leq \delta$, and either $L=L'=\infty$ or $|(a+L)-(a'+L')|\leq \delta$, then $\mathcal{E}(a,L,k)$ is $\delta$-quasiequivalent to $\mathcal{E}(a',L',k)$.  Moreover if $L\leq 2\delta$ then $\mathcal{E}(a,L,k)$ is $\delta$-quasiequivalent to the zero chain complex.
\end{lemma}

\begin{proof}  In the case that $L=L'=\infty$, the chain complexes underlying $\mathcal{E}(a,L,k)$ and $\mathcal{E}(a',L',k)$ are just one-dimensional, consisting of a copy of $\Lambda$ in degree $k$, with filtrations given by $\ell(\lambda)=a-\nu(\lambda)$ and $\ell'(\lambda)=a'-\nu(\lambda)$. Let $\mathbb{I}$ denote the identity on $\Lambda$. The fact that $|a-a'|\leq \delta$ then readily implies that $(\mathbb{I},\mathbb{I},0,0)$ is a $\delta$-quasiequivalence.  

Similarly if $L$ and hence (under the hypotheses of the lemma) $L'$ are both finite, the underlying chain complexes of $\mathcal{E}(a,L,k)$ and $\mathcal{E}(a',L',k)$ are both $\Lambda$-vector spaces generated by an element $x$ in degree $k$ and an element $y$ in degree $k+1$, with filtration functions $\ell$ and $\ell'$ given by saying that $(x,y)$ is an orthogonal ordered set with $\ell(x)=a$, $\ell(y)=a+L$, $\ell'(x)=a'$, and $\ell'(y)=a'+L'$.  The hypotheses imply that $|\ell(x)-\ell'(x)|\leq \delta$ and $|\ell(y)-\ell'(y)|\leq \delta$, and if $\mathbb{I}$ now denotes the identity on the two-dimensional vector space spanned by $x$ and $y$, $(\mathbb{I},\mathbb{I},0,0)$ is again a $\delta$-quasiequivalence.  

Finally, if similarly to the proof of Theorem \ref{if-part-THB} we define a linear transformation $K$ on $span_{\Lambda}\{x,y\}$ by $Kx=-y$ and $Ky=0$, then $(0,0,K,0)$ is readily seen to be a $\delta$-quasiequivalence between $\mathcal{E}(a,L,k)$ and the zero chain complex for all $\delta\geq L/2$, proving the last sentence of the lemma.
\end{proof}

\begin{proof}[Proof of Theorem \ref{convstab}]
Let $\delta = d_B(\mathcal B_C, \mathcal B_D)$; it suffices to prove the result under the assumption that $\delta< \infty$.

For any $k \in \Z$, $d_B(\mathcal B_{C,k}, \mathcal B_{D,k}) \leq \delta$. By the definition of the bottleneck distance (and using the fact that there are only finitely many partial matchings between the finite multisets $\mathcal{B}_{C,k}$ and $\mathcal{B}_{D,k}$, so the infimum in the definition is attained), there exists a partial matching $\mathfrak{m}_k=(\mathcal{B}_{C,k,short},\mathcal{B}_{D,k,short},\sigma_k)$ between $\mathcal{B}_{C,k}$ and $\mathcal{B}_{D,k}$ having defect $\delta(\mathfrak{m}_k)\leq \delta$. 

We claim that, for all $\ep>0$, \[ \oplus_k\oplus_{([a],L)\in \mathcal{B}_{C,k}}\mathcal{E}(a,L,k) \quad\mbox{and}\quad \oplus_k\oplus_{([a'],L')\in \mathcal{B}_{D,k}}\mathcal{E}(a',L',k) \] are $(\delta+\ep)$-quasiequivalent, for some representatives $a$ and $a'$ of the various cosets $[a]$ and $[a']$ in $\R/\Gamma$.  By Proposition \ref{if-part-THB} and Remark \ref{dqtri} this will imply that  $(C_*, \partial_C, \ell_C)$ and $(D_*, \partial_D, \ell_D)$ are $(\delta+\ep)$-quasiequivalent, which suffices to prove the theorem since by the definition of the quasiequivalence distance,  it will show that $d_Q( (C_*, \partial_C, \ell_C), (D_*, \partial_D, \ell_D)) \leq \delta+\ep = d_B(\mathcal B_C, \mathcal B_D)+\ep$ for all $\ep>0$.

To prove our claim, note that by Lemma \ref{elemqe} and the fact that $\delta(\mathfrak{m}_k)\leq \delta$, each $\mathcal{E}(a,L,k)$ for $([a],L)\in \mathcal{B}_{C,k,short}\cup\mathcal{B}_{D,k,short}$ is $(\delta+\ep)$-quasiequivalent to the zero chain complex (as these $\mathcal{E}(a,L,k)$ all have $L\leq 2\delta$).  Also, for $([a],L)\in \mathcal{B}_{C,k}\setminus\mathcal{B}_{C,k,short}$, if we write $([a'],L')=\sigma_k([a],L)$ where $\sigma_k$ is the bijection from the partial matching $\mathfrak{m}_k$, then there are representatives $a$ and $a'$ of the cosets $[a]$ and $[a']$ such that $|a-a'|\leq \delta+\ep$ and $|(a+L)-(a'+L')|\leq \delta+\ep$.  So by Lemma \ref{elemqe}, the associated summands $\mathcal{E}(a,L,k)$ and $\mathcal{E}(a',L',k)$ are $(\delta+\ep)$-quasiequivalent.

Moreover, it follows straightforwardly from the definitions that a direct sum of $(\delta+\ep)$-quasiequivalences is a $(\delta+\ep)$-quasiequivalence.  So we obtain a $(\delta+\ep)$-quasiequivalence between  $\oplus_k\oplus_{([a],L)\in \mathcal{B}_{C,k}}\mathcal{E}(a,L,k)$ and $ \oplus_k\oplus_{([a'],L')\in \mathcal{B}_{D,k}}\mathcal{E}(a',L',k)$ by taking a direct sum of:\begin{itemize}\item  a $(\delta+\ep)$-quasiequivalence between $\mathcal{E}(a,L,k)$ and $\mathcal{E}(a',L',k)$ for each $([a],L)\in  \mathcal{B}_{C,k}\setminus\mathcal{B}_{C,k,short}$, where $([a'],L')=\sigma_k([a],L)$;
\item a $(\delta+\ep)$-quasiequivalence between $\oplus_k\oplus_{([a],L)\in\mathcal{B}_{C,k,short}}\mathcal{E}(a,L,k)$ and the zero chain complex;
\item a $(\delta+\ep)$-quasiequivalence between the zero chain complex and  $\oplus_k\oplus_{([a'],L')\in\mathcal{B}_{D,k,short}}\mathcal{E}(a',L',k)$.
\end{itemize}
\end{proof}

\section{The interpolating distance} \label{interp}

In this section we introduce a somewhat more complicated distance function on Floer-type complexes, the interpolating distance $d_P$, and prove the isometry result Theorem \ref{globaliso} between this distance and the bottleneck distance between barcodes. We think that it is likely that $d_P$ is always equal to the quasiequivalence distance $d_Q$, and indeed in the case that $\Gamma$ is dense this equality can be inferred from our results (specifically, Theorem \ref{globaliso}, Remark \ref{densecase}, and Theorem \ref{convstab}), while in the case that $\Gamma$ is trivial it can be inferred from Theorem \ref{globaliso} and \cite[Theorem 4.11]{CSGO}. 

The definition of the distance $d_P$ will be based on a strengthening of the notion of quasiequivalence, asking not only for a quasiequivalence between the two complexes $C_*$ and $D_*$ but also for a one parameter family of complexes that interpolates between $C_*$ and $D_*$ in a suitably ``efficient'' way.  Our interest in $d_P$ is based on the facts that, on the one hand, we can prove Theorem \ref{globaliso} about it, and on the other hand standard arguments in Hamiltonian Floer theory (and other Floer theories) that give bounds for the quasiequivalence distance can be refined to give bounds on $d_P$, as we use in Section \ref{hamsect}.

\begin{dfn}\label{dfnitp} A {\bf $\delta$-interpolation} between two Floer-type complexes $(C_*, \partial_C, \ell_C)$ and $(D_*, \partial_D, \ell_D)$ is a family of Floer-type complexes $(C^s_*, \partial^s, \ell^s)$ indexed by a parameter $s$ that varies through $[0,1] \backslash \ S$ for some finite subset $S\subset (0,1)$,  such that: \begin{itemize}
\item $(C^0_*,\partial^0,\ell^0)=(C_*,\partial_C,\ell_C)$ and $(C^1,\partial^1,\ell^1)=(D_*,\partial_D,\ell_D)$; and
\item for all $s, t \in [0,1] \backslash S$, $(C^s_*, \partial^s, \ell^s)$ and $(C^t_*, \partial^t, \ell^t)$ are $\delta|s-t|$-quasiequivalent.
\end{itemize} The {\bf interpolating distance} $d_P$ between Floer-type complexes is then defined by 
\[  d_P((C_*,\partial_C,\ell_C),(D_*,\partial_D,\ell_D))=\inf\left\{\delta\geq 0\left|\begin{array}{cc}\mbox{There exists a $\delta$-interpolation between }\\(C_*,\partial_C,\ell_C)\mbox{ and }(D_*,\partial_D,\ell_D)\end{array}\right.\right\}.\] \end{dfn}

The following theorem gives a global isometry result between the bottleneck and interpolating distances.

\begin{theorem} \label{globaliso} For any two Floer-type complexes $(C_*,\partial_C,\ell_C)$ and $(D_*,\partial_D,\ell_D)$ we have \[ d_B(\mathcal{B}_C,\mathcal{B}_D)=d_P((C_*,\partial_C,\ell_C),(D_*,\partial_D,\ell_D)). \] \end{theorem}

\begin{proof} First, we will prove that for any degree $k \in \Z$,
\[ d_B(\mathcal{B}_{C,k},\mathcal{B}_{D,k}) \leq d_P((C_*,\partial_C,\ell_C),(D_*,\partial_D,\ell_D)), \] which will imply that
 $d_B(\mathcal{B}_{C},\mathcal{B}_{D}) \leq d_P((C_*,\partial_C,\ell_C),(D_*,\partial_D,\ell_D))$ by taking the supremum over $k$. Let $\lambda = d_P((C_*,\partial_C,\ell_C),(D_*,\partial_D,\ell_D))$, so by definition, given any $\epsilon>0$, there exists a $\delta$-interpolation between $(C_*,\partial_C,\ell_C)$ and $(D_*,\partial_D,\ell_D)$ with $\delta \leq \lambda + \epsilon$, denoted as $(C^s, \partial^s, \ell^s)$ with a finite singular set $S$. 

For any $p \in [0,1]\setminus S$ and any degree $k \in \Z$, choose $\epsilon_{p,k} >0$ such that $\Delta_{C^p_k} > 4 \delta \epsilon_{p,k}$, where the meaning of $\Delta_{C^p_k}$ is as in the last statement of Theorem \ref{stabthm}. By the definition of a $\delta$-interpolation, for any $s \in (p - \epsilon_{p,k}, p]$, $(C_*^s, \partial^s, \ell^s)$ and $(C_*^{p}, \partial^{p}, \ell^{p})$ are $(\delta  (p-s))$-quasiequivalent, which implies that
\[ 
d_Q ((C_*^{s}, \partial^{s}, \ell^{s}), (C_*^{p }, \partial^{p }, \ell^{p })) < \frac{\Delta_{C^p_k}}{4}. \]
Then by the last assertion from Theorem \ref{stabthm}, we know (again assuming $s\in (p-\ep_{p,k},p]$)
\[ 
d_B(\mathcal B_{C^{s},k}, \mathcal B_{C^{p},k})  = d_Q ((C_*^{s}, \partial^{s}, \ell^{s}), (C_*^{p }, \partial^{p }, \ell^{p }))  \leq \delta (p-s). 
\]
Symmetrically, for any $s' \in [p, p + \epsilon_{p,k})$, 
\[ 
d_B(\mathcal B_{C^{p},k}, \mathcal B_{C^{s'},k})  = d_Q ((C_*^{p }, \partial^{p }, \ell^{p }), (C_*^{s'}, \partial^{s'}, \ell^{s'}))  \leq \delta (s'-p).
\]
Therefore, by the triangle inequality, for $s, s'$ such that $p - \epsilon_{p,k} <s \leq p \leq s' < p+ \epsilon_{p,k}$, we have $d_B(\mathcal B_{C^{s},k}, \mathcal B_{C^{s'},k}) \leq \delta (s'-s)$. 

Now we claim that for any closed interval $[s,t] \subset [0,1]$ with $s,t\notin S$, the following estimate holds:
\begin{equation} \label{partial-est}
d_B(\mathcal B_{C^s,k}, \mathcal B_{C^t,k}) \leq (t-s) \delta.
\end{equation}
We will prove this by induction on the cardinality of $S\cap [s,t]$.  First, when $S\cap [s,t]$ is empty,  by considering a covering $\{(p- \epsilon_{p,k}, p + \epsilon_{p,k})\}_{p \in [s,t]}$ of $[s,t]$ where the $\ep_{p,k}$ are as above, we may take a finite subcover to obtain $s = s_0< s_1< ...< s_N =t$ such that $d_B(\mathcal B_{C^{s_{i-1}},k}, \mathcal B_{C^{s_i},k}) \leq \delta (s_i-s_{i-1})$. Therefore, by the triangle inequality again, 
\[ 
d_B(\mathcal B_{C^s,k}, \mathcal B_{C^t,k}) \leq \sum_{i=1}^N d_B(\mathcal B_{C^{s_{i-1}}_k}, \mathcal B_{C^{s_i}_k}) \leq (t-s)\delta. \]

Now inductively, we will assume that (\ref{partial-est}) holds when $|S\cap [s,t]| \leq m$. For the case that $|S\cap [s,t]|=m+1$, denote the smallest element of $S\cap [s,t]$ by $p^*$ and consider the intervals $[s, p^*-\epsilon']$ and $[p^* + \epsilon', t]$ for any sufficiently small $\epsilon'>0$. Applying the inductive hypothesis on both intervals, 
\[ \begin{array}{l}
d_B(\mathcal B_{C^s,k}, \mathcal B_{C^{p^* - \epsilon'},k}) \leq (p^*- \epsilon'-s) \delta
\end{array} \]
and 
\[ 
d_B(\mathcal B_{C^{p^* + \epsilon'},k}, \mathcal B_{C^t,k}) \leq (t- p^*- \epsilon') \delta
. \]
Meanwhile, by the first conclusion of Theorem \ref{stabthm}, 
\[
d_B(\mathcal B_{C^{p^* - \epsilon'}_k}, \mathcal B_{C^{p^*+ \epsilon'}_k}) \leq 2 d_Q(\mathcal B_{C^{p^* - \epsilon'}_k}, \mathcal B_{C^{p^*+ \epsilon'}_k})\leq 4\epsilon' \delta. \]
Together, we get 
\[
d_B(\mathcal B_{C^s,k}, \mathcal B_{C^t,k}) \leq (p^*- \epsilon' -s) \delta+(t- p^*- \epsilon') \delta+4\epsilon' \delta = (t-s) \delta+ 2\epsilon' \delta
. \]
Since $\epsilon'$ is arbitrarily small, it follows that $d_B(\mathcal B_{C^s,k}, \mathcal B_{C^t,k}) \leq (t-s) \delta$ whenever $s\leq t$ and $s,t\in [0,1]\setminus S$. So we have proven (\ref{partial-est}). 

In particular, letting $s=0$ and $t=1$, we get $d_B(\mathcal{B}_{C,k},\mathcal{B}_{D,k}) \leq \delta \leq \lambda + \epsilon$. Since $\epsilon$ is arbitrarily small, this shows that indeed $d_B(\mathcal{B}_{C,k},\mathcal{B}_{D,k}) \leq \lambda = d_P((C_*,\partial_C,\ell_C),(D_*,\partial_D,\ell_D)$.

Now we will prove the converse direction:
\[ d_P((C_*,\partial_C,\ell_C),(D_*,\partial_D,\ell_D)) \leq d_B(\mathcal{B}_{C},\mathcal{B}_{D}).\]

Let $\delta = d_B(\mathcal{B}_{C},\mathcal{B}_{D})$. It is sufficient to prove the result under the assumption that $\delta < \infty$. For any $k \in \Z$, $d_B(\mathcal{B}_{C,k},\mathcal{B}_{D,k}) \leq \delta$. By definition, there exists a partial matching $\mathfrak{m}_k=(\mathcal{B}_{C,k,short},\mathcal{B}_{D,k,short},\sigma_k)$ between $\mathcal{B}_{C,k}$ and $\mathcal{B}_{D,k}$ such that $\delta(\mathfrak{m}_k) \leq \delta$. We will prove that, for all $\ep>0$, there exists a $(\delta+\ep)$-interpolation between $(C_*,\partial_C,\ell_C)$ and $(D_*,\partial_D,\ell_D)$.

For each $([a],L)\in \mathcal{B}_{C,k,short}$, choose a representative $a$ of $[a]$; also if $([a],L)\in\mathcal{B}_{C,k}\setminus\mathcal{B}_{C,k,short}$ write $\sigma([a],L)=([a'],L')$ where the representative $a'$ is chosen so that both $|a'-a|\leq \delta+\ep$ and $|(a+L)-(a'+L')|\leq \delta+\ep$.  Now for $t\in (0,1)$ consider the Floer-type complex  $(C^{t}_{*},\partial^t,\ell^t)$ given by:  \begin{align*} &\bigoplus_{k\in \Z}\left(\left(\bigoplus_{([a'],L')\in\mathcal{B}_{D,k,short}}\mathcal{E}(a'+(1-t)L'/2,tL',k)\right)\oplus\left(\bigoplus_{([a],L)\in\mathcal{B}_{C,k,short}}\mathcal{E}(a+tL/2,(1-t)L,k)\right)
\right. \\ & \qquad\quad  \left.\oplus\left(\bigoplus_{([a],L)\in\mathcal{B}_{C,k}\setminus\mathcal{B}_{C,k,short}}\mathcal{E}((1-t)a+ta',(1-t)L+tL',k)\right)\right) \end{align*}

It is easy to see by Lemma \ref{elemqe} that, for $t_0,t_1\in (0,1)$, the $t_0$-version of each of the  summands above is  $(\delta+\ep)|t_0-t_1|$-quasiequivalent to its corresponding $t_1$-version.  So since the direct sum of  $(\delta+\ep)|t_0-t_1|$-quasiequivalences is a  $(\delta+\ep)|t_0-t_1|$-quasiequivalence this shows that $(C^{t_0}_{*},\partial^{t_0},\ell^{t_0})$ and $(C^{t_1}_{*},\partial^{t_1},\ell^{t_1})$ are  $(\delta+\ep)|t_0-t_1|$-quasiequivalent for $t_0,t_1\in (0,1)$.  Moreover $\mathcal{E}(a'+(1-t)L'/2,tL',k)$ is $t\delta$-quasiequivalent to the zero chain complex for each $([a'],L')\in \mathcal{B}_{D,k,short}$, and likewise $\mathcal{E}(a+tL/2,(1-t)L,k)$ is $(1-t)\delta$-quasiequivalent to the zero chain complex for each $([a],L)\in \mathcal{B}_{C,k,short}$.  In view of Proposition \ref{if-part-THB} it follows that $(C_*,\ell_C,\partial_C)$ is $t(\delta+\ep)$-quasiequivalent to $(C_{*}^{t},\partial^t,\ell^t)$, and that $(D_*,\ell_D,\partial_D)$ is $(1-t)(\delta+\ep)$-quasiequivalent to $(C_{*}^{t},\partial^t,\ell^t)$.  So extending the family  $(C_{*}^{t},\partial^t,\ell^t)$ to all $t\in [0,1]$ by setting  $(C_{*}^{0},\partial^0,\ell^0)=(C_*,\partial_C,\ell_C)$ and  $(C_{*}^{1},\partial^1,\ell^1)=(D_*,\partial_D,\ell_D)$, $\{(C^{t}_{*},\partial^t,\ell^t)\}_{t\in [0,1]}$ gives the desired $(\delta+\ep)$-interpolation between $(C_*,\partial_C,\ell_C)$ and $(D_*,\partial_D,\ell_D)$.
\end{proof}

\section{Applications in Hamiltonian Floer theory}\label{hamsect}

We now bring our general algebraic theory into contact with Hamiltonian Floer theory on compact symplectic manifolds, leading to a rigidity result for fixed points of Hamiltonian diffeomorphisms. 
 First we quickly review the geometric content of the Hamiltonian Floer complex; see, e.g., \cite{F89}, \cite{HS}, \cite{AD} for more background, details, and proofs.

Let $(M,\omega)$ be a compact symplectic manifold. Identifying $S^1=\R/\Z$,  a smooth function $H\co S^1\times M\to \R$ induces a family of diffeomorphisms $\{\phi_{H}^{t}\}_{t\in\R}$ obtained as the flow of the time-dependent vector field $X_{H(t,\cdot)}$ that is characterized by the property that, for all $t$, $\omega(\cdot,X_{H(t,\cdot)})=d(H(t,\cdot))$.  Let \[ \mathcal{P}(H)=\left\{\gamma\co S^1\to M\,\left|\,\gamma(t)=\phi_{H}^{t}(\gamma(0)),\,\gamma\mbox{ is contractible}\right.\right\} \] so that in particular $\mathcal{P}(H)$ is in bijection with a subset of the fixed point set of $\phi_{H}^{1}$ via the map $\gamma\mapsto \gamma(0)\in M$.  The Hamiltonian $H$ is called \textbf{nondegenerate} if for each $\gamma\in\mathcal{P}(H)$ the linearized map $(d\phi_{H}^{1})_{\gamma(0)}\co T_{\gamma(0)}M\to T_{\gamma(0)}M$ has all eigenvalues distinct from $1$.  Generic Hamiltonians $H$ satisfy this property.  We will assume in what follows that $H$ is nondegenerate, which guarantees in particular that $\mathcal{P}(H)$ is a finite set.

Viewing $S^1$ as the boundary of the disk $D^2$ in the usual way, given $\gamma\in\mathcal{P}(H)$ and a map $u\co D^2\to M$ with $u|_{S^1}=\gamma$, one has a well-defined ``action'' $\int_{0}^{1}H(t,\gamma(t))dt-\int_{D^2}u^*\omega$ and Conley--Zehnder index.  Define $\tilde{\mathcal{P}}(H)$ to be the set of equivalence classes $[\gamma,u]$ of pairs $(\gamma,u)$ where $\gamma\in\mathcal{P}(H)$, $u\co D^2\to M$ has $u|_{S^1}=\gamma$, and $(\gamma,u)$ is equivalent to $(\gamma',v)$ if and only if $\gamma=\gamma'$ and the map $u\#\bar{v}\co S^2\to M$ obtained by gluing $u$ and $v$ along $\gamma$ has both vanishing $\omega$-area and vanishing first Chern number.  Then there are well-defined maps $\mathcal{A}_H\co \tilde{\mathcal{P}}(H)\to \R$ and $\mu\co \tilde{\mathcal{P}}(H)\to \Z$ defined by setting $\mathcal{A}_{H}([\gamma,u])=\int_{0}^{1}H(t,\gamma(t))dt -\int_{D^2}u^*\omega$ and $\mu([\gamma,u])$ equal to the Conley--Zehnder index of the path of symplectic matrices given by expressing $\{(d\phi_{H}^{t})_{\gamma(0)}\}_{t\in [0,1]}$ in terms of a symplectic trivialization of $u^*TM$.

The degree-$k$ part of the Floer chain complex $CF_k(H)$  is then by definition (using the ground field $\mathcal{K}$) \[ \left\{\left.\sum_{\tiny{\begin{array}{c}[\gamma,u]\in\tilde{\mathcal{P}}(H),\\ \mu([\gamma,u])=k   \end{array}}}a_{[\gamma,u]}[\gamma,u]\,\right|\,a_{[\gamma,u]}\in\mathcal{K},(\forall C\in \R)(\#\{[\gamma,u]|a_{[\gamma,u]}\neq 0,\,\mathcal{A}_{H}([\gamma,u])>C\}<\infty) \right\}. \]

Let \begin{equation}\label{hamgamma} \Gamma=\left\{\left.\int_{S^2}w^*\omega\,\right|\,w\co S^2\to M,\, \langle c_1(TM),w_*[S^2]\rangle=0\right\}.\end{equation} Then $CF_k(H)$ is a vector space over $\Lambda=\Lambda^{\mathcal{K},\Gamma}$, with the scalar multiplication obtained from the action of $\Gamma$ on $\mathcal{P}(H)$ given by, for $g\in \Gamma$ and $[\gamma,u]\in\tilde{\mathcal{P}}(H)$, gluing a sphere of Chern number zero and area $g$ to $u$.

We make $CF_k(H)$ into a non-Archimedean normed vector space over $\Lambda$ by setting \[ \ell_H\left(\sum a_{[\gamma,u]}[\gamma,u]\right)=\max\{\mathcal{A}_H([\gamma,u])\,|\,a_{[\gamma,u]}\neq 0\}. \] Denote \begin{equation}\label{pkdef} \mathcal{P}_k(H)=\left\{\gamma\in\mathcal{P}(H)\,|\,(\exists u\co D^2\to M)(u|_{S^1}=\gamma,\,\mu([\gamma,u])=k)\right\}. \end{equation} Then it is easy to see that an orthogonal ordered basis for $CF_k(H)$ is given by $([\gamma_1,u_1],\ldots,[\gamma_{n_k},u_{n_k}])$ where $\gamma_1,\ldots,\gamma_{n_k}$ are the elements of $\mathcal{P}_k(H)$ and, for each $i$, $u_i$ is an arbitrarily chosen map $D^2\to M$ with $u_i|_{\partial D^2}=\gamma_i$ and $\mu([\gamma_i,u_i])=k$.  In particular $(CF_k(H),\ell_H)$ is an orthogonalizable $\Lambda$-space.

 The function $\mathcal{A}_H$ introduced above could just as well have been defined on the cover of the entire space of contractible loops of $M$ obtained by dropping the condition that $\gamma\in\mathcal{P}(H)$; then $\tilde{\mathcal{P}}(H)$ is the set of critical points of this extended functional. The degree-$k$ part of the Floer boundary operator $(\partial_H)_k\co CF_k(H)\to CF_{k-1}(H)$ is constructed by counting isolated formal negative gradient flowlines of this extended version of $\mathcal{A}_H$ in the usual way indicated in the introduction.  It is a deep but (at least when $(M,\omega)$ is semipositive, but see \cite{Pa} for the more general case) by now standard fact that $\partial_H$ can indeed be defined in this way, so that the resulting triple $(CF_*(H),\partial_H,\ell_H)$ obeys the axioms of a  Floer-type complex; thus in every degree $k$ we obtain a concise barcode $\mathcal{B}_{CF_*(H),k}$.  The construction of $\partial_H$ depends on some auxiliary choices, but the filtered chain isomorphism type of $(CF_*(H),\partial_H,\ell_H)$ is independent of these choices (see, e.g., \cite[Lemma 1.2]{U11}), so $\mathcal{B}_{CF_*(H),k}$ is an invariant of $H$.

\begin{prop} For two non-degenerate Hamiltonians 
$H_0,H_1\co S^1\times M\to \R$ on a compact symplectic manifold, the associated Floer chain complexes $(CF_*(H_0),\partial_{H_0},\ell_{H_0})$ and $(CF_*(H_1),\partial_{H_1},\ell_{H_1})$ obey \[ d_{P}\left((CF_*(H_0),\partial_{H_0},\ell_{H_0}),(CF_*(H_1),\partial_{H_1},\ell_{H_1})\right) \leq \int_{0}^{1}\|H_1(t,\cdot)-H_0(t,\cdot)\|_{L^{\infty}}dt. \]
\end{prop}

\begin{proof} Write $\delta=\int_{0}^{1}\|H_1(t,\cdot)-H_0(t,\cdot)\|_{L^{\infty}}dt $ and let $\ep>0$; we will show that there exists a $(\delta+\ep)$-interpolation between $(CF_*(H_0),\partial_{H_0},\ell_{H_0})$ and $(CF_*(H_1),\partial_{H_1},\ell_{H_1})$.  

Define $\hat{H}^{0}\co [0,1]\times S^1\times M\to \R$ by $\hat{H}^{0}(s,t,m)=sH_1(t,m)+(1-s)H_0(t,m)$.  A standard argument with the Sard-Smale theorem (see e.g., \cite[Propositions 6.1.2, 6.1.3]{Lee}) shows that, arbitrarily close to $\hat{H}^{0}$ in the $C^1$-norm, there is a smooth map $\hat{H}\co [0,1]\times S^1\times M\to \R$ such that \begin{itemize}
\item $\hat{H}(0,t,m)=H_0(t,m)$ and $\hat{H}(1,t,m)=H_1(t,m)$ for all $(t,m)\in S^1\times M$.
\item There are only finitely many $s\in [0,1]$ with the property that $H(s,\cdot,\cdot)\co S^1\times M\to\R$ fails to be nondegenerate.
\end{itemize}

In particular we can take $\hat{H}$ to be so $C^1$-close to $\hat{H}^0$ that $\left\|\frac{\partial \hat{H}}{\partial s}-\frac{\partial\hat{H}^0}{\partial s}\right\|_{L^{\infty}}<\ep$.  

For $s\in [0,1]$ write $\hat{H}_s(t,m)=\hat{H}(s,t,m)$.  Then for $0\leq s_0\leq s_1\leq 1$ and $(t,m)\in S^1\times M$ we have \begin{align*}
&|\hat{H}_{s_1}(t,m)-\hat{H}_{s_0}(t,m)|=\left| \int_{s_0}^{s_1}\frac{\partial\hat{H}}{\partial s}(s,t,m)ds\right|\\&\quad \leq \ep(s_1-s_0)+
\int_{s_0}^{s_1}\left|\frac{\partial\hat{H}^{0}}{\partial s}(s,t,m)ds\right|ds=
(\ep+|H_1(t,m)-H_0(t,m)|)(s_1-s_0).
\end{align*}  

Thus, for any $s_0,s_1\in [0,1]$, \begin{equation}\label{hath} \int_{0}^{1}\|\hat{H}_{s_1}(t,\cdot)-\hat{H}_{s_0}(t,\cdot)\|_{L^{\infty}}dt\leq \left(\ep+\int_{0}^{1}\|H_1(t,\cdot)-H_0(t,\cdot)\|_{L^{\infty}}dt\right)|s_1-s_0|=(\delta+\ep)|s_1-s_0|.
\end{equation}

Let $S=\{s\in [0,1]\,|\,\hat{H}_s\mbox{ is not non-degenerate}\}$, so by construction $S$ is a finite set, and for $s\in [0,1]\setminus S$ we have a Floer-type complex $(CF_*(\hat{H}_s),\partial_{\hat{H}_s},\ell_{\hat{H}_s})$.  Standard facts from filtered Hamiltonian Floer theory (summarized for instance in \cite[Proposition 5.1]{U13}, though note that the definition of quasiequivalence there is slightly different from ours)
show that, for $s_0,s_1\in [0,1]\setminus S$, the Floer-type complexes 
$(CF_*(\hat{H}_{s_0}),\partial_{\hat{H}_{s_0}},\ell_{\hat{H}_{s_0}})$ and $(CF_*(\hat{H}_{s_1}),\partial_{\hat{H}_{s_1}},\ell_{\hat{H}_{s_1}})$ are $\left(\int_{0}^{1}\|\hat{H}_{s_1}(t,\cdot)-\hat{H}_{s_0}(t,\cdot)\|_{L^{\infty}}dt\right)$-quasiequivalent, and hence $(\delta+\ep)|s_1-s_0|$-quasiequivalent by (\ref{hath}). 

Thus the family $(CF_{*}(\hat{H}_s),\partial_{\hat{H}_s},\ell_{\hat{H}_s})$ defines a $(\delta+\ep)$ interpolation between  $(CF_*(H_0),\partial_{H_0},\ell_{H_0})$ and $(CF_*(H_1),\partial_{H_1},\ell_{H_1})$.
Since this construction can be carried out for all $\ep>0$ the result immediately follows.
\end{proof}

Combining this proposition with Theorem \ref{globaliso}, we immediately get the following result:

\begin{cor} \label{hamint} If $H_0$ and $H_1$ are two non-degenerate Hamiltonians on any compact symplectic manifold $(M,\omega)$, then the bottleneck distance between the concise barcodes of $(CF_*(H_0),\partial_{H_0},\ell_{H_0})$ and $(CF_{*}(H_1),\partial_{H_1},\ell_{H_1})$ is less than or equal to ${\textstyle \int_{0}^{1}}\|H_1(t,\cdot)-H_0(t,\cdot)\|_{L^{\infty}}dt$.\end{cor}

Similar results apply to the way in which the  barcodes of Lagrangian Floer complexes $CF(L_0,\phi_{H}^{1}(L_1))$ depend on the Hamiltonian $H$, or for that matter to the dependence of Novikov complexes $CN_{*}(\tilde{f})$ on the function $\tilde{f}\co \tilde{M}\to \R$.   When $\Gamma$ is nontrivial these facts do not follow from previously-known results.  (When $\Gamma$ is trivial they can be inferred from \cite{CCGGO} and standard Floer-theoretic results like \cite[Proposition 5.1]{U13}.)

We now give an application of Corollary \ref{hamint} to fixed points of Hamiltonian diffeomorphisms.  Apart from its intrinsic interest, we also intend this as an illustration of how to use the methods developed in this paper.

It will be relevant  that the Floer-type complex $(CF_{*}(H),\partial_{H},\ell_H)$ of a nondegenerate Hamiltonian on a compact symplectic manifold obeys the additional property that $\ell_H(\partial_H c)< \ell_H(c)$ for all $c\in CF_*(H)$, rather than the weaker inequality ``$\leq$'' which is generally required in the definition of a Floer-type complex (this standard fact follows because the boundary operator $\partial_H$ counts nonconstant formal negative gradient flowlines of $\mathcal{A}_H$, and the function $\mathcal{A}_H$ strictly decreases along such flowlines).  Consequently there can be no elements of the form $([a],0)$ in the verbose barcode of $(CF_*(H),\partial_H,\ell_H)$ in any degree $k$, as such an element would correpond to elements $x\in CF_k(H)$ and $y\in CF_{k+1}(H)$ with $\partial_Hy=x$ and $\ell_H(y)=\ell_H(x)$.  In other words, for each degree $k$, the verbose barcode $\mathcal{\tilde{B}}_{CF_*(H),k}$ of $(CF_*(H),\partial_H,\ell_H)$ is equal to its concise barcode $\mathcal{B}_{CF_*(H),k}$.

To state the promised result, recall the notation  $\mathcal{P}_k(H)$ from (\ref{pkdef}), and for any subset $E\subset \R$, define \[ \mathcal{P}_{k}^{E}(H_0)=\{\gamma\in\mathcal{P}_k(H)\,|\, (\exists u\co D^2\to M)\left(u|_{S^1}=\gamma,\,\mathcal{A}_{H_0}([\gamma,u])\in E,\,\mu([\gamma,u])=k \right)\}. \]  

\begin{theorem}\label{pert} Let $H_0\co S^1\times M\to \R$ be a nondegenerate Hamiltonian on a compact symplectic manifold $(M,\omega)$, let $k\in \Z$, let $E\subset \R$ be any subset, and let $\Delta^E>0$ be the 
minimum of:
\begin{itemize}\item The smallest second coordinate $L$ of any element $([a],L)$ of the degree-$k$ part $\mathcal{B}_{CF_*(H_0),k}$ of the concise barcode such that some representative $a$ of the coset $[a]$ belongs to $E$; 
\item The smallest second coordinate of any $([a],L)\in \mathcal{B}_{CF_*(H_0),k-1}$ such that some $a\in [a]$ has $a+L\in E$.
\end{itemize}
 Let $H\co S^1\times M\to \R$ be any nondegenerate Hamiltonian with $\int_{0}^{1}\|H(t,\cdot)-H_0(t,\cdot)\|_{L^{\infty}}dt<\frac{\Delta^E}{2}$. Then there is an injection $f\co \mathcal{P}_{k}^{E}(H_0)\to \mathcal{P}_{k}(H)$ and, for each $\gamma\in \mathcal{P}_{k}(H_0)$, maps $u,\tilde{u}\co D^2\to M$ with $u|_{S^1}=\gamma$ and $\tilde{u}|_{S^1}=f(\gamma)$ such that \[ \left|\mathcal{A}_H([f(\gamma),\tilde{u}])-\mathcal{A}_{H_0}([\gamma,u])\right|\leq \int_{0}^{1}\|H(t,\cdot)-H_0(t,\cdot)\|_{L^{\infty}}dt. \]
\end{theorem}

\begin{proof}
As in the proof of Proposition \ref{ournormalform}, we can find singular value decompositions for $(\partial_{H_0})_{k+1}\co CF_{k+1}(H_0)\to \ker(\partial_{H_0})_{k}$ and $(\partial_{H_0})_{k}\co CF_{k}(H_0)\to \ker(\partial_{H_0})_{k-1}$ having the form \[ \left((y_{1}^{k},\ldots,y_{r_k}^{k},x_{1}^{k+1},\ldots,x_{m_{k+1}}^{k+1}),(x_{1}^{k},\ldots,x_{m_k}^{k})\right) \] and \[ \left((y_{1}^{k-1},\ldots,y_{r_{k-1}}^{k-1},x_{1}^{k},\ldots,x_{m_{k}}^{k}),(x_{1}^{k-1},\ldots,x_{m_{k-1}}^{k-1})\right). \] In particular $(y_{1}^{k-1},\ldots,y_{r_{k-1}}^{k-1},x_{1}^{k},\ldots,x_{m_{k}}^{k})$ is an orthogonal ordered basis for $CF_{k}(H_0)$.  Write the elements of $\mathcal{P}_{k}(H_0)$ as $\gamma_1,\ldots,\gamma_n$, ordered in such a way that $\mathcal{P}_{k}^{E}(H_0)=\{\gamma_1,\ldots,\gamma_s\}$ for some $s\leq n$.  As discussed before the statement of the theorem, if for each $i\in\{1,\ldots,n\}$ we choose an arbitrary $u_{i}\co D^2\to M$ with $u_{i}|_{S^1}=\gamma_i$ and $\mu([\gamma_i,u_i])=k$, and moreover $\mathcal{A}_{H_0}([\gamma_i,u_i])\in E$ for $i=1,\ldots,s$, then $([\gamma_1,u_1],\ldots,[\gamma_n,u_n])$ will be an orthogonal ordered basis for $CF_k(H_0)$.  So by Proposition \ref{multi} and the definition of $\ell_{H_0}$, there is a bijection $\alpha\co \mathcal{P}_k(H_0)\to \{y_{1}^{k-1},\ldots,y_{r_{k-1}}^{k-1},x_{1}^{k},\ldots,x_{m_{k}}^{k}\}$ such that $\ell_{H_0}(\alpha(\gamma_i))\equiv \mathcal{A}_{H_0}([\gamma_i,u_i])\, (mod\, \Gamma)$.

If $\alpha(\gamma_i)=y_{j_i}^{k-1}$ for some $j_i\in\{1,\ldots,r_{k-1}\}$, then the element $([a_i],L_i):=([\ell_{H_0}(x_{j_i}^{k-1})],\ell_{H_0}(y_{j_i}^{k-1})-\ell_{H_0}(x_{j_i}^{k-1}))$ of the degree-$(k-1)$ verbose barcode of $(CF_*(H_0),\partial_{H_0},\ell_{H_0})$ corresponds to a capped orbit $[\gamma_i, u_i]$ having filtration $\mathcal{A}_H([\gamma_i,u_i])\equiv a_i+L_i \, (mod\, \Gamma)$.  Otherwise, $\alpha(\gamma_i)=x_{j_i}^{k}$ for some $j_i\in\{1,\ldots,m_k\}$, and then we have an element $([a_i],L_i)$ of the degree-$k$ verbose barcode of  $(CF_*(H_0),\partial_{H_0},\ell_{H_0})$ where $a_i=\ell_{H_0}(x_{j_i}^{k})$ and $L_i=\ell_{H_0}(y_{j_i}^{k})-\ell_{H_0}(x_{j_i}^{k})$ if $1\leq i\leq m_k$ and $L_i=\infty$ otherwise; in this case $\mathcal{A}_H([\gamma_i,u_i])\equiv a_i\, (mod\, \Gamma)$.  As noted before the theorem, the verbose barcode of $(CF_*(H_0),\partial_{H_0},\ell_{H_0})$ is the same in every degree as its concise barcode, so in particular these elements $(a_i,L_i)$ of the verbose barcodes belong to the concise barcodes $\mathcal{B}_{CF_*(H_0),k}$ or $\mathcal{B}_{CF_*(H_0),k-1}$.

Considering now our new Hamiltonian $H$, write $\delta=\int_{0}^{1}\|H(t,\cdot)-H_0(t,\cdot)\|_{L^{\infty}}$.  Our hypothesis, along with the fact that $\mathcal{A}_{H_0}([\gamma_i,u_i])\in E$ for $i=1,\ldots,s$,  then guarantees that, for $i=1,\ldots,s$, the elements $([a_i],L_i)$ of the concise barcodes $\mathcal{B}_{CF_*(H_0),k}$ or $\mathcal{B}_{CF_*(H_0),k-1}$ described in the previous paragraph all have $L_i\geq \Delta^E>2\delta$.  On the other hand Corollary \ref{hamint} implies that there is a partial matching $\mathfrak m_k$ between $\mathcal{B}_{CF_*(H_0),k}$ and $\mathcal{B}_{CF_*(H),k}$, and  likewise a partial matching $\mathfrak m_{k-1}$ between  $\mathcal{B}_{CF_*(H_0),k-1}$ and $\mathcal{B}_{CF_*(H),k-1}$, with  both $\mathfrak{m}_k$ and $\mathfrak{m}_{k-1}$ having defects at most $\delta$.  So since each $L_i>2\delta$, none of the elements $([a_i],L_i)$ for $i=1,\ldots,s$ can be unmatched under these partial matchings.  So each of them is matched to an element, say $([\tilde{a}_i], \tilde{L}_i)$, of the degree-$k$ or $k-1$ concise barcode of $(CF_*(H),\partial_{H},\ell_{H})$. We will denote the multiset of all such ``targets'' by 
\begin{equation} \label{mathcalT}
\mathcal T_{k,k-1} = \{([\tilde{a}_i], \tilde{L}_i)\,| i=1,\ldots,s \}.
\end{equation}
Since the defect of our partial matching is at most $\delta$, we can each choose $\tilde{a}_i$ within its $\Gamma$-coset so that $|\tilde{a}_i-a_i|\leq \delta$ and either $\tilde{L}_i=L_i=\infty$ or $|(\tilde{a}_i+\tilde{L}_i)-(a_i+L_i)|\leq \delta$.

We now apply the reasoning that was used at the start of the proof to $CF_{*}(H)$ in place of $CF_{*}(H_0)$.
We may consider singular value decompositions for the maps $(\partial_H)_{k+1}$ and $(\partial_H)_k$ on $CF_*(H)$ having  the form \[ \left((z_{1}^{k},\ldots,z_{r'_k}^{k},w_{1}^{k+1},\ldots,w_{m'_{k+1}}^{k+1}),(w_{1}^{k},\ldots,w_{m'_k}^{k})\right) \] and \[ \left((z_{1}^{k-1},\ldots,z_{r'_{k-1}}^{k-1},w_{1}^{k},\ldots,w_{m'_{k}}^{k}),(w_{1}^{k-1},\ldots,w_{m'_{k-1}}^{k-1})\right). \] Then if the elements of $\mathcal{P}_{k}(H)$ are written as $\{\eta_1,\ldots,\eta_{p}\}$, we may choose $v_j\co D^2\to M$ with $v_j|_{S^1}=\eta_j$ for each $j\in\{1,\ldots,p\}$ in such a way that the multiset of real numbers $\mathcal{A}_H([\eta_j,v_j])$ is equal to the multiset $\{\ell_H(z_{j}^{k-1})|1\leq j\leq r'_{k-1}\}\cup\{\ell_H(w_{j}^{k})|1\leq j\leq m'_k \}$.  

This equality of multisets gives an injection $\iota$ from the submultiset $\mathcal T_{k,k-1}\subset\mathcal{B}_{CF_*(H),k}\cup\mathcal{B}_{CF_*(H),k-1}$ described in (\ref{mathcalT}) to $\mathcal{P}_k(H)$. Specifically: 
\begin{itemize}
\item For $i\in\{1,\ldots,s\}$  such that $\alpha(\gamma_i)=y_{j_i}^{k-1}$, the  element $([\tilde{a}_i],\tilde{L}_i)$ belongs to  $\mathcal{B}_{CF_*(H), k-1}$, and $\iota([\tilde{a}_i],\tilde{L}_i]) $ will be some $\eta_{q_i}\in\mathcal{P}_k(H)$ with $\mathcal{A}_{H}([\eta_{q_i},v_{q_i}])=\tilde{a}_i+\tilde{L}_i$;
\item For $i\in\{1,\ldots,s\}$ such that $\alpha(\gamma_i)=x_{j_i}^{k}$, the element $([\tilde{a}_i],\tilde{L}_i)$ belongs to   $\mathcal B_{CF_*(H), k}$, and $\iota([\tilde{a}_i],\tilde{L}_i])$ will be some $\eta_{q_i}$ with $\mathcal{A}_{H}([\eta_{q_i},v_{q_i}])=\tilde{a}_i$.
\end{itemize}  
The map $f\co \mathcal{P}_{k}^{E}(H_0)\to \mathcal{P}_k(H)$ promised in the theorem is then the one which sends each $\gamma_i$ to $\eta_{q_i}$; the fact that this obeys the required properties follows directly from the inequalities $|\tilde{a}_i-a_i|\leq \delta$ and $|(\tilde{a}_i+\tilde{L}_i)-(a_i+L_i)|\leq \delta$ and the fact that the value of $\mathcal{A}_H([\gamma_{q_i},v_{q_i}])$ can be varied within its $\Gamma$-coset, without changing the grading $k$, by using a different choice of capping disk $v_{q_i}$.
\end{proof}

\begin{remark}\label{pertrem} Theorem \ref{pert} may be applied with $E=\R$, in which case it shows that if $\int_{0}^{1}\|H(t,\cdot)-H(t,\cdot)\|_{L^{\infty}}dt$ is less than half of the minimal second coordinate of the concise barcode of $CF_*(H_0)$ in any degree, then the time-one flow of the perturbed Hamiltonian $H$ will have at least as many fixed points\footnote{with contractible orbit under $\phi_{H}^{t}$, though one can drop this restriction by using a straightforward variant of the Floer complex built from noncontractible orbits}
as that of the original Hamiltonian $H_0$.
This may appear somewhat surprising, as a $C^0$-small perturbation of the Hamiltonian function $H$ can still rather dramatically alter the Hamiltonian vector field $X_H$, which depends on the derivative of $H$.  However this basic phenomenon is by now rather well-known in symplectic topology; see in particular \cite[Theorem 2.1]{CR}, \cite[Corollary 2.3]{U11}, though these other results do not give control over the values of $\mathcal{A}_H$ on $\tilde{\mathcal{P}}_k(H)$ as in Theorem \ref{pert}.  

For a more general choice of $E$ our result does not appear to have analogues in the literature, particularly when $\Gamma\neq \{0\}$; this generalization is of interest when $\Delta^E$, thought of as the minimal length of a barcode interval with endpoint lying in $E$, is larger than the minimal length $\Delta^{\R}$ of all barcode intervals, in which case the Theorem shows that fixed points of $\phi_{H_0}^{1}$ with action lying in $E$ enjoy a robustness that the other fixed points of $\phi_{H_0}^{1}$ may not.  For instance in the case that
$E=\{a_0\}$ is a singleton and there is just one element $[\gamma_0,u_0]$ of $\tilde{\mathcal{P}}_k$ having $\mathcal{A}_{H}([\gamma_0,u_0])=a_0$, then $\Delta^E$ is bounded below by the lowest energy of a Floer trajectory converging to $\gamma_0$ in positive or negative time, whereas $\Delta^{\R}$ is bounded below by the lowest energy of \emph{all} Floer trajectories, which might be much smaller. 

In the special case that both $\Gamma=\{0\}$ and $E=\{a_0\}$ a version of Theorem \ref{pert} can be obtained using a standard argument in terms of the ``action window'' Floer homologies $HF^{[a,b]}_{*}(H)$ of the quotient complexes $\frac{\{c\in CF_{*}(H)|\ell_H(c)\leq b\}}{\{c\in CF_{*}(H)|\ell_H(c)<a\}}$.  Indeed, for any $\delta\in\R$ such that $\int_{0}^{1}\|H(t,\cdot)-H_0(t,\cdot)\|_{L^{\infty}}dt<\delta<\frac{\Delta^E}{2}$ we will have a commutative diagram of continuation maps (induced by appropriate monotone homotopies, cf. \cite[Section 6.6]{HZ}): \[ \xymatrix{ HF_{k}^{[a_0-\delta,a_0+\delta]}(H_0+\delta)\ar[rr]^{\Phi}\ar[dr] & &  HF_{k}^{[a_0-\delta,a_0+\delta]}(H_0-\delta)
\\ & HF_{k}^{[a_0-\delta,a_0+\delta]}(H) \ar[ur] } \] and the hypothesis on the barcode can be seen to imply that the above map $\Phi$ has rank at least equal to $\#\mathcal{P}_{k}^{E}(H_0)$, whence  $HF_{k}^{[a_0-\delta,a_0+\delta]}(H)$ has dimension at least equal to $\#\mathcal{P}_{k}^{E}(H_0)$.  When $\Gamma=\{0\}$ this last statement implies  that the number of fixed points of the time-one flow of $H$ with action in the interval $[a_0-\delta,a_0+\delta]$ is at least $\#\mathcal{P}_{k}^{E}(H_0)$.  However for $\Gamma\neq \{0\}$ the implication in the previous sentence may not be valid, since the above argument only estimates the dimension of $HF_{k}^{[a_0-\delta,a_0+\delta]}(H)$ over $\mathcal{K}$, and the contribution of a single fixed point to $\dim_{\mathcal{K}}HF_{k}^{[a_0-\delta,a_0+\delta]}(H)$ might be greater than one due to recapping.  

Thus Theorem \ref{pert} provides a way of avoiding difficulties with recapping that arise in arguments with action window Floer homology when $\Gamma\neq\{0\}$.  Even when $\Gamma=\{0\}$, if $E$  consists of, say, of two or more real numbers that are a distance less than $\Delta^E/2$ away from each other, then Theorem \ref{pert} can  be seen to give sharper results than are obtained by  action window arguments such as those described in the previous paragraph.
\end{remark}

\appendix

\section{Interleaving distance}\label{intsect}

In this brief appendix, we will discuss the relation of our quasiequivalence distance $d_Q$ to the notion of \emph{interleaving}, which is often used (e.g. in \cite{CCGGO}) as a measure of proximity between persistence modules.  Because the main objects of the paper are Floer-type complexes, rather than the persistence modules given by their filtered homologies, we will use the following definition; on passing to homology this gives (at least in principle) a slightly different notion than that used in \cite{CCGGO}, as the maps on filtered homology in \cite{CCGGO} are not assumed to be induced by maps on the original chain complexes.  

\begin{dfn} For $\delta\geq 0$, a \textbf{chain level $\delta$-interleaving} of two Floer-type  complexes $(C_*,\partial_C,\ell_C)$ and $(D_*,\partial_D,\ell_D)$ is a pair $(\Phi,\Psi)$ of chain maps $\Phi\co C_*\to D_*$ and $\Psi\co D_*\to C_*$ such that:
\begin{itemize}
\item $\ell_D(\Phi c)\leq \ell_C(c)+\delta$ for all $c\in C_*$
\item $\ell_D(\Psi d)\leq \ell_D(d)+\delta$ for all $d\in D_*$
\item For all $\lambda\in \R$ the compositions $\Psi\Phi\co C_{*}^{\lambda}\to C_{*}^{\lambda+2\delta}$ and $\Phi\Psi\co D_{*}^{\lambda}\to D_{*}^{\lambda+2\delta}$ induce the same maps on homology as the respective inclusions.
\end{itemize}
\end{dfn}

 It is  easy to see that a chain level $\delta$-interleaving induces maps $\Phi_*\co H^{\lambda}(C_*)\to H^{\lambda+\delta}(D_*)$ and $\Psi_*\co 
 H^{\lambda}(D_*)\to H^{\lambda+\delta}(C_*)$ (as $\lambda$ varies through $\mathbb{R}$) which give a strong $\delta$-interleaving between the persistence modules $\{H^{\lambda}(C_*)\}$ and $\{H^{\lambda}(D_*)\}$ in the sense of \cite{CCGGO}. It is also easy to see that if $(\Phi,\Psi,K_C,K_D)$ is a $\delta$-quasiequivalence between $(C_*, \partial_C, \ell_C)$ and $(D_*, \partial_D, \ell_D)$, then $(\Phi,\Psi)$ is a chain level $\delta$-interleaving.  We will see that the converse of this latter statement is true provided that $\Phi$ and $\Psi$ are split in the sense of Section \ref{splitsect}.

\begin{lemma} \label{htpform} Let $F_{*}^{C}$ be a splitting of a Floer-type complex $(C_*, \partial_C, \ell_C)$, and suppose that $A: C_* \rightarrow C_*$ is a chain map which is split with respect to this splitting, such that there exists $\epsilon>0$ such that $\ell_C(Ac)\leq \ell_C(c)+\ep$ for all $c\in C_*$ and, for all $\lambda\in\R$, the induced map $A_*: H_*(C^{\lambda}_*) \rightarrow H_*(C^{\lambda+\epsilon}_*)$ is zero. Then there exists a map $K: C_* \rightarrow C_{*+1}$ such that $\ell_C(Kc)\leq \ell_C(c) + \epsilon$ for all $c \in C_*$ and $A = \partial_C K + K \partial_C$.\end{lemma}

\begin{proof} Let $B_* = {\rm Im} (\partial_C)_{*+1}$. Then the boundary operator $\partial_C$ restricts as an isomorphism $(\partial_C)_{*+1}\co F_{*+1}^{C} \rightarrow B_*$. Let $L_*=\oplus_k L_k$ where each $L_k$ is a complement to $B_k$ in $\ker(\partial_C)_k$,  so that $\ker (\partial_C)_* = B_* \oplus L_*$,

Let $s: C_* \rightarrow C_{*+1}$ be the linear map such that $s|_{L_* \oplus F_*} = 0$ and $s|_{B_*} = (\partial_C |_{F_{*+1}})^{-1}$. Therefore, $\partial_C s|_{B_*}$ is the identity map on $B_*$, and for any $b \in B_*$, $s(b)$ is the unique element of $F^{C}_{*}$ such that $\partial_C s(b) = b$. Moreover, because $F_{*+1}^{C}$ is orthogonal to $\ker (\partial_C)_{*+1}$
\begin{equation}\label{opt2} 
\ell_C(s(b)) = \inf\{\ell_C(c) \,| \, c \in C_{*+1}\,, \partial_C c =b \}
. \end{equation}
Now let $K = sA$; we will check that $A = \partial_C K + K \partial_C$. Indeed, 
\begin{itemize}
\item[(i)] For $x \in \ker (\partial_C)_*$, we have $(\partial_C K + K \partial_C) x = \partial_C K x = \partial_C s A x = Ax$, since $Ax\in B_*$ by the hypothesis on $A_*\co H_*(C^{\lambda}_*)\to H_*(C^{\lambda+\ep}_{*})$
\item[(ii)] For $y \in F_{*}^{C}$, since $A$ is split and so $Ay \in F_{*}^{C}$, $Ky = sAy = 0$. Therefore, $(\partial_C K + K \partial_C) y = sA\partial_C y = s \partial_C Ay = Ay$, where the last equality comes from the fact that $\partial_C s \partial_C A y = \partial_C Ay$ and that both $s\partial_CAy$ and $Ay$ belong to $F_{*}^{C}$, together with the injectivity of $\partial_C|_{F_{*}^{C}}$. \end{itemize}

Finally, by the hypothesis that each $A_*: H_*(C^{\lambda}_*) \rightarrow H_*(C^{\lambda+\epsilon}_*)$ is zero, for any $x \in \ker(\partial_C)_*$, there exists some $z \in C_{*+1}$ such that $\partial_C z=Ax$ and $\ell_C(z) \leq \ell_C(x)  + \epsilon$. Since $Kx = sAx$ also obeys $\partial_CKx=Ax$, (\ref{opt2}) implies that 
\[ 
\ell_C(Kx) \leq \ell_C(z) \leq \ell_C(x) + \epsilon. \]
More generally any $c\in C_*$ can be written $c=x+f$ where $x\in \ker (\partial_C)_*$ and $f\in F_{*}^{C}$, and by definition $Kf=0$, so \[ \ell_C(Kc)=\ell_C(Kx)\leq \ell(x)+\ep\leq \ell_C(c)+\ep \] where the final inequality follows from the orthogonality of $\ker(\partial_C)_*$ and $F_{*}^{C}$.
\end{proof}

\begin{cor} If there is a chain-level $\delta$-interleaving between the Floer-type complexes $(C_*,\partial_C,\ell_C)$ and $(D_*,\partial_D,\ell_D)$, then there exists a $\delta$-quasiequivalence between $(C_*, \partial_C, \ell_C)$ and $(D_*, \partial_D, \ell_D)$.\end{cor}

\begin{proof} By Lemma \ref{repcomp}, we can replace both $\Phi$ and $\Psi$ by $\Phi^{\pi}$ and $\Psi^{\pi}$ which are split with respect to splittings $F_{*}^{C}$ and $F_{*}^{D}$ of our two complexes; then we will have 
\[ \begin{array}{l}
(\Psi^{\pi} \circ \Phi^{\pi} - I_C)(F^C_*) \subset F^C_* \,\,\,\,\,\,$and$\,\,\,\,\,\,(\Phi^{\pi} \circ \Psi^{\pi} - I_D)(F^D_*) \subset F^D_*.
\end{array} \]
Note that, due to condition (ii) in Lemma \ref{repcomp}, $\Phi^{\pi}$ and $\Psi^{\pi}$ induce the same maps on homology as do $\Phi$ and $\Psi$, so the fact that $(\Phi,\Psi)$ is a chain level $\delta$-interleaving implies that the maps $\Psi^{\pi}_{*}\Phi^{\pi}_{*}-I_{C*}\co H^{\lambda}(C_*)\to H^{\lambda+2\delta}(C_*)$ and $\Phi^{\pi}_{*}\Psi^{\pi}_{*}-I_{D*}\co H^{\lambda}(D_*)\to H^{\lambda+2\delta}(D_*)$ are all zero. Hence applying Lemma \ref{htpform} to $\Psi^{\pi}\Phi^{\pi}-I_C$ and to $\Phi^{\pi}\Psi^{\pi}-I_D$ gives maps  $K_C$ and $K_D$  such that ($\Phi^{\pi},\Psi^{\pi},K_C,K_D)$ is a $\delta$-quasiequivalence.  \end{proof}

In other words, if we define the (chain level)  interleaving distance $d_I$ by, for any two Floer-type complexes $(C_*,\partial_C,\ell_C)$ and $(D_*,\partial_D,\ell_D)$, \[ d_I((C_*,\partial_C,\ell_C), (D_*,\partial_D,\ell_D))=\inf\left\{\delta\geq 0\left|\begin{array}{c}\mbox{There exists a chain level $\delta$-interleaving}\\ \mbox{between  $(C_*,\partial_C,\ell_C)$ and $(D_*,\partial_D,\ell_D)$} \end{array}\right.\right\},\] then we have an equality of distance functions $d_I=d_Q$ where $d_Q$ is the quasiequivalence distance.


\begin{thebibliography}{9999}

\bibitem[AD14]{AD} M. Audin and M. Damian. \emph{Morse theory and Floer homology}. Translated from the 2010 French original by Reinie Erné. Universitext. Springer, London; EDP Sciences, Les Ulis, 2014.

\bibitem[B94]{Bar} S. Barannikov. \emph{The framed Morse complex and its invariants}. Singularities and bifurcations, 93--115, Adv. Soviet Math., \textbf{21}, Amer. Math. Soc., Providence, RI, 1994.

\bibitem[BL14]{BL} U. Bauer and M. Lesnick.  \emph{Induced Matchings of Barcodes and the Algebraic Stability of Persistence}.  Proceedings of the twenty-ninth annual symposium on Computational geometry (2014), 355–-364.

\bibitem[BC09]{BC} P. Biran and O. Cornea. \emph{Rigidity and uniruling for Lagrangian submanifolds}. Geom. Topol. \textbf{13} (2009), no. 5, 2881--2989.

\bibitem[BD13]{BD} D. Burghelea and T. Dey. \emph{Topological persistence for circle-valued maps}. Discrete Comput. Geom. \textbf{50} (2013), no. 1, 69--98.

\bibitem[BH13]{BH} D. Burghelea and S. Haller. \emph{Topology of angle valued maps, bar codes and Jordan blocks}. arXiv:1303.4328.

\bibitem[Ca09]{CBull} G. Carlsson. \emph{Topology and data}. Bull. Amer. Math. Soc. \textbf{46} (2009), no. 2, 255--308.

\bibitem[CCGGO09]{CCGGO} F. Chazal, D. Cohen-Steiner, M. Glisse, L. Guibas, and S. Oudot. \emph{Proximity of persistence modules and their diagrams}. Proceedings of the 25th Annual Symposium on Computational Geometry, SCG '09, 237--246.  ACM, 2009

\bibitem[CdSGO12]{CSGO} F. Chazal, V. de Silva, M. Glisse, S. Oudot. \emph{Structure and stability of persistence modules}, arXiv:1207:3674.

\bibitem[CEH07]{CEH} D. Cohen-Steiner, H. Edelsbrunner, and J. Harer. \emph{Stability of persistence diagrams}. Discrete Comput. Geom. \textbf{37} (2007), 103--120.

\bibitem[CR03]{CR} O. Cornea and A. Ranicki. \emph{Rigidity and gluing for Morse and Novikov complexes}. J. Eur. Math. Soc. (JEMS) \textbf{5} (2003), no. 4, 343--394.

\bibitem[Cr12]{CB} W. Crawley-Boevey. \emph{Decomposition of pointwise finite-dimensional persistence modules}.  J. Algebra Appl. \textbf{14} (2015), 1550066.

\bibitem[dSMVJ]{SMV} V. de Silva, D. Morozov, and M. Vejdemo-Johansson. \emph{ Dualities in persistent (co)homology}. Inverse Problems \textbf{27} (2011), no. 12, 124003, 17 pp.

\bibitem[EP03]{EP03} M. Entov and L. Polterovich. \emph{Calabi quasimorphism and quantum homology}. Int. Math. Res. Not.  \textbf{2003},  no. 30, 1635--1676.

\bibitem[Fa04]{Fa} M. Farber. \emph{Topology of closed one-forms}. Mathematical Surveys and Monographs \textbf{108}, AMS, Providence, 2004.

\bibitem[Fl88a]{F88} A. Floer. \emph{ Morse theory for Lagrangian intersections}. J. Differential Geom. \textbf{28} (1988), no. 3, 513--547.

\bibitem[Fl88b]{F88b} A. Floer. \emph{An instanton-invariant for 3-manifolds}. Comm. Math. Phys. \textbf{118} (1988), no. 2, 215--240.

\bibitem[Fl89]{F89} A. Floer. \emph{Symplectic fixed points and holomorphic spheres}. Comm. Math. Phys. \textbf{120} (1989), no. 4, 575--611. 

\bibitem[FlH94]{FH} A. Floer and H. Hofer. \emph{Symplectic homology. I. Open sets in $\mathbb{C}^n$}. 
Math. Z. \textbf{215} (1994), no. 1, 37--88. 

\bibitem[Fr04]{Fr} U. Frauenfelder. \emph{The Arnold-Givental conjecture and moment Floer homology}. Int. Math. Res. Not. \textbf{2004}, no. 42, 2179--2269.

\bibitem[FO99]{FO} K. Fukaya and K. Ono. \emph{Arnold conjecture and Gromov--Witten invariants}.  Topology \textbf{38} (1999), 933--1048. 

\bibitem[FOOO09]{FOOO09} K.Fukaya, Y.-G. Oh, H. Ohta, and K. Ono. \emph{Lagrangian Intersection Floer Theory: Anomaly and Obstruction}. 2 vols. AMS, Providence, 2009.

\bibitem[FOOO13]{FOOO13}  K.Fukaya, Y.-G. Oh, H. Ohta, and K. Ono. \emph{Displacement of polydisks and Lagrangian Floer theory}. J. Symplectic Geom. \textbf{11} (2013), no. 2, 231--268.

\bibitem[Ghr08]{Ghr08} R. Ghrist. \emph{Barcodes: The persistent topology of data}. Bull. Amer. Math. Soc.  \textbf{45} (2008), 61--75.

\bibitem[HZ94]{HZ} H. Hofer and E. Zehnder. \emph{Symplectic invariants and Hamiltonian dynamics}. Birkh\"auser Verlag, Basel, 1994.

\bibitem[HS95]{HS} H. Hofer and D. Salamon. \emph{Floer homology and Novikov rings}. In \emph{The Floer memorial volume}, 483--524, Progr. Math., \textbf{133}, Birkh\"auser, Basel, 1995.

\bibitem[HLS15]{HLS} V. Humili\`ere, R. Leclercq, and S. Seyfaddini. \emph{Coisotropic rigidity and $C^0$-symplectic geometry}. Duke Math. J. \textbf{164} (2015), no. 4, 767-799.

\bibitem[Ke10]{Ked} K. Kedlaya. \emph{$p$-adic differential equations}. Cambridge Studies in Advanced Mathematics, \textbf{125}. Cambridge University Press, Cambridge, 2010.

\bibitem[LNV13]{LNV} D. Le Peutrec, F. Nier, and C. Viterbo. \emph{Precise Arrhenius law for p-forms: the Witten Laplacian and Morse-Barannikov complex}.
Ann. Henri Poincar\'e \textbf{14} (2013), no. 3, 567–-610. 

\bibitem[Le05]{Lee} Y.-J. Lee. \emph{Reidemeister torsion in Floer-Novikov theory and counting pseudo-holomorphic tori. I}. J. Symplectic Geom. 3 (2005), no. 2, 221--311.

\bibitem[LT98]{LT} G. Liu and G. Tian. \emph{Floer homology and Arnold
conjecture}. J. Diff. Geom. \textbf{49} (1998), no. 1, 1--74.

\bibitem[MSa04]{McSa} D. McDuff, D. Salamon. \emph{J-holomorphic Curves and Symplectic Topology}, AMS, Providence, RI, 2004.


\bibitem[M34]{Morse} M. Morse. \emph{The Calculus of Variations in the Large}. AMS. Colloq. Publ. \textbf{18}, AMS, New York, 1934.

\bibitem[MS65]{MS65} A. Monna and T. Springer. \emph{Sur la structure des espaces de Banach non-archim\'ediens}.  Nederl. Akad. Wetensch. Proc. Ser. A \textbf{68}=Indag. Math. \textbf{27} (1965), 602--614.

\bibitem[N81]{Nov} S. Novikov. \emph{Multivalued functions and functionals. An analogue of the Morse theory}.
Soviet Math. Dokl. \textbf{24} (1981), 222--226.

\bibitem[Oh05]{Oh05} Y.-G. Oh. \emph{Construction of spectral invariants of Hamiltonian paths on closed symplectic manifolds}. In  \emph{The breadth of symplectic and Poisson geometry}, 525--570, Progr. Math., \textbf{232}, Birkh\"auser Boston, Boston, MA, 2005.

\bibitem[Pa13]{Pa} J. Pardon. \emph{An algebraic approach to virtual fundamental cycles on moduli spaces of $J$-holomorphic curves}. arXiv:1309:2370, to appear in Geom. Topol.

\bibitem[PS14]{PS} L. Polterovich and E. Shelukhin. \emph{Autonomous Hamiltonian flows, Hofer's geometry and persistence modules}. Selecta Math. (2015), published online, doi:
10.1007/s00029-015-0201-2.

\bibitem[RS93]{RS} J. Robbin and D. Salamon. \emph{The Maslov index for paths}. Topology, (1993), 827-844.

\bibitem[Sal97]{Sal97} D. Salamon. \emph{Lectures on Floer homology}. Lecture Notes for the IAS/PCMI Graduate Summer School on Symplectic Geometry and Topology, 1997.

\bibitem[Sc93]{Sc93} M. Schwarz. \emph{Morse Homology}. Progr. Math. \textbf{111}, Birkh\"auser Verlag, Basel, 1993.

\bibitem[Sc00]{Sc00} M. Schwarz. \emph{On the action spectrum for closed symplectically aspherical manifolds}. Pacific J. Math. \textbf{193} (2000), 419--461.

\bibitem[U08]{Ush08} M. Usher. \emph{Spectral numbers in Floer theories}, Compositio Math. \textbf{144} (2008), 1581--1592.

\bibitem[U10]{U10} M. Usher. \emph{Duality in filtered Floer-Novikov complexes}. J. Topol. Anal. 2 (2010), no. 2, 233--258.

\bibitem[U11]{U11} M. Usher. \emph{Boundary depth in Hamiltonian Floer theory and its applications to Hamiltonian dynamics and coisotropic submanifolds}.  Israel J. Math. \textbf{184} (2011), 1--57.

\bibitem[U13]{U13} M. Usher. \emph{Hofer's metrics and boundary depth}. Ann. Sci. \'Ec. Norm. Sup\'er. (4) \textbf{46} (2013), no. 1, 57--128. 

\bibitem[ZC05]{ZC} A. Zomorodian and G. Carlsson. \emph{Computing persistent homology}. Discrete Comput. Geom. \textbf{33} (2005), 249--274.
\end{thebibliography}
\end{document}